\documentclass[10pt, oneside]{amsart}   	
\usepackage{geometry}                		
\usepackage{graphicx}				
					\usepackage{color}			
\usepackage{amssymb,mathrsfs}

\theoremstyle{plain} 
\newtheorem{theorem}{Theorem}[section]

\newtheorem{lemma}[theorem]{Lemma}
\newtheorem{proposition}[theorem]{Proposition}
\newtheorem{corollary}[theorem]{Corollary}
\newtheorem{remark}[theorem]{Remark}

\newtheorem{definition}[theorem]{Definition}

\newcommand{\flip}{\mathcal{T}}
\newcommand{\Mat}{{\rm Mat}}
\newcommand{\Vee}{\mathcal{V}}

\title{Weyl Algebras  for Quantum Homogeneous Spaces}
\author{Gail Letzter}
\address{Gail Letzter, Mathematics Research Group, National Security Agency}
\email{ gletzter@verizon.net}
\author{Siddhartha Sahi}
\address{Siddhartha Sahi, Department of Mathematics, Rutgers University}
\email{sahi@math.rutgers.edu}
\author{Hadi Salmasian}
\address{Hadi Salmasian, Department of Mathematics and Statistics, University of Ottawa}
\email{ hadi.salmasian@uottawa.ca}

\begin{document}
\maketitle

\begin{abstract}
We present a new family of quantum Weyl algebras where the polynomial part is the quantum analog of functions on homogeneous spaces corresponding to symmetric matrices, skew symmetric matrices, and the entire space of matrices of a given size.  The construction  uses twisted tensor products and their deformations combined with invariance properties derived from  quantum symmetric pairs. These quantum Weyl algebras admit $U_q(\mathfrak{gl}_N)$-module algebra structures compatible with  standard ones on the polynomial part, have relations that are expressed nicely via matrices,  and are closely related to an algebra arising in the theory of quantum bounded symmetric domains.  \end{abstract}
\section{Introduction}

Let $\Vee$ denote a complex vector space. Write $\mathscr{P}(\Vee)$ for the ring of polynomial functions on $\Vee$ and $\mathscr{D}(\Vee)$ for the algebra of constant coefficient differential operators.  Together, 
$\mathscr{P}(\Vee)$ and $\mathscr{D}(\Vee)$ generate the algebra of polynomial-coefficient differential operators on $\Vee$, denoted by $\mathscr{P}\mathscr{D}(\Vee)$. Note that 
$\mathscr{P}\mathscr{D}(\Vee)$  is just the Weyl algebra with polynomial part equal to $\mathscr{P}(\Vee)$.  Now suppose that $\Vee$ is a $\mathfrak{k}$-module where $\mathfrak{k}$ is a complex reductive Lie algebra.  The study of invariants inside $\mathscr{P}\mathscr{D}(\Vee)$ with respect to the action of $\mathfrak{k}$   yields beautiful results at the intersection of representation theory, invariant theory and special functions.  

In \cite{B}, Bershtein studies this setup in the quantum case where $\Vee=\Mat_N$, the vector space  of $N\times N$ matrices, and $\mathfrak{k}$ is the Lie subalgebra  of $\mathfrak{sl}_{2N}$ generated by two copies of $\mathfrak{sl}_N$ and the Cartan element  $h_N$ connecting the two. 
 The quantum Weyl algebra $\mathscr{P}\mathscr{D}_q(\Mat_N)$  in \cite{B} is a normalized version of the algebra ${\rm Pol}(\Mat_N)_q$ introduced and studied by Shklyarov, Sinel'schchikov, and Vaksman (\cite{SV}, \cite{VSS}) as part of the theory of  quantum bounded symmetric domains. In this context, 
 ${\rm Pol}(\Mat_N)_q$ is a $*$-algebra.  Moreover, the $U_q(\mathfrak{k})$-module action for ${\rm Pol}(\Mat_N)_q$ is inherited from its structure as a $U_q(\mathfrak{su}_{N,N})$-module.  Here,  $U_q(\mathfrak{su}_{N,N})$ is  the Hopf $*$-algebra with underlying Hopf algebra equal to $U_q(\mathfrak{sl}_{2N})$.  


 In this paper, we construct a new family of quantum Weyl algebras associated to  symmetric pairs $(\mathfrak{g},\mathfrak{k})$  of Type AI, Type AII, and the Type A diagonal case.  The polynomial part  is  a quantum version of $\mathscr{P}(\Vee)$, where $\Vee$ is an affine space which contains the standard realization of the homogeneous space $G/K$ as an orbit.  Here $G$, $K$ are the Lie groups corresponding to $\mathfrak{g}, \mathfrak{k}$.  This vector space $\Vee$  consists of  symmetric matrices in Type AI, skew symmetric matrices in Type AII, and all matrices in the Type A diagonal case.    The construction 
takes advantage of the theory of quantum symmetric pairs as developed by Noumi in terms of reflection equations (\cite{N}) and by the first author in terms of generators and relations (\cite{L1999}, \cite{L2002}).  The relations for the resulting quantum Weyl algebras resemble those of $\mathscr{P}\mathscr{D}_q(\Mat_N)$.
 
 
 The motivation for looking at these three types of symmetric pairs is their connection to Jordan algebras that leads to a rich invariant theory in the classical setting.  In a follow-up paper (\cite{LSS}), we study the quantum Capelli operators, which are  invariants inside the quantum Weyl algebras with respect to a quantum analog of $U(\mathfrak{k})$, and determine their eigenvalues. These operators satisfy vanishing and Weyl  group invariance properties as in \cite{S1996} and \cite{Knop}.  In particular, the quantum Weyl algebras presented here provide a natural setting for quantum versions of results due to the second author (see for example \cite{S1994}, \cite{KS1991} and \cite{KS1993}). 
 Similar results on Capelli operators have been obtained in the Lie superalgebra setting (see \cite{ASS}, \cite{SS2016}, \cite{SSS2020},  and \cite{SSS2021}).
  
 In order to construct these quantum Weyl algebras, we first provide a new method for creating   the algebra $\mathscr{P}\mathscr{D}_q(\Mat_N)$ as a $U_q(\mathfrak{gl}_N)$-bimodule  instead of the $U_q(\mathfrak{su}_{N,N})$-module structure of \cite{B} and \cite{VSS}.  
Using this version of $\mathscr{P}\mathscr{D}_q(\Mat_N)$, it is straightforward to produce a quantum Weyl algebra for non-square matrices $\Mat_{n\times m}$ that is a $U_q(\mathfrak{gl}_n)\otimes U_q(\mathfrak{gl}_{m})$-module algebra.  The resulting $U_q(\mathfrak{gl}_n)\otimes U_q(\mathfrak{gl}_{m})$-module algebra, 
 $\mathscr{P}\mathscr{D}_q(\Mat_{n\times m})$, is studied in \cite{LSS0} where a quantum analog of the double commutant property and the  first fundamental theorem of invariant theory are proved.
 
 Our approach for defining $\mathscr{P}\mathscr{D}_q(\Mat_N)$ has some of the same ingredients as in \cite{VSS} and \cite{VSS2}.  Indeed,  $R$-matrices, the algebra of quantized functions $\mathcal{O}_q(\Mat_N)$ on $N\times N$ matrices,  and an invariant bilinear form play a role in these references as well as in this paper. However, we take here a more intensive algebraic approach based on the theory of twisted tensor products and their PBW deformations  (see \cite{CSV}, \cite{WW2018}, and \cite{WW2014}).  One of the nice aspects of our approach is that 
 the  $U_q(\mathfrak{gl}_N)$-bimodule structure for  $\mathscr{P}\mathscr{D}_q(\Mat_N)$ is inherited from the bimodule structures of $\mathcal{O}_q(\Mat_N)$ and $\mathcal{O}_q(\Mat_N)^{op}$ in a natural way via the twisted tensor product formulation.

Our methods yield four twisted tensor products of $\mathcal{O}_q(\Mat_N)$ and $\mathcal{O}_q(\Mat_N)^{op}$.  The definition of the twisting map is based on the $R$-matrix defined by the standard vector representation for $U_q(\mathfrak{gl}_N)$ and resembles -- though is not the same as --  a quantum double (see Section \ref{section:comp}).  
These four twisted tensor products can each be viewed as graded algebras since all defining relations are homogeneous. 
 Using a $U_q(\mathfrak{gl}_N)$ bi-invariant bilinear form and a refined version of  criteria due to Walton and Witherspoon (see  \cite{WW2018}, and Section \ref{section:PBWdef}), we show that two of these twisted tensor products can be deformed into quantum analogs of the Weyl algebra that respect the $U_q(\mathfrak{gl}_N)$-bimodule structure.  Moreover one of the deformations is isomorphic to $\mathscr{P}\mathscr{D}_q(\Mat_N)$ as an algebra. 
 
 Denote the generators for $\mathcal{O}_q(\Mat_N)$ by $t_{ij}, 1\leq i,j\leq N$ and the generators for $\mathcal{O}_q(\Mat_N)^{op}$ by $\partial_{ij}, 1\leq i,j\leq N$.   Let $R_0$ denote the $R$-matrix defined by the standard vector representation for $U_q(\mathfrak{gl}_N)$ (see (\ref{Rmatrix2})) and set  $R_1 = (R_0)_{21}^{-1}$. 
 These results can be summed up as follows.

\medskip
\noindent
{\bf Theorem A.} {\it There is a twisted tensor product $\mathcal{A}_{\upsilon,\sigma} = \mathcal{O}_q(\Mat_N)\otimes_{\tau_{\upsilon,\sigma}}\mathcal{O}_q(\Mat_N)^{op}$ for each $\upsilon,\sigma\in \{0,1\}$, with the  following relations derived from the twisting map $\tau_{\upsilon,\sigma}$:}
\begin{align*}
\partial_{ea} t_{fb} = \sum_{j,k,d,l}(R_{\sigma}^{t_2})^{dl}_{fe} (R_{\upsilon}^{t_2})^{jk}_{ba}t_{dj}\partial_{lk}
\end{align*}
\emph{for all $e,a,f,b,$ where  $t_2$ denotes the transpose in the second component. 
Each $\mathcal{A}_{\upsilon,\sigma}$ inherits a $U_q(\mathfrak{gl}_N)$-bimodule structure from $\mathcal{O}_q(\Mat_N)$ and $\mathcal{O}_q(\Mat_N)^{op}$. 
Moreover, when $\upsilon=\sigma$, $\mathcal{A}_{\upsilon,\upsilon}$ admits a PBW deformation $\mathcal{W}_{\upsilon,\upsilon}$  that preserves the bimodule structure and $\mathcal{W}_{00}\cong \mathscr{P}\mathscr{D}_q(\Mat_N)$ as algebras.}

\medskip

Using Theorem A and the methods developed for its proof, we obtain  analogous results for the three sets of symmetric pairs under consideration.  In particular, let $\mathfrak{g} = \mathfrak{gl}_n$ in Type AI, $\mathfrak{gl}_{2n}$ in Type AII, and $\mathfrak{gl}_n\oplus \mathfrak{gl}_n$ in the Type A diagonal case. Write $\theta$ for the involution defining the quantum symmetric pair for each of these three families and so $\mathfrak{k} = \mathfrak{g}^{\theta}$ (see Section \ref{section:three-families-defn}).  Let $\mathcal{B}_{\theta}$ be the corresponding  coideal subalgebra of  $U_q(\mathfrak{g})$ which is a quantum analog of $U(\mathfrak{k})$ as defined in \cite{L2002} (see also \cite{K}).   Let $\mathscr{P}_{\theta}$ be the  algebra of quantized functions  on the affine space $\mathcal{V}$ associated to the homogeneous space  defined by  $\mathfrak{g}, \mathfrak{k}$.  This algebra $\mathscr{P}_{\theta}$ is a large subalgebra of the right $\mathcal{B}_{\theta}$ invariants inside of $\mathcal{O}_q(\Mat_n)$  in Type AI, $\mathcal{O}_q(\Mat_{2n})$ in Type AII, and $\mathcal{O}_q(\Mat_n\oplus \Mat_n)$ in the Type A diagonal case.  Similarly, the quantum algebra of constant coefficient differential operators $\mathscr{D}_{\theta}$ consists of  almost all of the right $\mathcal{B}_{\theta}$ invariants inside the opposite algebra of these quantized function algebras.

 Write $x_{ij}$ for the generators of $\mathscr{P}_{\theta}$ and $d_{ij}$ denote the generators of $\mathscr{D}_{\theta}$. We construct four twisted tensor products of  $\mathscr{P}_{\theta}$ and $\mathscr{D}_{\theta}$ by embedding slightly bigger versions of these algebras inside  large graded algebras formed using combinations of the $\mathcal{A}_{\upsilon,\sigma}$, $\upsilon, \sigma\in \{0,1\}$.  In analogy to the $\Mat_N$ setting, we identify a left $U_q(\mathfrak{g})$ and right $\mathcal{B}_{\theta}$ invariant bilinear form and again use the fine-tuned version of the criteria of (\cite{WW2018}) to create quantum analogs of Weyl algebras. This yields the following result similar to Theorem A in the symmetric pair setting.  
 
\medskip
\noindent
{\bf Theorem B.} {\it 
For each $\upsilon,\sigma\in \{0,1\}$, there is a twisted tensor product $\mathcal{A}^{\theta}_{\upsilon,\sigma} =\mathscr{P}_{\theta}\otimes_{\tau^{\theta}_{\upsilon,\sigma}}\mathscr{D}_{\theta}$ with the  following relations derived from the twisting map $\tau^{\theta}_{\upsilon,\sigma}$:}
\begin{align*} 
d_{ab}x_{ef} = \sum_{r,w,p,q,z,y,m,l}(R_{\sigma}^{t_2})_{zq}^{wr}(R_{\sigma}^{t_2})_{ma}^{pq}(R_{\upsilon}^{t_2})^{zy}_{fl}(R_{\upsilon}^{t_2})^{ml}_{eb}x_{pw}d_{ry}
\end{align*}
\emph{for all $e,a,f,b$.
Each $\mathcal{A}^{\theta}_{\upsilon,\sigma}$ inherits a left $U_q(\mathfrak{g})$ and trivial right $\mathcal{B}_{\theta}$ module structure from $\mathscr{P}_{\theta}$ and $\mathscr{D}_{\theta}$.
Moreover, when $\upsilon=\sigma$, $\mathcal{A}^{\theta}_{\upsilon,\upsilon}$ admits a PBW deformation $\mathcal{W}^{\theta}_{\upsilon,\upsilon}$  that preserves the  module structures and is a quantum analog of the Weyl algebra associated to the homogeneous space defined by $(\mathfrak{g},\mathfrak{k})$.}
 
Recall that the relations for $\mathcal{O}_q(\Mat_N)$ can be expressed in a compact matrix format using the 
Faddev-Reshetikhin-Takhtajan construction of quantized function algebras.  We express the relations for the algebras of Theorem A using similar matrix equations (see Sections \ref{section:as} and \ref{section:four}).  For the three families of symmetric pairs, the relations for $\mathscr{P}_{\theta}$ can also be expressed in matrix format.  In this setting, the relations closely resemble the reflection equations and $\mathscr{P}_{\theta}$ is a quotient of a reflection equation algebra (see Proposition \ref{prop:Palgebra}).  For Types AI and AII, this result traces back to \cite{N}.  We show the generators of  $\mathscr{P}_{\theta}$ in the diagonal case also satisfy reflection equation type relations.   Passing to the quantum Weyl algebras, we see that the relations, including the ones in Theorem B, can all be expressed in matrix format in such a way that they resemble reflection equations. 

In the diagonal case, it turns out that there also are isomorphisms  $\mathscr{P}_{\theta}\cong \mathcal{O}_q(\Mat_n)$ and $\mathcal{W}^{\theta}_{00}\cong \mathscr{P}\mathscr{D}_q(\Mat_n)$ as algebras.  Moreover, this becomes a $U_q(\mathfrak{gl}_n)$-bimodule isomorphism by converting the right action of $U_q(\mathfrak{gl}_n)$ on $\mathscr{P}\mathscr{D}_q(\Mat_n)$ to a left action on $\mathcal{W}^{\theta}_{00}$. Similar isomorphisms in the diagonal case hold for the other graded and non-graded algebras of Theorems A and B (see Corollary \ref{diagiso}).

Note that the three families of symmetric pairs in this paper are each closely connected with a Hermitian symmetric pair using the correspondence of \cite{KS1991}, Section 1. 
On the other hand,  ${\rm Pol}(\Mat_N)_q$ is part of a larger family of quantum algebras  associated to Hermitian symmetric pairs 
(see  \cite{VSS}). 
 Explicit generators and relations for ${\rm Pol}(\Mat^{sym}_2)_q$ where $\Mat^{sym}_2$ is the space of $2\times 2$ symmetric matrices  can be found in \cite{B3}.   A straightforward analysis shows that ${\rm Pol}(\Mat^{sym}_2)_q$ does not agree with the corresponding quantum Weyl algebra studied here  (see Remark \ref{remark:comparisons}). One of the problems is that the constant terms showing up in ${\rm Pol}(\Mat^{sym}_2)_q$ are not all equal.  Hence, it is 
 not even clear how to convert ${\rm Pol}(\Mat^{sym}_2)_q$ into a quantum Weyl algebra by  normalizing the generators as is done for ${\rm Pol}(\Mat_N)_q$ in \cite{B}.

 Many of the  results of this paper rely on general properties of $R$-matrices not specific to Type A.  Thus it is likely that these methods yield quantum Weyl algebras for other symmetric pairs, especially the classical ones that have definitions via reflection equations. It would be expected that such Weyl algebras would be deformations of twisted tensor products of two algebras $\mathscr{P}_{\theta}$ and $\mathscr{D}_{\theta}$ just as for the three types studied here.  Moreover, these two algebras, $\mathscr{P}_{\theta}$ and $\mathscr{D}_{\theta}$, are likely to be quotients of reflection equation algebras.  There are other quantum Weyl algebras in the literature built from two algebras, one corresponding to the polynomials and the other to the differentials,  with reflection equation type relations.  At the end of the paper, we discuss two such instances; neither look like they are related to  the Weyl algebras presented here (see Remark \ref{remark:comp2}).

This paper is organized as follows. Section \ref{section:bandn} introduces basic notation concerning Hopf algebras, vector spaces and root systems in Type A.  We turn our attention to the quantized enveloping algebra of $\mathfrak{gl}_N$ in Section \ref{section:QEA}.  After reviewing the definition of $U_q(\mathfrak{gl}_N)$ and its associated vector representation $\rho$, we describe properties of the universal $R$-matrix $\mathcal{R}$ and determine the image of $\mathcal{R}$ under $\rho$ and related maps.  These images and their relation to $\mathcal{R}$ are crucial in our construction of quantum Weyl algebras. 

Section \ref{section:FRT} is devoted to the quantized function algebra on matrices using the standard FRT construction.  Our presentation here takes a module perspective as compared to the usual coalgebra point of view.  This way, we can track the $U_q(\mathfrak{gl}_N)$-bimodule actions from underlying vector space to 
$\mathcal{O}_q(\Mat_N)$  as well as to its opposite algebra $\mathcal{O}_q(\Mat_N)^{op}$.  The three families of symmetric pairs are introduced in Section \ref{section:three-families} along with a specification of the involution $\theta$ for each type and the quantum analog $\mathcal{B}_{\theta}$ of $U(\mathfrak{k})$ in terms of generators.   The quantum function algebra $\mathscr{P}_{\theta}$ and differential algebra $\mathscr{D}_{\theta}$ are described and analyzed both from an algebraic and representation theoretic point of view.  It is here that we see for all three types, $\mathscr{P}_{\theta}$ and $\mathscr{D}_{\theta}$ are just quotients of reflection equation algebras.

In Section \ref{section:graded-weyl}, we define twisted tensor products based on dual pairings  and use them to form four graded versions $\mathcal{A}_{\upsilon,\sigma}$ of quantum Weyl algebras.  The twisted tensor product formulation takes advantage of the bialgebra structure of $\mathcal{O}_q(\Mat_N)$ and its opposite. However, the resulting twisted tensor products do not admit an obvious bialgebra structure (see Section \ref{section:comp}). Although only two of the twisted tensor products $\mathcal{A}_{\upsilon,\sigma}$ can be extended to  non-graded quantum Weyl algebras, all four play a role in constructing quantum Weyl algebras for the three families of quantum symmetric pairs under consideration. 

Section \ref{section:graded-wafhs} creates four twisted tensor products that glue together $\mathscr{P}_{\theta}$ and $\mathscr{D}_{\theta}$. 
Note that the algebra generated by $\mathscr{P}_{\theta}$ and $\mathscr{D}_{\theta}$ inside of $\mathcal{A}_{\upsilon,\sigma}$ is not  isomorphic to a twisted tensor product and in particular contains additional terms such as sums of terms of the form $t_{ij}\partial_{kl}$.  (This happens in the classical case as well.) The desired twisted tensor products are formed by embedding (possibly bigger versions) of  $\mathscr{P}_{\theta}$ and $\mathscr{D}_{\theta}$ inside  algebras created using combinations of  twisting maps from Section \ref{section:graded-weyl}.

The final section, Section \ref{section:QWA}, specifies bilinear forms relating the differential and polynomial parts for both the space of matrices $\Mat_N$ and the homogeneous spaces arising from the three families of symmetric pairs. These bilinear forms are $U_q(\mathfrak{gl}_N)$ bi-invariant  in the first setting and left $U_q(\mathfrak{g})$ 
invariant in the second setting. The key method of this section is a refined version of  a criteria introduced by Walton and Witherspoon in \cite{WW2018}  for producing PBW deformations.  This criteria allows us to check whether or not the deformations produced by these bilinear forms yield quantum analogs  $\mathcal{W}$ and $\mathcal{W}_{\theta}$ of the Weyl algebra satisfying the crucial property that multiplication defines a vector space isomorphism from $\mathcal{O}_q(\Mat_N)\otimes \mathcal{O}_q(\Mat_N)$ (resp. $\mathscr{P}_{\theta}\otimes \mathscr{D}_{\theta}$) to $\mathcal{W}$ (resp. $\mathcal{W}_{\theta}$). As mentioned above and explained further in this final section, this happens for precisely two of the graded quantum Weyl algebras both in the matrix  and the homogeneous space settings.

\medskip\noindent
{\bf Acknowledgements.} The first author presented some of these results in a seminar at Northeastern University in 2021.  She  would like to thank Milen Yakimov for the invitation to speak and  for his insightful suggestions and helpful pointers to the literature after the talk.  The second author was partially supported by NSF grants DMS-1939600, DMS-2001537, and Simons foundation grant 509766.  The third author  was partially supported by an NSERC Discovery Grant (RGPIN-2018-04004).

\section{Background and Notation}\label{section:bandn}
\subsection{Hopf algebras and bialgebras} Many of the algebras we consider in this paper are either bialgebras or Hopf algebras.  So they come equipped with a coproduct $\Delta$ and a couint $\epsilon$; Hopf algebras also have an antipode map $S$.   We use Sweedler notation writing $\Delta(a) = \sum a_{(1)}\otimes a_{(2)}$ for the coproduct of an element $a$.

Let $H$ be a Hopf algebra. A bialgebra $A$ is a left $H$-module algebra  if 
\begin{align*} h\cdot ac = \sum (h_{(1)}\cdot a)(h_{(2)}\cdot c)
\end{align*} for all $h\in H$ and  $a,c\in A$.  The analogous definition works for right $H$-module algebras.  
Given  left $H$-modules $M$ and $M'$, a scalar-valued bilinear form $\langle \cdot, \cdot \rangle$ on $M\times M'$ is called a (left) $H$-invariant bilinear form provided 
\begin{align*}\sum \langle a_{(1)}\cdot m,\ a_{(2)}\cdot  r\rangle =\epsilon(a)\langle m,r\rangle\end{align*} for all $a\in H$, $m\in M$ and $r\in M'$. Right $H$-invariant bilinear forms are defined in a similar fashion.  If $M$ and $M'$ are $H$-bimodules then we say that the bilinear form  is $H$-bi-invariant if it is $H$-invariant with respect to both the right and left actions. 

\subsection{Tensor Product Basics}\label{section:Tensor}
 Unless otherwise specified,  tensor products are  over the base field  which  is $\mathbb{C}(q)$ where 
$q$ is an indeterminate.  Given a tensor product of two vector spaces,  let  $\flip$ denote the flip map  which interchanges the tensor components (i.e. $\flip(a\otimes b) = b\otimes a$)).  

Let  $N$ be a positive integer and write ${\rm Mat}_N$ for the set of $N\times N$ matrices  over $\mathbb{C}$.  Every element in  ${\rm Mat}_{N}\otimes {\rm Mat}_{N}$ can be expressed as a linear combination of  tensors of matrix units $e_{ij}\otimes e_{kl}$.   
Write $t_1$ for the transpose in the first component and $t_2$ for the transpose in the second component so that  $(e_{ij}\otimes e_{kl})^{t_1} =e_{ji}\otimes e_{kl} $ and $(e_{ij}\otimes e_{kl})^{t_2} =e_{ji}\otimes e_{lk}$.  

Consider  $C=\sum C_{(1)}\otimes C_{(2)} \in {\rm Mat}_N\otimes {\rm Mat}_N$.  Write $C_{ij}$ for the operator   acting on a tensor product of $r$ vector spaces all of dimension $N$ in such a way that the first component of $C$ acts on the $i^{th}$ vector space and the second component of $C$ acts on the $j^{th}$ vector space.  In particular, $C_{ij}$  is the sum of  $r$-long tensor products of matrices with $C_{(1)}$ in position $i$ and $C_{(2)}$ in position $j$ and the identity in all other positions.  Note that $C=C_{12}$ and $\flip(C) =C_{21}$.

\subsection{Vector Space Notation} \label{subsection:vector_space}
Let $V$ be a vector space with a distinguished basis $v_1, \dots, v_N$.  Define the action of the  matrix unit $e_{ij}$ on $V$ by   $e_{ij}v_k = \delta_{jk}v_i$ for each $k=1, \dots, N$. Consider the element  $C = \sum_{i,j,k,l}c_{kl}^{ij}e_{ik}\otimes e_{jl}$ in ${\rm Mat}_N\otimes {\rm Mat}_N$ and note that
\begin{align*}
C\cdot (v_k\otimes v_l) = \sum_{s,t}
c_{kl}^{st}e_{sk}\otimes e_{tl} (v_k\otimes v_l) =\sum_{s,t}c_{kl}^{st} v_s\otimes v_t
\end{align*}
The action of the linear transformation on $V\otimes V$ defined by $C$  can be expressed in compact  form as
\begin{align}\label{vector-form2}
V\otimes V\ \rightarrow\  C\cdot (V\otimes V).
\end{align}

We have analogous notions for matrices acting on the right.  For instance, let  $W$ be another  $N$-dimensional vector space with basis $w_1,\dots, w_N$ and set $w_ke_{ij} = \delta_{ik}w_j$ for all $i,j,k$. The map 
\begin{align}\label{right-R-action}
W\otimes W \rightarrow (W\otimes W)\cdot C
\end{align} 
 sends  $(w_a\otimes w_b)$ to $(w_a\otimes w_b)\cdot  C$ where
\begin{align*}
(w_a\otimes w_b)\cdot  C=\sum_{j,s}(w_a\otimes w_b) m^{ab}_{js}e_{aj}\otimes e_{bs}=\sum_{j,s}(w_j\otimes w_s) m^{ab}_{js}.\end{align*}

We often encounter maps similar to (\ref{vector-form2})  and (\ref{right-R-action})  with the extra involvement of  a reordering of vector spaces. This reordering is expressed using  subscripts to denote the position of a particular vector space.  For example, $V\otimes V$ can be written as $V_{(1)}\otimes V_{(2)}$ while 
$\flip(V\otimes V)$ is written as $V_{(2)}\otimes V_{(1)}$.  
This subscript notation enables us to express maps arising from a combination of matrix actions and tensor component permutations in a compact form along the lines of (\ref{vector-form2})
and (\ref{right-R-action}).   For example, the map 
\begin{align*}
V_{(2)}\otimes W_{(2)} \otimes V_{(1)} \otimes W_{(1)} \rightarrow C_{13}\cdot (V_{(1)} \otimes W_{(1)} \otimes V_{(2)}\otimes W_{(2)})\cdot C_{24}
\end{align*}
sends  $v_k\otimes w_l\otimes v_i\otimes w_j$ to 
$\sum_{s,t,a,b}m^{st}_{ik} (v_s\otimes w_a \otimes v_t \otimes w_b) m^{jl}_{ab}$ for all $k,l,i,j$.

\subsection{Roots and Weights}
Let $\epsilon_1, \dots, \epsilon_N$ denote a fixed  orthonormal basis for $\mathbb{R}^N$ with respect to the inner product $(\cdot,\cdot)$.    The set of positive simple roots for the root system of Type A$_{N-1}$ consists of  
$\alpha_i = \epsilon_i-\epsilon_{i+1}$ for $i=1, \dots, N-1$.  The dominant integral weights for the root system of $\mathfrak{gl}_N$ is the set 
 $\Lambda_N$ consisting of those  $\lambda = \lambda_1\epsilon_1 + \cdots +  \lambda_N\epsilon_N$ where each $\lambda_i \in \mathbb{Z}$ and $\lambda_1\geq \lambda_2\geq \cdots \geq \lambda_N$. 

\section{Quantized enveloping algebra for $\mathfrak{gl}_N$}\label{section:QEA}
\subsection{Basic definitions}
%
 Let $q$ be an indeterminate.   The Drinfeld-Jimbo quantized enveloping algebra $U_q(\mathfrak{gl}_N)$   is an algebra over $\mathbb{C}(q)$ generated by $K_{\epsilon_1}^{\pm 1}, \dots, K_{\epsilon_N}^{\pm 1}$,   $E_1, \dots, E_{N-1},$ $ F_1, \dots, F_{N-1}$ subject to the algebra relations as stated in \cite{N} (see also \cite{KS}, Section 10).   Given an integer linear combination $\beta = \sum_{i=1}^N\beta_j\epsilon_j$, write $K_{\beta}$ for the product $K_{\epsilon_1}^{\beta_1}\cdots K_{\epsilon_N}^{\beta_N}$.  Set  $K_i = K_{\alpha_i}=K_{\epsilon_i}K_{\epsilon_{i+1}}^{-1}$ for $i=1, \dots, N-1$.   The subalgebra of $U_q(\mathfrak{gl}_N)$ generated by \begin{align*}K_1^{\pm 1}, \dots, K_{N-1}^{\pm 1}, E_1, \dots, E_{N-1}, F_1, \dots, F_{N-1}\end{align*} is the quantized enveloping algebra $U_q(\mathfrak{sl}_N)$. 

Both $U_q(\mathfrak{gl}_N)$ and $U_q(\mathfrak{sl}_N)$ are Hopf algebras with coproduct $\Delta$, counit $\epsilon$, and antipode $S$ defined on generators by 
\begin{itemize}
\item $\Delta(E_i) = E_i\otimes 1 + K_i\otimes E_i$, $\epsilon(E_i) = 0$ and $S(E_i) = -K_i^{-1}E_i$
\item $\Delta(F_i) = F_i\otimes K_i^{-1} + 1\otimes F_i$, $\epsilon(F_i) = 0$ and $S(F_i) = -F_iK_i$
\item $\Delta(K) = K\otimes K$, $\epsilon(K) = 1$ and $S(K) = K^{-1}$ for all $K=K_{\beta}, \beta\in \sum_j{\mathbb{Z}}\epsilon_j$.
\end{itemize}
Here, we follow the notation in \cite{K} (and \cite{N}) so that our definition of the symmetric pair coideal subalgebras is consistent with this and many of the other papers in the subject.  However, we will also be quoting basic results on quantum groups
 from \cite{KS}.  Although the algebra structure  is the same for these references, the Hopf structure in \cite{KS} is the opposite  one from \cite{K} and thus in this paper. We will automatically adjust, interchanging  tensor components when necessary,  formulas taken from \cite{KS}.

The adjoint action of $U_q(\mathfrak{gl}_N)$ on itself is defined by  $({\rm ad}\ g)\cdot a = \sum g_{(1)} a S(g_{(2)})$ for all $a,g\in U_q(\mathfrak{gl}_N)$.  For generators, the adjoint action takes the following form.
\begin{align*} 
({\rm ad}\ E_i)\cdot a =E_ia-K_iaK_i^{-1} \quad ({\rm ad}\ F_i)\cdot a = F_iaK_i-F_iK_i \quad ({\rm ad} K_{\epsilon_j})\cdot a = K_{\epsilon_j}aK_{\epsilon_j}^{-1}
\end{align*}
for all $i=1,\dots, N-1$ and $j=1,\dots, N$.

The Hopf algebra $U_q(\mathfrak{gl}_N)$ admits an algebra antiautomorphism that preserves the coalgebra structure
defined by  
 \begin{align}\label{starstructure}
K_{\epsilon_i}^{\natural} = K_{\epsilon_i} \quad E_j^{\natural} = q^{-1} F_jK_j \quad F_j^{\natural} = qK_j^{-1}E_j
\end{align}
for $j=1, \dots, N-1$ and $i = 1, \dots, N$.  It should be noted that this algebra antiautomorphism is closely related to the one used to define the Hopf $*$ structure on $U_q(\mathfrak{gl}_N)$.  The difference is that the map $a\mapsto a^{\natural}$ is an algebra map over $\mathbb{C}(q)$ while $a\mapsto a^{*}$ is  conjugate linear. 
The fact that these two maps  agree on the real part of $U_q(\mathfrak{gl}_N)$ allows us to quote properties based on $a\mapsto a^*$ and use them instead in the context of the map  $a\mapsto a^{\natural}$.

The map  $\natural\circ S$ is the algebra isomorphism that satisfies 
\begin{align*}
(S(E_j))^{\natural} = -q^{-1}F_j \quad (S(F_j))^{\natural} = -qE_j \quad (S(K_{\epsilon_i}^{\pm 1}))^{\natural}= K_{\epsilon_i}^{\mp 1}
\end{align*}
and ${\natural}\circ S^{-1}$ is the algebra isomorphism such that 
\begin{align}\label{starinversestructure}
(S^{-1}(E_j))^{\natural} = -qF_j \quad (S^{-1}(F_j))^{\natural} = -q^{-1}E_j \quad (S^{-1}(K_{\epsilon_i}^{\pm 1}))^{\natural}= K_{\epsilon_i}^{\mp 1}
\end{align}
for $j=1, \dots, N-1$ and $i = 1, \dots, N$.  
Since ${\natural}$ preserves the coalgebra structure while $S$ sends the coalgebra structure to its opposite,   $\natural\circ S$ sends $U_q(\mathfrak{gl}_N)$ to $U_q(\mathfrak{gl}_N)^{cop}$.  By \cite{KS}, Section 1.2.7, we have $\natural\circ S\circ \natural= S^{-1}$.

We will also find it helpful to use the $h$-adic version of the Drinfeld-Jimbo algebra, $U_h(\mathfrak{gl}_N)$,  associated to $\mathfrak{gl}_N$.   Recall that  $U_h(\mathfrak{gl}_N)$ is an algebra over $\mathbb{C}[[h]]$ generated by $E_i,F_i,H_{\epsilon_j}$, $i=1, \dots, N-1 ,j=1,\dots, N$ such that 
\begin{align*}
[H_{\epsilon_i},H_{\epsilon_j}] =0, \quad  [H_{\epsilon_i}, E_{\epsilon_j}] = \delta_{ij}  E_i, \quad [H_{\epsilon_i}, F_j] = -\delta_{ij}  F_i,
 \end{align*}
 and the remaining relations, as well as the Hopf algebra structure, agrees with that of  $U_q(\mathfrak{gl}_N)$   with  $K_{\epsilon_i}$ replaced by $e^{hH_{\epsilon_i}}$ and $q$ replaced by $e^h$.  The  map $\natural$ extends to $U_h(\mathfrak{gl}_N)$ with $H_{\epsilon_i}^{\natural} = H_{\epsilon_i}$ all $i$. 

Just as in the classical case, the  weights in $\Lambda_N$ parametrize the finite-dimensional simple $U_q(\mathfrak{gl}_N)$-modules.  We write $L(\lambda)$ for the simple module of highest weight  $\lambda$. All modules  in this paper are  of type 1 (see for example \cite{KS}, Section 6.2).

\subsection{Vector Representations}\label{section:vector}
As in \cite{N} (see also  \cite{KS} Section 8.4.1),  the vector 
representation $\rho$ for $U_q(\mathfrak{gl}_N)$ is defined by  \begin{align*} \rho(K_{\epsilon_i}) = qe_{ii} + \sum_{j\neq i}e_{jj}, i = 1, 2, \dots, n\cr
\rho(E_i)=e_{i,i+1}, \quad \rho(F_i) = e_{i+1,i}, \quad i = 1, 2, \dots, N-1.
\end{align*}  
Let $V$ be the vector space with basis   $\{v_1, \dots, v_N\}$ and set $e_{ik}v_j = \delta_{kj}v_i$ for all $i,j,k$.  The space $V$ becomes the (left) module for the vector representation $\rho$ via the action 
defined by $
a v_j = \rho(a)v_j$
for all $j=1, \dots, n$ and all $a\in U_q(\mathfrak{gl}_N)$.   As a $U_q(\mathfrak{gl}_N)$-module, $V$ is isomorphic to the irreducible module $L(\epsilon_1)$ with highest weight generating vector $v_1$.

Let $W$ be another $N$-dimensional vector space with basis $\{w_1, \dots, w_N\}$ where $w_je_{ki} = \delta_{jk}w_i$ for all $i,j,k$. Give $W$ a right   $U_q(\mathfrak{gl}_N)$-module structure defined by $ w_k\cdot a = w_k({\rho(a)})$
for all $k$ and $a\in U_q(\mathfrak{gl}_N)$. 
The  $U_q(\mathfrak{gl}_N)$-module $W$ is a right dual  to the left $U_q(\mathfrak{gl}_N)$-module $V$ with pairing $\langle w_j, v_k\rangle = \delta_{jk}$
such that \begin{align*}
\langle w_{j}, av_k\rangle = \langle w_{j}, \rho(a)v_k\rangle = \langle w_{j}{ \rho(a)},v_k\rangle
\end{align*}
for all $j,k$ and $a\in U_q(\mathfrak{gl}_N)$.  

Note that we can move the right action of $U_q(\mathfrak{gl}_N)$ on $W$ to a left action.  Indeed, since $\rho(a)$ is a matrix, 
we have $
w_{j} {\rho({a})} = {{\rho({a})}}^t w_j$
where here $t$ is just the standard transpose on $N\times N$ matrices.   A straightforward check using the definition of $\rho$ and $\natural$, we see that 
\begin{align*}
{\rho({a})}^t = \rho(a^{\natural})
\end{align*}
for all $a\in U_q(\mathfrak{gl}_N)$.  Thus acting on the right via the matrix representation ${\rho}$ is the same as acting on the left using the matrix representation $\rho\circ \natural$. 

We can also define a left dual for the $U_q(\mathfrak{gl}_N)$-module $V$ and a right dual  for the $U_q(\mathfrak{gl}_N)$-module $W$  as follows.
Let $V^*$ be the vector space with basis  $\{v_1^*,\dots, v_N^*\}$  and define a linear map from $V$ to $V^*$ that sends $v_i$ to $v_i^*$ all $i=1, \dots, N$.  This allows us to define a (left) $U_q(\mathfrak{gl}_N)$-module action on $V^*$ via
\begin{align}\label{firstaction}
(a\cdot v)^* = (S(a))^{\natural}\cdot v^*
\end{align}
for all $v\in V$ and $a\in U_q(\mathfrak{gl}_N)$.  Similarly, let $W^{*}$ be the right $U_q(\mathfrak{gl}_N)$-module  dual to the right module $W$ with basis $\{w_1^*,\dots, w_N^*\}$  such that  \begin{align}\label{secondaction}
(w\cdot  a)^*  = w^*\cdot (S^{-1}(a))^{\natural}
\end{align}
for all $w\in W$ and $a\in U_q(\mathfrak{gl}_N)$.
It follows that  elements $a$ of $U_q(\mathfrak{gl}_N)$ act on elements of $V^*$ using the matrix $\rho( (S(a))^{\natural}) = (\rho\circ \natural\circ S)(a)$.
A similar analysis shows that elements $a$ of $U_q(\mathfrak{gl}_N)$ act on elements of $W^*$ on the right via the matrix $(\rho( (S^{-1}(a))^{\natural}) )= (\rho\circ {\natural}\circ S^{-1})(a)$.

We can extend
$\rho$
to  $U_h(\mathfrak{gl}_N)$ in a way that is compatible with its definition on  $U_q(\mathfrak{gl}_N)$ by insisting that the image of $E_i$ and $F_i$ is the same under $\rho$ as defined above and 
$
\rho(H_{\epsilon_i}) = e_{ii }
$
for $i=1, \dots, N$. 
Thus $V$ can also be viewed as a (left) $U_h(\mathfrak{gl}_N)$-module with $H_{\epsilon_k}v_i = \delta_{ki}v_i$ for all $i,k$.  Analogous assertions holds for $W$, $V^*$, and $W^*$.

Using the definition of $\rho$ and the formulas for $\natural\circ S$, we see that $(\rho\circ \natural\circ S)$ defines the representation specified by 
\begin{align*}
&(\rho\circ\natural\circ S) (K_{\epsilon_k}) = q^{-1} e_{kk} + \sum_{j\neq k} e_{jj}, \quad (\rho\circ \natural\circ S)(H_{\epsilon_k}) =-e_{kk}
\cr& (\rho\circ\natural\circ S) (E_i) = -q^{-1}e_{i+1,i}, \quad {\rm and }\quad  (\rho\circ\natural\circ S) (F_i) = -qe_{i,i+1}
\end{align*}
for $k=1, \dots, N$ and $i=1, \dots, N-1$. Similar formulas hold for $(\rho\circ \natural\circ S^{-1})$. 


\subsection{The universal $R$-matrix}\label{section:universal}
Recall that $U_h(\mathfrak{gl}_N)$  comes equipped with a universal $R$-matrix $\mathcal{R}$.  The matrix $\mathcal{R}$   is an invertible element in a completion of $U_h(\mathfrak{gl}_N){\otimes} U_h(\mathfrak{gl}_N)$ that satisfies the Quantum Yang-Baxter Equation $\mathcal{R}_{12}\mathcal{R}_{13}\mathcal{R}_{23}=\mathcal{R}_{23}\mathcal{R}_{13}\mathcal{R}_{12}$ along with compatibility relations with the coproduct (see \cite{KS}, Chapter 8 for more details).  

Let $M$ and $M'$ be left $U_q(\mathfrak{gl}_N)$-modules.
As explained in \cite{KS}, 8.1.2, the map  $\flip\cdot (T_M\otimes T_{M'})(\mathcal{R})$ defines a left  $U_q(\mathfrak{gl}_N)$-module isomorphism from $M\otimes {M'}$ to ${M'}\otimes M$ given by 
\begin{align}\label{firstiso}
m\otimes m'\rightarrow \flip\cdot \left((T_M\otimes T_{M'})(\mathcal{R})(m\otimes m') \right)= \sum_iT_{M'}(y_i)m'\otimes T_M(x_i)m
\end{align}
where $\mathcal{R} = \sum_ix_i\otimes y_i$ and $T_M, T_{M'}$ define the respective $U_q(\mathfrak{gl}_N)$ left representations. 
Similarly,  if $M$ and $M'$ are  right $U_q(\mathfrak{gl}_N)$-modules, then there is a right module isomorphism taking $M\otimes M'$ to $M'\otimes M$ given by 
\begin{align}\label{secondiso}
(m\otimes m') \rightarrow 
&\flip\cdot \left((m\otimes m') \cdot (T'_{M}\otimes T'_{M'})(\mathcal{R}^{-1})\right)  \end{align}
where $T'_{M}$, $T'_{M'}$ define the right representations on ${M}$ and ${M'}$ respectively. 
Since $\mathcal{R}_{21}^{-1} $ is also a universal $\mathcal{R}$-matrix for $U_q(\mathfrak{gl}_N)$ (\cite{KS}, Proposition 1),  $\mathcal{R}$ can be replaced with $\mathcal{R}_{21}^{-1}$ in each of the above isomorphisms to get other isomorphisms of $U_q(\mathfrak{gl}_N)$-modules. 

An explicit formula for the  universal $R$-matrix for $U_h(\mathfrak{gl}_N)$  is   (see  \cite{KS} Section 8.3.2, Theorem 17 and Remarks 6 and 7 for more details),\begin{align}\label{universal}
\mathcal{R} =\exp\left( h\sum_{i=1}^NH_{\epsilon_i}\otimes H_{\epsilon_i}\right)\sum_{r_1, \dots, r_m=0}^{\infty}\prod_{j=1}^mq^{{{1}\over{2}}r_j(r_j+1)}{{(1-q^{-2})^{r_j}}\over{[r_j]_{q}!}}F_{\beta_j}^{r_j}\otimes E_{\beta_j}^{r_j}
\end{align}
Here,  $\exp$ is the power series version of the exponential function, $[r]_q!$ is the $q$-factorial at $r$ as defined in \cite{KS}, Section 2.1.1, and $E_{\beta_j}, F_{\beta_j}$ are  the root vectors associated to the positive root $\beta_j$ defined using Lusztig's braid group automorphisms.  

 One can give explicit formulas for the root vectors 
using $q$ commutators as in \cite{KS}, Section 7.3.1.  In particular, define the $q$ commutator by $[a,b]_q = ab-qba$  for all $a,b\in U_q(\mathfrak{gl}_N)$. The analogous definition holds for $q$ replaced 
by $q^{-1}$.  Set $E_{i,i+1} = E_i {\rm \ and \ }F_{i+1,i} = F_i.$
Inductively define 
\begin{align*}
E_{i, j+1} = [E_{i,j}, E_{j,j+1}]_q  {\rm \ and \ } F_{j+1,i} = [F_{j+1,j}, F_{j, i}]_{q^{-1}}.  \end{align*}
Then up to nonzero scalar multiples, $E_{i,j+1}$ is the root vector $ E_{\beta_{ij}}$   and $F_{j+1,i} $ is the root vector $ F_{\beta_{ij}}$ where $\beta_{ij} = \alpha_i+\cdots + \alpha_j$.  
Moreover these scalars are inverses to each other and so $F_{j+1,i}\otimes E_{i,j+1} =F_{\beta_{ij}}\otimes E_{\beta_{ij}}$.

\begin{lemma}\label{lemma:rho-image}
 For all $i,j$ with $1\leq i<j\leq N$, we have $\rho(E_{ij}) = e_{ij}$ and $\rho(F_{ji}) = e_{ji}$.
 \end{lemma}
 \begin{proof} Recall that 
$\rho(E_{i,i+1}) = \rho(E_i) = e_{i,i+1}$ and $\rho(F_{i+1,i})=\rho(F_i) = e_{i+1,i}.$ 
Hence the first two equalities of the lemma hold for $j=i+1$.  Now assume they hold 
for $j-1$ with $j-i>1$.  Then 
\begin{align*} \rho(E_{{ij}}) = \rho( [E_{i,i+1}, E_{i+1,j}]_{q}) =[e_{i,i+1}, e_{i+1,j}]_q = e_{ij}
\end{align*}
since $e_{i+1,j}e_{i,i+1}=0$ and the second equality follows by induction.  A similar induction argument establishes the second equality.  
 \end{proof}


\subsection{Images of the universal $R$-matrix}\label{section:images}

Define the matrix $R$ by \begin{align}\label{Rmatrix2}
R
&= \sum_{1\leq i\leq N} 
qe_{ii}\otimes e_{ii} +\sum_{1\leq i<j\leq N}(e_{ii}\otimes e_{jj} + e_{jj}\otimes e_{ii})
+ (q-q^{-1})\sum_{1\leq j<i \leq N} 
e_{ij}\otimes e_{ji}
\end{align}
This matrix can be written as  \begin{align*}
R = \sum_{i,j,k,l} r^{ij}_{kl} e_{ik}\otimes e_{jl}
\end{align*}
where 
\begin{itemize}
\item $r^{ii}_{ii} = q,$ $r_{ij}^{ij} = 1$
for all $i,j, $ with $i\neq j$.
\item 
$r_{ji}^{ij} = (q-q^{-1})$
for all $j<i$.
\item $r^{ts}_{ji} = 0$ for all other choices of $s,t,i,j$.
\end{itemize}

The next lemma relates the image of the universal $R$-matrix to $R$ under maps  involving $\rho$. 
The argument follows closely a similar computation for $\mathfrak{sl}_N$ in \cite{KS}, Section 8.4.2.

\begin{lemma} \label{Rimage} The image of $\mathcal{R}$ as given in (\ref{universal}) under $\rho\otimes \rho$ is 
the matrix $R$  (\ref{Rmatrix2}) and the image of $\mathcal{R}$ under  both $(\rho\otimes \natural\otimes S)\otimes (\rho\otimes \natural\otimes S)$
and $(\rho\otimes \natural\otimes S^{-1})\otimes (\rho\otimes \natural\otimes S^{-1})$  is $R_{21}.$
\end{lemma}

\begin{proof} 
Since $\rho(H_{\epsilon_i}) = e_{ii}$  we have 
\begin{align*}
(\rho\otimes \rho)(\sum_{i}H_{\epsilon_i}\otimes H_{\epsilon_i})  = \sum_{i}e_{ii}\otimes e_{ii}
\end{align*}
and hence 
\begin{align*}
(\rho\otimes \rho)&(\exp\left( h\sum_{i=1}^nH_{\epsilon_i}\otimes H_{\epsilon_i}\right))=\exp(h\sum_{i=1}^Ne_{ii}\otimes e_{ii}) 
\cr&= \exp(h e_{ii}\otimes e_{ii}) + \sum_{1\leq j<i\leq n}(e_{ii}\otimes e_{jj} + e_{jj}\otimes e_{ii}) 
\cr&= qe_{ii}\otimes e_{ii} + 
\sum_{1\leq j<i\leq N}(e_{ii}\otimes e_{jj} + e_{jj}\otimes e_{ii})
\end{align*}
Using Lemma \ref{lemma:rho-image}, we see that for $i>j$
\begin{align}\label{expansion}
(\rho\otimes \rho)&\left({\rm exp}_q((1-q^{-2})(F_{{ij}}\otimes E_{ji}))\right)
\cr&={\rm exp}_q[(1-q^{-2})(e_{ij}\otimes e_{ji})]
=1 + (q-q^{-1})(e_{ij}\otimes e_{ji})
\end{align}
where 
$
{\rm exp}_qx = \sum_{r=0}^{\infty}q^{r(r+1)/2}x^r/[r]_q!$.
Observe that terms of degree $2$ or higher in the final term of (\ref{expansion}) vanish because $e_{ij}e_{ij}=e_{ji}e_{ji}=0$ for $i\neq j$.
Therefore
 \begin{align*}
 (\rho\otimes \rho)&(\prod_{1\leq j<i\leq N} \left({\rm exp}_q((1-q^{-2})(F_{{ij}}\otimes E_{ji}))\right) = \prod_{1\leq j<i\leq N} \left(1 + (q-q^{-1})(e_{ij}\otimes e_{ji})\right)
 \cr & = 1 + \sum_{1\leq j<i \leq N}(q-q^{-1})(e_{ij}\otimes e_{ji})
 \end{align*}
 because for $j< i$ and $k< l$ we cannot have both 
$e_{ji}e_{kl}$ and $e_{lk}e_{ij}$ nonzero. 
Hence, 
\begin{align*}
(\rho\otimes \rho)(\mathcal{R}) = (qe_{ii}\otimes e_{ii} + 
\sum_{1\leq j<i\leq N}(e_{ii}\otimes e_{jj} + e_{jj}\otimes e_{ii}))\prod_{1\leq i<j\leq N}(1 + (q-q^{-1})(e_{ij}\otimes e_{ji}))\end{align*}
which equals $R$ of (\ref{Rmatrix2}) as claimed.  

By  \cite{KS}, Propostion 2, $({S}\otimes {S})(\mathcal{R}) = \mathcal{R}$.  Clearly, we also have $({S}^{-1}\otimes {S}^{-1})(\mathcal{R}) = \mathcal{R}$. Hence
\begin{align*}
((\rho\otimes \natural\otimes S)\otimes (\rho\otimes \natural\otimes S))(\mathcal{R}) &= ((\rho\otimes \natural)\otimes (\rho\otimes \natural))(\mathcal{R}) \cr &= 
((\rho\otimes \natural\otimes S^{-1})\otimes (\rho\otimes \natural\otimes S^{-1}))(\mathcal{R}). 
\end{align*}
Thus, to complete the lemma,  it is sufficient to show that  $((\rho\otimes \natural)\otimes (\rho\otimes \natural))(\mathcal{R}) = R_{21}$.
Recall that $\rho\circ \natural(a) = {\rho({a})}^t$.  Therefore 
\begin{align*}
((\rho\otimes \natural)\otimes (\rho\otimes \natural))(\mathcal{R}) = \left((\rho\otimes \rho)(\mathcal{R})\right)^{t_1t_2} = R^{t_1t_2}
\end{align*}  It follows from direct inspection of the 
formula for $R$ in (\ref{Rmatrix2}) that $R^{t_1t_2} = R_{21}$.
\end{proof}

The next lemma evaluates $\mathcal{R}$ with respect to  other combinations of $\rho$, $\rho\circ \natural\circ S$ and $\rho\circ \natural\circ S^{-1}$. 
 \begin{lemma}\label{moreRimage} Let $R$ be the matrix defined in (\ref{Rmatrix2}).  
We have 
\begin{itemize}
\item[(i)] $( (\rho\circ \natural\circ S)\otimes \rho)(\mathcal{R}) = (R_{21}^{-1})^{t_2}$ 
\item[(ii)]  $((\rho\circ \natural\circ S)\otimes \rho)(\mathcal{R}_{21}^{-1}) = R^{t_2}$
\item[(iii)] $(\rho\otimes (\rho\circ \natural\circ S^{-1}))(\mathcal{R}) = (R_{21}^{-1})^{t_1}$
\item[(iv)] $(\rho\otimes (\rho\circ \natural\circ S^{-1}))(\mathcal{R}_{21}^{-1})= R^{t_1} $
\end{itemize}
\end{lemma}
\begin{proof}
By \cite{KS}, Proposition 2, $(S\otimes Id)(\mathcal{R}) = \mathcal{R}^{-1}$ and $(Id\otimes  {S})(\mathcal{R}^{-1}) = \mathcal{R}$. It follows that \begin{align*}( (\rho\circ \natural\circ S)\otimes \rho)(\mathcal{R})=((\rho\circ \natural)\otimes \rho)(\mathcal{R}^{-1})\end{align*} The same assertion holds for $\mathcal{R}$ replaced by $\mathcal{R}_{21}^{-1}$ since the latter is also a universal $R$-matrix.  Hence 
we prove (i) and (ii) by determining the image of $\mathcal{R}$ and $\mathcal{R}_{21}$ under $(\rho\circ \natural)\otimes \rho$.  

Using the fact that $(\rho\circ\natural) (a) = {\rho({a})}^t$   and Lemma \ref{Rimage}, we see that 
\begin{align*}
((\rho\circ \natural)\otimes \rho)(\mathcal{R}^{-1}) =\left( (\rho\otimes \rho)(\mathcal{R}^{-1})\right)^{t_1}=\left( ((\rho\otimes \rho)(\mathcal{R}))^{-1}\right)^{t_1} = (R^{-1})^{t_1} = (R_{21}^{-1})^{t_2}.
\end{align*}
Similarly, 
\begin{align*}
((\rho\circ \natural)\otimes \rho)(\mathcal{R}_{21}) =\left( (\rho\otimes \rho)(\mathcal{R}_{21})\right)^{t_1} = (R_{21})^{t_1} = R^{t_2}.
\end{align*}
This proves (i) and (ii). 

Since $(S\otimes S)(\mathcal{R}) = \mathcal{R}$ (\cite{KS}, Proposition 2), we have
\begin{align*}
((\rho\otimes (\rho\circ \natural\circ S^{-1}))(\mathcal{R}))^{t_1} &= (\rho^t\otimes (\rho\circ \natural\circ S^{-1}))(\mathcal{R}) = ((\rho\circ \natural)\otimes (\rho\circ \natural\circ S^{-1}))(\mathcal{R})
\cr&=((\rho\circ \natural\circ S)\otimes (\rho\circ \natural))(\mathcal{R})=((\rho\circ \natural\circ S)\otimes \rho)(\mathcal{R})^{t_2}.
\end{align*}
Therefore $((\rho\otimes (\rho\circ \natural\circ S^{-1}))(\mathcal{R})) = ((\rho\circ \natural\circ S)\otimes \rho)(\mathcal{R})^{t_1t_2}$. Using the same argument, this equality holds with $\mathcal{R}$ replaced with $\mathcal{R}_{21}^{-1}$.  Thus  assertions (iii) and (iv) follow from applying $t_1t_2$ to 
(i) and (ii).
\end{proof}

\section{Quantized Functions on Matrices}\label{section:FRT}
\subsection{FRT Construction} \label{section:FRT1}We review here the basics about the Faddeev-Reshetikhin-Takhtajan (FRT) construction of quantized functions on $N\times N$ matrices.  A good reference for additional details is \cite{KS}, Chapter 9.

 Let $\zeta$ be an $s$-dimensional representation of $U_q(\mathfrak{gl}_N)$ and set $R_{\zeta} = (\zeta\otimes \zeta)(\mathcal{R})$.  Let $M$ be the $s^2$-dimensional vector space spanned by the $m_{ij}, 1\leq i,j\leq s$. The  FRT bialgebra $A(R_{\zeta})$ is the quotient of the tensor algebra $T(M)$ by the ideal generated by 
 \begin{align}\label{FRTrelns}\sum_{j,k}(R_{\zeta})_{jk}^{ld} m_{ja}\otimes m_{kb} - \sum_{j,k}m_{dk}\otimes m_{lj}(R_{\zeta})^{jk}_{ab}.
\end{align}
for all $i,j,a,b$ and coalgebra structure inherited from $T(M)$ so that 
$
\Delta(m_{ij}) = \sum_{k}m_{ik}\otimes m_{kj}$ and $\epsilon(m_{ij}) = \delta_{ij}
$
for all $i,j$.

\subsection{Algebra Structure} \label{section:as}
When $\zeta=\rho$, $R_{\rho} = (\rho\otimes \rho)(\mathcal{R}) = R$ and  the FRT bialgebra $A(R_{\rho})$ is  the quantized function algebra $\mathcal{O}_q(\Mat_N)$
on $N\times N$ matrices.  In this case,  $M$ can be identified with  $V\otimes W$ and each $m_{ij}$ wiith $v_i\otimes w_j$.  Moreover, we write $t_{ij}$ for   the image of $m_{ij}$ in $A(R_{\rho})$.  Using the explicit formula for entries of $R$, relations  (\ref{FRTrelns})  at $\zeta=\rho$  become
  \begin{itemize}
\item[(i)]$t_{ki}t_{kj} = qt_{kj}t_{ki}$, $t_{ik}t_{jk} = qt_{jk}t_{ik}$ ($i<j$)
\item[(ii)] $t_{il}t_{kj}  = t_{kj}t_{il}, t_{ij}t_{kl} -t_{kl}t_{ij} = (q-q^{-1})t_{kj}t_{il}$ ($i<k;j<l$)
\end{itemize}
Set $T=(t_{ij})$, the matrix with $ij$ entry equal to $t_{ij}$.  Set $T_1=T\otimes Id$ and $T_2 = Id\otimes T$.   As in \cite{KS}, Section 9.1.1, these relations can be written in matrix form as 
\begin{align}\label{FRTmatrix}
RT_1T_2 = T_2T_1R.
\end{align}
It is straightforward to check that the map $\iota$ defined by 
\begin{align}\label{iota} \iota(t_{ij}) = t_{ji} 
\end{align} 
for all $i,j=1,\dots, n$ defines an algebra automorphism of $\mathcal{O}_q(\Mat_N)$.  
It is well-known (see for example \cite{NYM}, Theorem 1.4) that $\mathcal{O}_q(\Mat_N)$  has a PBW type basis consisting of elements of the form 
\begin{align}\label{PBWbasis-ordinary-Weyl}
t_{11}^{m_{11}}t_{12}^{m_{12}}\cdots t_{1N}^{m_{1N}}t_{2,1}^{m_{21}}t_{22}^{m_{22}}\cdots t_{2N}^{m_{2N}}\cdots t_{N1}^{m_{N1}}\cdots t_{NN}^{m_{NN}}.
\end{align}
Examining the relations, we see that  each $t_{ij}$ can be replaced with $t_{N-i,N-j}$ in (\ref{PBWbasis-ordinary-Weyl}) and yield another set of monomials that form a basis for $\mathcal{O}_q(\Mat_N)$.

Applying $\flip$ to both sides of (\ref{FRTmatrix}) provides an  equivalent set of equations in matrix form 
$
R_{21}T_2T_1=T_1T_2R_{21}$.
Multiplying on the left and right of both sides by $R_{21}^{-1}$ and then switching the sides gives us 
$R_{21}^{-1}T_1T_2 = T_2T_1R_{21}^{-1}.$
In other words, the same FRT construction using the universal $R$-matrix $\mathcal{R}_{21}^{-1}$ instead of $\mathcal{R}$, results in  the same algebra $\mathcal{O}_q(\Mat_N)$.  Applying the map $\iota$ to both sides yields yet another formulation of these relations
$RT_1^tT_2^t=T_2^tT_1^tR$.


\subsection{Module Realization}\label{section:mr}
 Recall that $V$ is a left $U_q(\mathfrak{gl}_N)$-module and $W$ is a right $U_q(\mathfrak{gl}_N)$-module.  Using the coproduct  for $U_q(\mathfrak{gl}_N)$, $M$ becomes a  $U_q(\mathfrak{gl}_N)$-bimodule
 and its tensor algebra $T(M)$ becomes  a $U_q(\mathfrak{gl}_N)$-bimodule algebra.  
 \begin{lemma} \label{lemma:actions}The algebra $\mathcal{O}_q(\Mat_N)$ is a $U_q(\mathfrak{gl}_N)$-bimodule algebra with left  action defined by 
 \begin{align*}
 &E_{k} \cdot t_{i+1,j} =\delta_{ik}t_{ij},
 \quad 
F_{k} \cdot t_{ij} =\delta_{ik}  t_{i+1,j},
\quad K_{\epsilon_r} \cdot t_{ij} = q^{\delta_{ir}} t_{ij},
\end{align*} 
and
 right action defined by 
\begin{align*} &t_{ij} \cdot E_k=\delta_{jk}t_{i,j+1},\quad
 t_{i,j+1}\cdot F_k=\delta_{jk} t_{ij}, \quad
t_{ij} \cdot K_{\epsilon_r}= q^{\delta_{jr}}t_{ij},
\end{align*} 
for $r,i,j= 1, \dots, N$ and $k=1, \dots, N-1$.  Moreover, the right action is related to the left via 
$t_{ij}\cdot a = \iota({a}^{\natural}\cdot t_{ji})$ for all $a\in U_q(\mathfrak{gl}_N)$ and $i,j\in \{1, \dots, N\}$.
 \end{lemma}
 \begin{proof}  Taking $\zeta=\rho$ and noting that $(R^{t_1t_2})^{jk}_{ld} = R^{ld}_{jk}$, we can rewrite (\ref{FRTrelns}) in this case as 
 \begin{align*}\sum_{j,k}(R^{t_1t_2})^{jk}_{ld} (v_j\otimes w_a)\otimes (v_k\otimes w_b) - \sum_{j,k}(v_d\otimes w_k)\otimes (v_l\otimes w_j)\flip(R^{t_1t_2})_{kj}^{ba}.
\end{align*}
By Lemma \ref{Rimage}, we have $(\rho\otimes \rho)(\mathcal{R}_{21}) = R^{t_1t_2}$.  It follows that $(\rho\otimes \rho) (\mathcal{R}_{12}) = \flip (R^{t_1t_2})$.  Thus, the combined relations for all choices of $a,b,d,l$ corresponds to  the  map of vector spaces 
 \begin{align}\label{matrixrelnVW}\mathcal{R}_{31}\cdot (V_{(2)}\otimes W_{(2)})\otimes (V_{(1)}\otimes W_{(1)}) \rightarrow (V_{(1)}\otimes W_{(1)})\otimes (V_{(2)}\otimes W_{(2)})\cdot \mathcal{R}_{24}.
\end{align}
In particular, these relations are simply the difference of a typical element preceding the arrow in (\ref{matrixrelnVW}) with its image after the arrow.  If we ignore the contributions from $W$, this becomes 
\begin{align}\label{justV}
R_{21}\cdot (V_{(2)}\otimes V_{(1)}) \rightarrow V_{(1)}\otimes V_{(2)} 
\end{align}
Recall that $V$ is a left $U_q(\mathfrak{gl}_N)$-module defined by the representation $\rho$ and that by Lemma \ref{Rimage} we have  $(\rho\otimes \rho)(\mathcal{R}) = R$.  Furthermore, by (\ref{firstiso}), the map 
\begin{align*}
V_{(1)}\otimes V_{(2)} \longrightarrow \flip\cdot ((\rho\otimes \rho)(\mathcal{R})(V_{(1)} \otimes V_{(2)})) = R_{21}\cdot (V_{(2)}\otimes V_{(1)})
\end{align*}
is an isomorphism of left $U_q(\mathfrak{gl}_N)$-modules.  It follows that (\ref{justV}) and hence (\ref{matrixrelnVW}) is invariant under the left action of $U_q(\mathfrak{gl}_N)$ and hence $\mathcal{O}_q(\Mat_N)$ inherits the structure of a left $U_q(\mathfrak{gl}_N)$-module algebra from $T(M)$. 

For the right action, we ignore the contributions from $V$  in (\ref{matrixrelnVW}).  Using (\ref{secondiso}) with $\mathcal{R}_{21}^{-1}$ instead of $\mathcal{R}$ in this mapping, we see that the map 
\begin{align*} 
W_{(2)}\otimes W_{(1)} \rightarrow \flip\cdot \left((W_{(2)}\otimes W_{(1)})\cdot (\rho\otimes \rho) (\mathcal{R}_{21})\right)
\end{align*}
is an isomorphism of right $U_q(\mathfrak{gl}_N)$-modules.   It follows that (\ref{matrixrelnVW}) is invariant under the right action of $U_q(\mathfrak{gl}_N)$. Thus $\mathcal{O}_q(\Mat_N)$ inherits the  right $U_q(\mathfrak{gl}_N)$-module algebra structure from $T(M)$.

 Explicit formulas  for this bimodule action on the generators of $\mathcal{O}_q(\Mat_N)$  follow directly from the explicit formulas for the actions on $V$ and  $W$ derived from the definition of $\rho$.  The final assertion is easily checked using (\ref{starstructure}). 
\end{proof}

Let $\mathcal{O}_q(\Mat_N)^{op}$ denote the bialgebra   with the same coalgebra structure  and opposite multiplication as $\mathcal{O}_q(\Mat_N)$.  Write $\partial_{ij}$, $1\leq i,j\leq N$ for the generators of $\mathcal{O}_q(\Mat_N)^{op}$ so that the  map sending $t_{ij}$ to $\partial_{ij}$, which we denote by $t_{ij}\rightarrow t_{ij}^* = \partial_{ij}$, for all $i,j$ is an algebra anti-isomorphism and coalgebra isomorphism.  
The bialgebra  $\mathcal{O}_q(\Mat_N)^{op}$ is also an FRT bialgebra.  In this case,  $M= V^*\otimes W^*$,     $m_{ij} = v_i^*\otimes w_j^*$ for each $i,j$, 
and $\partial_{ij}$ is the image of $m_{ij}$ when we pass from the tensor algebra of $M$ to the FRT algebra $A(R_{\rho\circ \natural\circ S})$.  Set $P$ equal to the matrix with 
$ij$ entry equal to $\partial_{ij}$ and write $P_1$ for $P\otimes Id$ and $P_2$ for $Id \otimes P$.  The  relations for $\mathcal{O}_q(\Mat_N)^{op}$ can be written in matrix form as 
\begin{align}\label{FRTmatrixop}
R_{21}P_1P_2 = P_2P_1R_{21}.
\end{align} Just as for the $t_{ij}$, the map $\iota:\mathcal{O}_q(\Mat_N)^{op} \rightarrow
\mathcal{O}_q(\Mat_N)^{op}$ defined on generators by $\iota(\partial_{ij}) = \partial_{ji}$ all $i,j$ is an algebra isomorphism.

\begin{lemma} \label{lemma:opposite relation}The algebra $\mathcal{O}_q(\Mat_N)^{op}$ is a $U_q(\mathfrak{gl}_N)$-bimodule algebra with left action given by 
 \begin{align*}
 &E_{k} \cdot \partial_{ij} =-\delta_{ik}q^{-1}\partial_{i+1,j},
 \quad 
F_{k} \cdot \partial_{i+1,j} =-\delta_{ik}  q\partial_{ij},
\quad K_{\epsilon_r} \cdot \partial_{ij} = q^{-\delta_{ir}} \partial_{ij},
\end{align*} 
and right action given by 
\begin{align*} &\partial_{i,j+1} \cdot E_k=-\delta_{jk}q\partial_{ij},\quad
 \partial_{i,j}\cdot F_k=-\delta_{jk} q^{-1}\partial_{i,j+1}, \quad
\partial_{ij} \cdot K_{\epsilon_r}= q^{-\delta_{jr}}\partial_{ij},
\end{align*} 
for $r,i,j= 1, \dots, N$ and $k=1, \dots, N-1.$  Moreover, the right action is related to the left via 
$\partial_{ij}\cdot a = \iota({a}^{\natural}\cdot \partial_{ji})$ for all $a\in U_q(\mathfrak{gl}_N)$ and $i,j\in \{1, \dots, N\}$.
 \end{lemma}
 \begin{proof}
By Lemma \ref{Rimage}, we have  $((\rho\otimes \natural\otimes S)\otimes (\rho\otimes \natural\otimes S))(\mathcal{R}) = R_{21}$. The proof is identical to the proof of Lemma \ref{lemma:actions} with  $R$ replaced by $R_{21}$ and $\rho$ replaced by $\rho\circ \natural\circ S$. \end{proof}

It is straightforward to check that  the   $U_q(\mathfrak{gl}_N)$-bimodule structures of $\mathcal{O}_q(\Mat_N)$ and $\mathcal{O}_q(\Mat_N)^{op}$ are related by the following formulas
$(a\cdot f)^* = (S( a))^{\natural}\cdot f^* $ and $ (f\cdot a)^*=f^*\cdot (S^{-1}(a))^{\natural}$
for all $f\in \mathcal{O}_q(\Mat_N)$  and all $a\in U_q(\mathfrak{gl}_N)$.

\section{Quantum Homogeneous Spaces}\label{section:three-families}
\subsection{Three families}\label{section:three-families-defn} In this section, we consider three families of symmetric pairs $\mathfrak{g}, \mathfrak{k}$ where $\mathfrak{g}$ is a complex Lie algebra,  $\theta$ is an involution of $\mathfrak{g}$, and $\mathfrak{k} = \mathfrak{g}^{\theta}$.
In each case,  $\theta$ takes the form $x\mapsto -Jx^tJ^{-1}$ for an appropriate matrix $J$. We give  the generators for the fixed subalgebras $U(\mathfrak{k})$ in terms of standard Chevalley generators.     Afterwards, we specify the generators for the right coideal subalgebras $\mathcal{B}_{\theta}(b)$  of $U_q(\mathfrak{g})$ which are quantum analogs of $U(\mathfrak{k})$.  Here, we follow the presentation in \cite{K} with the obvious extension from the semisimple case to the reductive setting. 
 Note that the $b$ stands for nonzero parameters $b=(b_1, \dots, b_m)$ that  correspond to  Hopf algebra automorphisms of $U_q(\mathfrak{g})$.  In particular, the quantum analogs given below are all related to each other via Hopf algebra automorphisms. 
  In the presentation below, all Lie algebras are complex and we omit $\mathbb{C}$ from the notation.

\medskip
\noindent
{\bf Type AI}:  $\mathfrak{g} = \mathfrak{gl}_n$ and $\theta$ is defined by $\theta(x) = -x^t$ for all $x\in \mathfrak{gl}(n)$ and $J=I_n$, the $n\times n$ identity matrix. 
In terms of Chevalley generators, $\theta(e_i) = -f_i$, 
$\theta(f_i) = -e_i$ and $\theta(h_{\epsilon_j}) = -h_{\epsilon_j}$ each $i=1, \dots, n-1$ and $j=1, \dots, n$.  Hence $\mathfrak{k}$ is generated by $e_i-f_i, i=1, \dots, n-1$.
Passing to the quantum case,   $\mathcal{B}_{\theta}(b)$  is generated by $F_i-b_iE_iK_i^{-1}$, for $i=1, \dots, n-1$.   

\medskip
\noindent
{\bf Type AII:} $\mathfrak{g} = \mathfrak{gl}_{2n}$ and $\theta$ is defined by $\theta(x) = -Jx^tJ^{-1}$ where 
\begin{align*} J = \sum_{k=1}^ne_{2k-1,2k}-e_{2k,2k-1}.
\end{align*}  
In this case, we have
\begin{itemize}
\item $\theta(e_i) = e_i, \theta(f_i) = f_i, \theta(h_i) = h_i$ for $i=1, 3, \dots, 2n-1$
\item $\theta(e_i) = - [f_{i-1}, [f_{i+1}, f_{i}]]$ for  $i$ even.
\item $\theta(h_{\epsilon_{2i-1}}) = -h_{\epsilon_{2i}}$ 
 for $i=1, \dots, n$.
\end{itemize}
Hence $\mathfrak{k}$ is generated by $h_i, e_i, f_i $ for $i$ odd and  $f_i -[f_{i-1},[f_i, f_{i+1}]]$ for $i$ even.  
 Passing to the quantum case, we have that 
$\mathcal{B}_{\theta}(b) $ is generated by 
\begin{itemize}
\item $K_i^{\pm 1}, E_i,F_i$ for $i$ odd 
\item $B_i= F_i -b_i (({\rm ad}\ E_{i-1}E_{i+1})E_i)K_i^{-1}= F_i-b_i[E_{i-1}, [E_{i+1}, E_i]_q]_qK_i^{-1}$ for $i$ even.
\end{itemize}

\bigskip
\noindent
{\bf Type A diagonal case}:   $\mathfrak{g} = \mathfrak{gl}_n\oplus \mathfrak{gl}_n$ viewed as the  Lie subalgebra  of $\mathfrak{gl}_{2n}$ consisting of block diagonal matrices 
\begin{align*}\left(\begin{matrix} \mathfrak{gl}_n &0\cr 0 &\mathfrak{gl}_n\cr \end{matrix}\right)\end{align*}  and $\theta$ is defined by  \begin{align*}\theta\left(\begin{matrix}x &0\cr 0 &y\cr \end{matrix}\right)=-J\left(\begin{matrix}x^t &0\cr 0 &y^t\cr \end{matrix}\right)J^{-1}=\left(\begin{matrix}-y^t &0\cr 0 &-x^t\cr \end{matrix}\right)\end{align*} 
where 
\begin{align*}
J= \sum_{k=1}^ne_{k,n+k} + e_{n+k,k}=\left(\begin{matrix} 0 &I_n\cr I_n &0\cr\end{matrix}\right).
\end{align*}
Using this notation, the set of Chevalley generators  for $\mathfrak{gl}_n\oplus \mathfrak{gl}_n$ is the union of two sets, the first consisting of the generators for the first copy of $\mathfrak{gl}_n$ and the second consisting of  the generators for the second copy of $\mathfrak{gl}_n$.  We write this as 
 $e_k,f_k,h_{\epsilon_j} $  for the first copy of $\mathfrak{gl}_n$ and $e_{n+k}, f_{n+k}, h_{\epsilon_{n+j}}$   for the second copy where  $k=1, \dots, n-1$ and $ j=1, \dots, n$. Note that  $\theta$ satisfies 
$\theta(f_i) =- e_{n+i}, \ \theta(f_{n+i}) =- e_i, \theta(h_{\epsilon_j}) = -h_{\epsilon_{n+j}}$
for $i=1, \dots, n-1$ and $j=1, \dots, n$.  Hence,   $\mathfrak{k}$ is generated by $f_i-e_{n+i}$, $ f_{n+i}-e_i$, and $\epsilon_j-\epsilon_{n+j}$ for $ i=1, \dots, n-1$.
Passing to the quantum case, the corresponding quantum symmetric pair coideal subalgebra $\mathcal{B}_{\theta}(b)$ is generated by 
\begin{align*}
B_i = F_i -b_iE_{n+i}K_{i}^{-1}, \quad B_{n+i}= F_{n+i} -b_iE_{i}K_{n+i}^{-1},{\rm \quad and\quad}(K_{\epsilon_j}^{- 1}K_{\epsilon_{n+j}})^{\pm 1}
\end{align*}
for $i=1, \dots, n-1$ and $j=1, \dots, n$.  

\bigskip
 In the remainder of the paper, we frequently use the rank of the Lie algebra $\mathfrak{g}$ to specify various parameters.  This rank, denoted ${\rm rank}(\mathfrak{g})$,  is just the dimension of the Cartan subalgebra of $\mathfrak{g}$.   For the Type AI family, this rank is $n$. It is $2n$ for Type AII as well as for the diagonal family.

\subsection{Invariant Elements}\label{section:ie}
To simplify notation, we write $\mathscr{P}= \mathcal{O}_q(\Mat_N)$ in Type AI (with $N=n$) and Type AII (with $N=2n$) and let $\mathscr{P} = \mathcal{O}_q({\rm Mat}_n)\otimes \mathcal{O}_q({\Mat}_n)$ for diagonal type. We represent the generators of $\mathscr{P}$ using $t_{ij}$ for all three families and write $T=(t_{ij})$, the matrix with $i,j$ entry equal to $t_{ij}$.  This is the standard way for the first two families. For the diagonal type,   $t_{ij}$, $i,j=1, \dots, n$ are  generators of the first copy of $\mathcal{O}_q(\Mat_n)$ and $t_{n+i,n+j} $, $i,j=1, \dots, n$ are the generators of the second copy.  Moreover $t_{i,n+j} = t_{n+i,j} = 0$ for all $i,j=1, \dots, n$. 

Using Lemma \ref{lemma:actions}, we give $\mathscr{P}$  the structure of a $U_q(\mathfrak{g})$-bimodule.  Again, for Types AI and AII, this is standard.   For the diagonal type, the $U_q(\mathfrak{g})$-bimodule structure is set so that the first copy of $U_q(\mathfrak{gl}_n)$ (i.e. the one generated by the $E_i,F_i, K_{\epsilon_j}^{\pm 1}$, $i=1,\dots, n-1$, $j=1, \dots, n$) acts on both the left and the right on 
the copy of $\mathcal{O}_q(\Mat_N)$ generated by the $t_{ij}$, $1\leq i,j\leq n$, just as in Lemma \ref{lemma:actions}.  The left and right action of the  second copy of $U_q(\mathfrak{gl}_n)$ (where here we take as generators $E_{n+i}, F_{n+i}, K_{\epsilon_{n+j}}^{\pm 1}, $  $i=1,\dots, n-1$, $j=1, \dots, n$) on 
the copy of $\mathcal{O}_q(\Mat_n)$ generated by the $t_{n+i,n+j}$, $1\leq i,j\leq n$, is also the same, where here, each term with an $i$ subscript is replaced with a term using $i+n$ as subscript.   
  We also insist that elements of the first copy of $U_q(\mathfrak{gl}_n)$ act trivially on each element of the second copy of $\mathcal{O}_q(\Mat_n)$ and vice versa.

For Types AI and AII, set $R_{\mathfrak{g}}$ equal to the matrix defined by (\ref{Rmatrix2}) where $N=n$ in Type AI and $N=2n$ in Type AII. 
The reflection equation associated  to the families of Type AI and AII is the matrix equation defined  as in \cite{N} by
\begin{align}\label{reflection-equation}
R_{\mathfrak{g}}J_1 R_{\mathfrak{g}}^{t_1} J_2 = J_2 R_{\mathfrak{g}}^{t_1} J_1 R_{\mathfrak{g}}
\end{align}
where $J$ is a ${\rm rank}(\mathfrak{g})\times {\rm rank}(\mathfrak{g})$ matrix, $I$ is the ${\rm rank}(\mathfrak{g})\times {\rm rank}(\mathfrak{g})$  identity matrix,   $J_1 = J\otimes I$ and 
$J_2 = I\otimes J$. By \cite{N},  the following  solutions  to these reflection equations 
\begin{align*}
J_{(n)}(a) &= \sum_{k=1}^na_ke_{kk}{\rm  \quad and\quad }
J_{(n)} (a)= \sum_{k=1}^na_k(e_{2k-1, 2k}-qe_{2k,2k-1})\end{align*}
are used in the construction of quantum symmetric pairs, the first in Type AI and the second in Type AII.
Here, $n$ refers to the number of parameters, not the size of the matrix.  In particular, for both solutions,    $a=(a_1, \dots, a_n)$ is an $n$-tuple of nonnegative scalars.

We can also define an $R$-matrix in the diagonal setting and the associated reflection equations.  In particular, for the diagonal family, let $R_{\mathfrak{g}}$ be the $4n^2\times 4n^2$ matrix with entries 
\begin{align*}
(R_{\mathfrak{g}})^{ij}_{kl} = (R_{\mathfrak{g}})^{i+n,j+n}_{k+n,l+n} = r^{ij}_{kl} {\rm \quad and \quad }
(R_{\mathfrak{g}})^{i+n,j}_{i+n,j} = (R_{\mathfrak{g}})^{i,j+n}_{i,j+n} =1
\end{align*} for $i,j=1, \dots, n$
and all other entries equal to $0$.
Note that in the diagonal case, $R_{\mathfrak{g}}$ can be viewed as a block diagonal matrix with diagonal entries $(R, I_{n^2}, I_{n^2}, R)$.  Moreover,  this is the $R$-matrix associated with the $U_q(\mathfrak{g})$ representation $(V_0\otimes \mathbb{C}(q)) \oplus (\mathbb{C}(q)\otimes V_1)$ where $V_i$ is the standard representation associated to the $i^{th}$ copy of $U_q(\mathfrak{gl}_n)$ inside of $U_q(\mathfrak{g})$ and $\mathbb{C}(q)$ is the trivial representation. 

The algebra $\mathscr{P}$ is a quotient of the FRT matrix defined by $R_{\mathfrak{g}}$ with all matrix entries $t_{i,j+n}$ and $t_{i+n,j}$, for $i,j=1, \dots, n$, set to zero.  Using the similarity between the reflection equation in Type AI and the diagonal type, it is straightforward to check that the  matrices $J_{(n)}(a)$ defined by \begin{align*}
{J}_{(n)}(a) = \sum_{k=1}^na_k(e_{k,n+k} + e_{n+k,k})
\end{align*}
satisfies the reflection equations with respect to the matrix  $R_{\mathfrak{g}}$.

We frequently write $J(a)$ for $J_{(n)}(a)$ where the subscript $n$, which equals  the length of the tuple $a$, can be understood from context. Let $J^k$ be the submatrix of $J(a)$ corresponding to the term with coefficient $a_k$ .  
We further write $J_{r,s}^k$ for the $r,s$ entry of $J^k$ viewed as a matrix.  
Set 
  \begin{align}\label{xa-defn1} x_{ij} (a)= \sum_{k=1}^na_k\left(\sum_{r,s}t_{ir}J^k_{r,s}t_{js}\right)
  \end{align}
  for each choice of $i,j$ and each $n$-tuple $(a_1,\dots, a_n)$.  These terms can be written explicitly for each family as 
  \begin{align*}
x_{ij}(a)
= \left\{\begin{matrix} \sum_{k=1}^na_kt_{ik}t_{jk} &{\rm for\ Type \ AI}\cr\cr
\sum_{k=1}^na_k(t_{i,2k-1}t_{j,2k} -qt_{i,2k}t_{j,2k-1}) &{\rm for\ Type \ AII}\cr\cr
\sum_{k=1}^na_kt_{ik}t_{j,n+k} {\rm\ for \ }i\leq n<j &{\rm \  for \ the \ diagonal \ type}\cr
\sum_{k=1}^na_kt_{i,n+k}t_{jk} {\rm\ for \ }i> n\geq j &{\rm \  for \ the \ diagonal \ type}
 \end{matrix}\right.
\end{align*}
These elements can also be expressed in matrix form  as 
  \begin{align*}
  X(a)= TJ(a)T^t
  \end{align*}
  where  $X(a)$  is a  matrix  with entries $x_{ij}(a)$ of size  $n\times n$ in Type AI and $2n\times 2n$ in both type  AII and the diagonal type. 
For the diagonal family, $t_{ik}$ commutes with $t_{j+n,r+n}$ for all $i,k,j,r$.  Hence
 \begin{align*} x_{i,j+n}(a) = \sum_{k=1}^n a_kt_{ik}t_{j+n,k+n} = \sum_{k=1}^na_kt_{j+n,k+n}t_{ik} = x_{j+n,i}(a)
 \end{align*}
for all $i\leq n<j$.

  In the next lemma, we determine the relationship between the $n$-tuple $a$ and the $n$-tuple $b$ so that $x_{ij}(a)$ is invariant with respect to the right action of $\mathcal{B}_{\theta}(b)$.  
  Moving the right action to the left, one could deduce this from \cite{N}, Proposition 2.3 for Types AI and AII.  However, it would still be necessary to translate between the Noumi construction of right coideal subalgebras  based on $J(a)$  to the $\mathcal{B}_{\theta(b)}$ of this paper. By \cite{L1999}, Section 5 (see also  \cite{L2002}, Theorem 7.5), these two families are essentially the same.  However, the matching between parameters $a$ and $b$ has only been made explicit in Type AI (see \cite{N}, Section 2, equation (2.21)). It is easier to determine this matching directly using the action of the generators for $\mathcal{B}_{\theta}(b)$.  This is the approach taken in the proof of the next lemma.



\begin{lemma}\label{lemma:inv-typeAIAII} The elements $x_{ij}(a)$, $1\leq i,j\leq {\rm rank}(\mathfrak{g})$   are right $\mathcal{B}_{\theta}(b)$ invariants of $\mathscr{P}$   if and only if for $u=1, \dots, n -1$, 
\begin{itemize}
\item [(i)]  $b_u= a_{u+1}a_u^{-1}$ in Type AI
\item[(ii)]  $b_{2u}=q^3a_{u+1}a_u^{-1}$    in Type AII.
\item[(iii)] $b_u= qa_{u+1}a_u^{-1}$  in the  diagonal type.
\end{itemize}
\end{lemma}

\begin{proof} 
 To make the notation somewhat easier to read, we suppress the $a$ and write $x_{ij}$ for $x_{ij}(a)$ throughout this proof.  It follows from Lemma \ref{lemma:actions}
that 
 \begin{align}\label{t-eqn}
(t_{ir}\cdot K_{\eta})(t_{js}\cdot K_{\eta}) = q^{\eta_r + \eta_s}t_{ir}t_{js}.\end{align}
for all $i,r,j,s$ and all  $\eta= \sum_{k}\eta_k\epsilon_k$ with $\eta_k\in \mathbb{Z}$.
Hence $x_{ij}\cdot K_{\eta} = \epsilon(K_{\eta})x_{ij} = x_{ij}$ if and only if $\eta_r+ \eta_s=0$ 
for all choices of $r$ and $s$ with $t_{ir}t_{js}$ showing up as a summand of $x_{ij}$.  
Using  the description  of elements of the form $K_{\eta}$ in $\mathcal{B}_{\theta}(b)$ combined with the explicit formulas for the $x_{ij}$, it is straightforward to check that this condition on $\eta$ holds for all  $K_{\eta}\in \mathcal{B}_{\theta}(b)$.  Note that this is independent of the choice of $b_k$ and $a_k$.

We now evaluate the action of the other generators of $\mathcal{B}_{\theta}(b)$ on the $x_{ij}(a)$.  By Lemma \ref{lemma:actions}, we have
\begin{align}\label{t-formula}(t_{ir}t_{js})\cdot F_u
&=  (t_{ir}\cdot  F_u)(t_{js}\cdot K_u^{-1})+t_{ir}(t_{js}\cdot F_u))\cr
&=\delta_{r,u+1}q^{\delta_{s,u+1}-\delta_{su}}t_{iu}t_{js} +\delta_{s,u+1}t_{ir}t_{ju}.
\end{align}
A similar computation yields
\begin{align}\label{t-formula2}(t_{ir}t_{js}) \cdot E_{u}
&=\delta_{ru}t_{i,u+1}t_{js} +\delta_{su}q^{\delta_{ru}-\delta_{r,u+1} }t_{ir}t_{j,u+1}.
\end{align}

 In case (i),  $B_u = F_u-b_uE_uK_u^{-1}$ and $x_{ij}= \sum_ka_kt_{ik}t_{jk}$.  Hence, it follows from (\ref{t-eqn}), (\ref{t-formula}), (\ref{t-formula2}) that 
\begin{align*}x_{ij}\cdot B_u = qa_{u+1}t_{iu}t_{j,u+1} +a_{u+1} t_{i,u+1}t_{ju}-b_ua_ut_{i,u+1}t_{ju}-qb_ua_ut_{iu}t_{j,u+1}.
\end{align*}
 Thus $x_{ij} \cdot {B}_u =\epsilon(B_u)x_{ij} = 0$ if and only if $b_u= a_{u+1}a_u^{-1}$ for each $u$.
For case (iii), $B_u=F_u-b_uE_{u'}K_u^{-1}$ where $u' = u+n$ if $u< n$ and $u' = u-n$ otherwise.  This gives us
\begin{align*}
x_{ij} \cdot {B}_u&= \delta_{u\leq n}\left(a_{u+1}t_{iu}t_{j,u'+1} - q^{-1}b_ua_ut_{iu}t_{j,u'+1}\right)\cr &+
\delta_{u> n}\left(a_{u+1}t_{i,u+1}t_{j,u'}
-q^{-1}b_ua_ut_{i,u+1}t_{j,u'}\right).
\end{align*}
Hence,  $x_{ij}\cdot B_u=0$ if and only if $b_u = qa_{u+1}a_u^{-1}$ for $u=1, \dots, 2n$.

We only need to finish case (ii). Using (\ref{t-formula}), we see that 
\begin{align*} 
(t_{i,2k-1}t_{j,2k}-qt_{i,2k}t_{j,2k-1j})\cdot F_u = \delta_{2k,u+1}\left(t_{iu}t_{ju} -qq^{-1}t_{iu}t_{ju}\right)=0
\end{align*}
for $u$ odd and for all $k$.   Hence $x_{ij}\cdot F_u=0$ for all $u$ odd. A similar argument shows that $x_{ij}\cdot E_u=0$ for all $u$ odd.

It remains to determine conditions for $x_{ij}\cdot B_{2u} = 0$ for $u=1,\dots, n-1$. Recall that  $B_{2u} = F_{2u}-b_{2u}(({\rm ad}\ E_{2u-1}E_{2u+1})E_{2u})K_{2u}^{-1}.$
By (\ref{t-formula}), we have 
\begin{align}\label{Fform}(t_{i,2k-1}t_{j,2k}-qt_{i,2k}t_{j,2k-1})\cdot F_{2u} &=\delta_{2k-1,2u+1}t_{i,2u}t_{j,2u+2}
-q\delta_{2k-1,2u+1}t_{i,2u+2}t_{j,2u}.
\end{align}
A straightforward computation using the formulas for the adjoint action yields
\begin{align}\label{ad-formula} ({\rm ad}\ E_{2u-1}E_{2u+1})E_{2u}&= E_{2u-1}[E_{2u+1},E_{2u}]_{q^{-1}}-q^{-1}[E_{2u+1},E_{2u}]_{q^{-1}}E_{2u-1}\cr
&\in q^{-2} E_{2u}E_{2u+1}E_{2u-1} + E_{2u-1}U_q(\mathfrak{g})+E_{2u+1}U_q(\mathfrak{g}).
\end{align}
We showed earlier that  $x_{ij}\cdot E_{2u\pm 1} = 0$.  Hence,  $$x_{ij}\cdot (({\rm ad}\ E_{2u-1}E_{2u+1})E_{2u})K_{2u}^{-1} = x_{ij}\cdot \left(q^{-2} E_{2u}E_{2u+1}E_{2u-1}K_{2u}^{-1}\right).$$
Note that 
$t_{a,2k-1}\cdot E = 0 $
whenever $E \in E_{2u}U_q(\mathfrak{g}) +  E_{2u-1}U_q(\mathfrak{g}).$  Similarly, $t_{a,2k}\cdot E=0$ for 
all $E \in E_{2u-1}U_q(\mathfrak{g}) +  E_{2u+1}U_q(\mathfrak{g}).$ Hence for all $u=1, \dots, n-1$, we have  \begin{align*}\Delta( E_{2u}E_{2u+1}E_{2u-1}) 
=K_{2u}K_{2u+1}E_{2u-1}\otimes E_{2u}E_{2u+1} + E_{2u}K_{2u-1}E_{2u+1}\otimes E_{2u-1}
+Y
\end{align*}
where $Y=\sum_ry_r\otimes y'_r$ is a term in $U_q(\mathfrak{g})\otimes U_q(\mathfrak{g})$ such that 
\begin{align*}
\sum_r(t_{i,2k-1}\cdot y_r)( t_{j,2k}\cdot y'_r)= 0
{\rm \quad
and \quad }
\sum_r(t_{i,2k}\cdot y_r)( t_{j,2k-1}\cdot y'_r)= 0
\end{align*}
for all $k$.  Hence $x_{ij}\cdot  (({\rm ad}\ E_{2u-1}E_{2u+1})E_{2u})K_{2u}^{-1}$ equals
\begin{align*}
&\ q^{-2}\sum_ka_k(t_{i,2k-1}\cdot K_{2u}K_{2u+1}E_{2u-1}K_{2u}^{-1})( t_{j,2k}\cdot E_{2u}E_{2u+1}K_{2u}^{-1})\cr&-q^{-1}\sum_ka_k(t_{i,2k}\cdot E_{2u}K_{2u-1}E_{2u+1}K_{2u}^{-1})( t_{j,2k-1}\cdot E_{2u-1}K_{2u}^{-1} )
\cr &=( q^{-3}a_ut_{i,2u}t_{j,2u+2} -q^{-2}a_ut_{i,2u+2}t_{j,2u}).
\end{align*}
This combined with (\ref{Fform}) implies that $x_{ij}\cdot B_{2u}$ equals
\begin{align*}
a_{u+1}t_{i,2u}t_{j,2u+2}-qa_{u+1}t_{i,2u+2}t_{j,2u}-( q^{-3}b_{2u}a_ut_{i,2u}t_{j,2u+2} -q^{-2}b_{2u}a_ut_{i,2u+2}t_{j,2u}).
\end{align*}
and so $b_{2u} = q^3a_{u-1}a_u^{-1}$ for all $u=1, \dots, n-1$ as claimed.
\end{proof}

Set $\mathscr{D}=\mathscr{P}^{op}$.   In particular, write $\mathscr{D}= \mathcal{O}_q(\Mat_N)^{op}$ in Type AI (with $N=n$) and Type AII (with $N=2n$) and let $\mathscr{D} = \mathcal{O}_q(\Mat_n)^{op}\otimes \mathcal{O}_q(Mat_n)^{op}$ in the diagonal setting.  Using Lemma \ref{lemma:opposite relation},   $\mathscr{D}$ is given a $U_q(\mathfrak{g})$-bimodule structure for Types AI and AII.
 In the diagonal setting,  for $k=1,2$,  the $k^{th}$ copy of $U_q(\mathfrak{gl}_n)$ acts on the $k^{th}$  copy $\mathcal{O}_q(\Mat_n)^{op}$
as detailed in Lemma \ref{lemma:opposite relation} and trivially on the other  copy.

 For each $n$-tuple $(c_1, \dots, c_n)$, set \begin{align}\label{d-form}d_{ij}(c) = \sum_{k=1}^nc_k\left(\sum_{r,s}\partial_{ir}J^k_{s,r}\partial_{js}\right).
\end{align}  Comparing (\ref{d-form}) to (\ref{xa-defn1}) yields that $(x_{ji}({c}))^* = d_{ij}(c)$ for all $i,j$.
Recall that $f^*\cdot (S^{-1}( g))^{\natural} = (f\cdot g)^*$ for all $f\in \mathscr{P}$ and $g\in U_q(\mathfrak{g})$ (as in the discussion at the end of Section \ref{section:FRT}).  The next result is an analog of Lemma \ref{lemma:inv-typeAIAII} with the $t_{ij}$ replaced by the $\partial_{ij}$.

\begin{lemma} \label{lemma:diffmodule} The elements $d_{ij}(c)$, $1\leq i,j\leq {\rm rank}(\mathfrak{g})$ are right $\mathcal{B}_{\theta}(b)$ invariants if and only if for $u=1, \dots, n-1$ we have 
\begin{itemize}
\item[(i)] $ b_u = q^{-2}c_{u+1}^{-1}c_u$ in Type AI
\item[(ii)] $b_{2u}= q^{-1}c_{u+1}^{-1}c_u $ in Type AII
\item[(iii)] $ b_u = q^{-1}c_{u+1}^{-1}c_u$ for the diagonal type.
\end{itemize}
\end{lemma}
\begin{proof} Given $K_{\eta}$ in $\mathcal{B}_{\theta}$, we have $(S^{-1}(K_{\eta}))^{\natural}= K_{\eta}^{-1}$.  Hence by Lemma \ref{lemma:inv-typeAIAII} and its proof, 
\begin{align*}d_{ij}(c)\cdot K_{\eta} = ((x_{ji}( c))\cdot K_{\eta}^{-1}) ^{\natural} = (x_{ji}( c))^* = d_{ij}(c)\end{align*} 
for all $K_{\eta}\in \mathcal{B}_{\theta}$ independent of the relationship between the $n$-tuples $b$ and $c$.

It follows from (\ref{starinversestructure}) that  \begin{align*}S^{-1}(F_u-q^{2\delta_{u',u}-1}  b_{u'}^{-1}E_{u'}K_u^{-1})^{\natural} &= -q^{-1} E_u +q^{2\delta_{u',u}} b_{u'}^{-1}F_{u'} K_u\cr &=-q^{2\delta_{u',u}}  b_{u'}^{-1}(F_{u'}-q^{-2\delta_{u',u}} b_{u'}E_{u}K_{u'}^{-1}(K_{u'}K_{u}^{-1}))K_u\cr
&=-q^{2\delta_{u',u}} b_{u'}^{-1}(F_{u'}-b_{u'}(K_{u'}K_{u}^{-1})E_{u}K_{u'}^{-1})K_u.\end{align*} Now consider Type AI with  $u'=u$ and the diagonal type with $u'=u+n$ or $u'=u-n$.  In both cases, we have
$K_{u'}K_{u}^{-1}\in \mathcal{B}_{\theta}$ and so by the previous paragraph, $d_{ij}(c)\cdot K_{u'}K_{u}^{-1}=1$ for all $n$-tuples $c$.
Hence, the above computation shows that
\begin{align*} d_{ij}(c)\cdot (F_{u'}-b_{u'}E_uK_u^{-1}) = 0{\rm \ if \ and\ only\ if \ }x_{ji}( c)\cdot (F_u-q^{2\delta_{u',u}}  b_{u'}^{-1}E_{u'}K_u^{-1})=0.
\end{align*}
 By Lemma \ref{lemma:inv-typeAIAII} (i) , the latter equality holds in Type AI provided $q^{-2} b_{u}^{-1} = c_{u+1}c_{u}^{-1}$. This proves (i).   For the diagonal type, the latter equality holds provided 
 $ b_{u}^{-1} = qc_{u+1}c_{u}^{-1}$ which is equivalent to (iii).

For Type AII, it follows from (\ref{starinversestructure}) that  $x_{ji}(c)\cdot E_{2u-1} =0$ implies 
$d_{ij}(c)\cdot F_{2u-1}=0$ for $u=1, \dots, n$.  The same holds for the roles of $E_{2u-1}$ and $F_{2u-1}$ interchanged.  Hence, we only need to analyze the action terms of the form $B_{2u}= F_{2u} -b_{2u}(({\rm ad}\ E_{2u-1}E_{2u+1})E_{2u})K_{2u}^{-1}$.  Using (\ref{t-formula2}), we get
\begin{align*} d_{ij}(c)\cdot F_{2u} &= -q^{-1}((x_{ji}( c)\cdot E_{2u})^{\natural} \cr &= -q^{-1}(\sum_k c_k(t_{j,2k-1}t_{i,2k} - qt_{j,2k}t_{i,2k-1})\cdot E_{2u})^{\natural} \cr
&=-q^{-1} c_u(t_{j,2u-1}t_{i,2u+1} -qt_{j,2u+1}t_{i,2u-1})^*
\end{align*}
Arguing as in the proof of Lemma \ref{lemma:inv-typeAIAII}, we have \begin{align*} &d_{ij}(c)\cdot (({\rm ad}\ E_{2u-1}E_{2u+1})E_{2u})K_{2u}^{-1}  = d_{ij}(c)\cdot q^{-2}(E_{2u}E_{2u+1}E_{2u-1})K_{2u}^{-1}
\cr &=q^{-2}\sum_k c_k(\partial_{i,2k}\cdot K_{2u}K_{2u-1}E_{2u+1}K_{2u}^{-1})(\partial_{j,2k-1}\cdot E_{2u}E_{2u-1}K_{2u}^{-1})
\cr&-q^{-1}\sum_k c_k(\partial_{i,2k-1}\cdot E_{2u}K_{2u+1}E_{2u-1}K_{2u}^{-1})(\partial_{j,2k}\cdot E_{2u+1}K_{2u}^{-1})
\cr&=c_{u+1}\partial_{i,2u+1}\partial_{j,2u-1}-qc_{u+1}\partial_{i,2u-1}\partial_{j,2u+1}
\end{align*}
Hence 
\begin{align*}
d_{ij}(c) \cdot B_{2u} &= -q^{-1}c_u(\partial_{i,2u+1}\partial_{j,2u-1}-q\partial_{i,2u-1}\partial_{j,2u+1})\cr&-b_{2u}c_{u+1}(\partial_{i,2u+1}\partial_{j,2u-1}-q\partial_{i,2u-1}\partial_{j,2u+1})
\end{align*}
Thus $d_{ij}(c)\cdot B_{2u}=0$ when $b_{2u} = q^{-1}c_uc_{u+1}^{-1}$ which proves (ii).
\end{proof}

Using Lemmas \ref{lemma:actions} and  \ref{lemma:diffmodule}, we can choose parameters $a$ and $c$ so that they are compatible with the parameters $b$ at the same time.  In other words, assume that the $n$-tuple $a$ has been chosen using the conditions of Lemma \ref{lemma:actions} so that $x_{ij}(a)$ is right invariant with respect to the action of $\mathcal{B}_{\theta}(b)$.  Then setting  $c_u=q^{-2u}a_u^{-1}$ in Type AI, $c_u = q^{-2u}a_u^{-1}$ in the diagonal setting and $c_u = q^{-4u+2}a_u^{-1}$ in Type AII for all $u=1, \dots, n$ yields  right invariant elements $d_{ij}(c)$. 
Write this choice of $c$ as $a'$. 
Expanding these elements out as we did for the $x_{ij}(a)$ with respect to the $n$-tuple $a$ gives us 
\begin{align*}
d_{ij}(a') 
= \left\{\begin{matrix} \sum_{k=1}^nq^{-2k}a_k^{-1}\partial_{ik}\partial_{jk} &{\rm for\ Type \ AI}\cr\cr
\sum_{k=1}^nq^{-4k+2}a_k^{-1}(\partial_{i,2k}\partial_{j,2k-1}-q\partial_{i,2k-1}\partial_{j,2k} ) &{\rm for \ Type \ AII}\cr\cr
\sum_{k=1}^nq^{-2k}a_k^{-1}\partial_{ik}\partial_{j,n+k}{\rm \ for \ }i\leq n<j &{\rm for  \ diagonal \ type}\cr
\sum_{k=1}^nq^{-2k}a_k^{-1}\partial_{i,n+k}\partial_{jk}{\rm \ for \ }i>n\geq j &{\rm for \ diagonal \ type}
 \end{matrix}\right.
\end{align*}

For the remainder of the paper, we make a choice for $a$ and $b$ parameters in order to make the arguments easier. These parameters can be easily converted into another set using Hopf algebra automorphisms of $U_q(\mathfrak{g})$.   In particular, we choose $a_1=\cdots = a_n=1$ for each of the three cases.  Note that this means for $u=1, \dots, n$ we have $b_u=1$   in Type AI, $b_{2u} = q^3$ in Type AII,  and  $b_u=q$  in the diagonal setting.  
Set $x_{ij} = x_{ij}(1,\dots, 1)$ and $d_{ij}=d_{ij}((1,\dots, 1)')$.  Note that we can write both $x_{ij}$ where the summand no longer depends on $k$, just 
$r$ and $s$ where $J_{r,s}$ is the coefficient of $e_{rs}$ in the matrix $J$.  Indeed, it is easy to see directly from (\ref{xa-defn1}) and 
the choice of $a$ that 
\begin{align*}
x_{ij} =\sum_{r,s}t_{ir}J_{r,s}t_{js}.
\end{align*}
One further checks from the formulas for $d_{ij}(a')$ above that for the choice of $a'$ that 
\begin{align*}
d_{ij} =\sum_{r,s} q^{-2\hat{s}} \partial_{ir}J_{r,s}\partial_{js}
\end{align*}
where $\hat{s}= s$ in Types AI and AII and for the diagonal type $\hat{s}= s$ for $s\leq  n$, and $\hat{s} = s - n$ for $s\geq n+1$.  Note that we are expressing $d_{ij}$ in terms of the matrix $J$ and not its transpose as above.  This makes no difference in Types AI and the diagonal type since in both these cases, $J=J^t$.   On the other hand, for Type AII, we have 
\begin{align*}q^{-4k+2}\partial_{i,2k-1}\partial_{j,2k}-q^{-4k}q\partial_{i,2k}\partial_{j,2k-1}  = -q^{-1}q^{-4k+2}\left(\partial_{i,2k}\partial_{j,2k-1} -q\partial_{i,2k-1}\partial_{j,2k} \right).
\end{align*}
In other words, this expression for $d_{ij}$ differs from the earlier one by a scalar multiple, namely $-q^{-1}$.  Thus, without loss of generality, we may use this new formula based on $J$ instead 
of the earlier one using $J^t$.

\subsection{Module Realization}
Let $\mathscr{P}_{\theta}$ be the subalgebra of $\mathscr{P}$ generated by the $x_{ij}$, for $1\leq i,j\leq {\rm rank({\mathfrak{g}}})$ where here we are setting $x_{ij}=0$ in the diagonal case for all $i,j$ satisfying  $i-j\leq n$. Similarly, let $\mathscr{D}_{\theta}$ be the subalgebra of 
$\mathscr{D}$  generated by the $d_{ij}$ with the same restriction as above. Both of these algebras are generated by   right $\mathcal{B}_{\theta}$ invariant elements.  It follows from the next lemma  that the entire algebras $\mathscr{P}_{\theta}$ and $\mathscr{D}_{\theta}$ are right invariant with respect to the action of $\mathcal{B}_{\theta}$. It should be noted that a priori this result is not obvious since $\mathcal{B}_{\theta}$ is not a bialgebra.

\begin{lemma}\label{lemma:B-invariance} Let $u$ and $v$ be elements of a $U_q(\mathfrak{g})$-bimodule $A$ such that both are invariant with respect to the right action of $\mathcal{B}_{\theta}$.  Then the same is true for the product $uv$.
\end{lemma}
\begin{proof} 
Let $\mathcal{M}^+$ denote the algebra  generated by those $E_i$  for which $E_i\in \mathcal{B}_{\theta}$ and write $\mathcal{T}_{\theta}$ for the group-like elements in $\mathcal{B}_{\theta}\cap U_q(\mathfrak{g})$. 
It is straightforward to see from their description in Section \ref{section:three-families-defn} that the elements  $B_r$  can be written in the form $F_r + (({\rm ad}\ Z^+_r)E_{p(r)})K_r^{-1}$ for some choice of $Z^+_r \in \mathcal{M}^+$.  It follows from the definition of the comultiplication $\Delta$ for $U_q(\mathfrak{g})$ that 
\begin{align*} \Delta(B_r) &= \Delta(F_r + [({\rm ad} Z_r^+)E_{p(r)})] K_r^{-1})\cr &\in B_r\otimes 1 + 1\otimes B_r + (K_{p(r)}K_r^{-1}-1)\otimes  [({\rm ad} Z_r^+)E_{p(r)})] K_r^{-1}) + \mathcal{M}_+\otimes U_q(\mathfrak{g}) + U_q(\mathfrak{g})\otimes \mathcal{M}_+
\end{align*}
where $\mathcal{M}_+$ denotes the augmentation ideal of $\mathcal{M}$.  
Hence 
\begin{align}\label{relnBr1}\Delta(B_r) \in (\mathcal{B}_{\theta})_+\otimes U_q(\mathfrak{g}) + U_q(\mathfrak{g})\otimes (\mathcal{B}_{\theta})_+
\end{align}
for each $B_r\in \mathcal{B}_{\theta}.$
On the other hand, given $K\in \mathcal{T}_{\theta}$, we have 
\begin{align*}\Delta(K) = K\otimes K = (K-1)\otimes K +1\otimes (K-1) +1\otimes 1.
\end{align*}
It  follows that 
\begin{align}\label{relnK2}
\Delta(K) \in 1\otimes 1 + (\mathcal{B}_{\theta})_+\otimes U_q(\mathfrak{g}) + U_q(\mathfrak{g})\otimes (\mathcal{B}_{\theta})_+
\end{align}
for all $K\in  \mathcal{T}_{\theta}$.  Recall that $\epsilon(K) = 1$ and $\epsilon(B_r)=0$ for all choices of $K\in \mathcal{T}_{\theta}$ and all $r$. Therefore by (\ref{relnBr1}) and (\ref{relnK2}),
$uv\cdot b \in \epsilon(b)uv+  (u\cdot (\mathcal{B}_{\theta})_+)A + A(v\cdot  (\mathcal{B}_{\theta})_+) =\epsilon(b)uv$
for the generators of $\mathcal{B}_{\theta}$ as given in Section \ref{section:three-families-defn}. Thus $uv$ is right invariant with respect to  the action of $\mathcal{B}_{\theta}$.  
\end{proof}

The next lemma shows that the the subspace   of $\mathscr{P}_{\theta}$ spanned by the $x_{ij}$ is a left $U_q(\mathfrak{g})$-submodule of $\mathscr{P}$.  Thus  $\mathscr{P}_{\theta}$ inherits the structure of a   left $U_q(\mathfrak{g})$-module algebra from $\mathscr{P}$.   The analogous assertion holds for $\mathscr{D}_{\theta}$ and its subspace spanned by the $d_{ij}$. To make the notation easier, we set $x_{ij} = 0=d_{ij}$ whenever either $i$ or $j$ is in the set $\{0,{\rm rank}(\mathfrak{g})\}$.

\begin{lemma}\label{lemma:left_action}
The (left) action of $U_q(\mathfrak{g})$ on the subalgebra $\mathscr{P}_{\theta}$ is defined by 
the formulas 
\begin{align*}E_r\cdot x_{ij} &= \delta_{i-1,r}x_{i-1,j} + \delta_{j-1,r}q^{\delta_{ri}-\delta_{r,i-1}}x_{i,j-1}\cr
F_r\cdot x_{ij} &= \delta_{ir}q^{-\delta_{rj} + \delta_{r,j-1}}x_{i+1,j} + \delta_{jr}x_{i,j+1}\cr
K_{\epsilon_s}\cdot x_{ij} &=q^{\delta_{is}+\delta_{js}}x_{ij}
\end{align*}
and  the (left) action of $U_q(\mathfrak{g})$ on the subalgebra $\mathscr{D}_{\theta}$ is defined by 
the formulas 
\begin{align*}F_r\cdot d_{ij} &= -(q^{1+\delta_{rj}-\delta_{r,j-1}}\delta_{i-1,r}d_{i-1,j} -q^2 \delta_{j-1,r}d_{i,j-1})\cr
E_r\cdot d_{ij} &=-(q^{-1}\delta_{ir}d_{i+1,j} + q^{1-\delta_{ri} + \delta_{r,i-1}}\delta_{jr}d_{i,j+1})\cr
K_{\epsilon_s}\cdot d_{ij} &=q^{-\delta_{is}-\delta_{js}}d_{ij}
\end{align*}
for all $i,j,s \in \{1,\dots, {\rm rank(\mathfrak{g})}\}$, $1\leq r\leq n-1$ in Type AI,  $1\leq r\leq 2n-1$ in Type AII, $1\leq r\leq n-1$ and $n+1\leq r\leq 2n-1$ in the diagonal case, and either $ i\leq n<j$ or $ j\leq n<i$ in the diagonal case.  
\end{lemma}
\begin{proof} It follows from  the formulas for the action on the $t_{ij}$ as given in Lemma \ref{lemma:actions} that
\begin{align*} E_r\cdot (t_{ia}t_{jb})&=  (E_r\cdot t_{ia}) t_{jb} + (K_r\cdot t_{ia})(E_r\cdot (t_{jb}) \cr
&=\delta_{i-1,r}t_{i-1,a}t_{jb} + \delta_{j-1,r}q^{\delta_{ir}-\delta_{i-1,r}}t_{ia}t_{j-1,b}
\end{align*}
\begin{align*}
F_r\cdot (t_{ia}t_{jb} )&=  (F_r\cdot t_{ia}) (K_r^{-1}\cdot t_{jb}) + t_{i+1,a}(F_r\cdot (t_{jb}) \cr
&=\delta_{ir}q^{-\delta_{jr}+\delta_{j-1,r}}t_{i+1,a}t_{jb} + \delta_{jr}t_{ia}t_{j+1,b}
\end{align*}
and 
\begin{align*}
K_{\epsilon_s}\cdot (t_{ia}t_{jb} )=  (K_{\epsilon_s}\cdot t_{ia}) (K_{\epsilon_s}\cdot t_{jb}) =q^{\delta_{is}+\delta_{js}}t_{ia}t_{jb}
\end{align*}
for all choices of $a$ and $b$.
Note that $x_{ij}$ is a sum of terms of the form $t_{ia}t_{jb}$. Hence the action of   $E_r$ on 
the $x_{ij}$ and the $F_r,$, $K_{\epsilon_s}$ on the $x_{ij}$ follow directly from the above formulas.  The argument that checks the action of these elements on the $d_{ij}$ is similar.
\end{proof}

It is straightforward to check from the above lemma that $\sum_{i,j}\mathbb{C}(q)x_{ij}$ forms a simple module for the action of $U_q(\mathfrak{g})$.  In Type AI, this module is isomorphic to  $L(2\epsilon_1)$ with highest weight generating vector $x_{11}$.  In Type AII, this module is isomorphic to $L(\epsilon_1+\epsilon_2)$ with highest weight generating vector $x_{12}$.  For diagonal type, this module is isomorphic to  $L(\epsilon_1+ \epsilon_{n+1})$ with highest weight generating vector $x_{1,n+1}$.  Similarly,  $\sum_{i,j}\mathbb{C}(q)d_{ij}$ is a simple module generated by the lowest weight vector $d_{1r}$ of weight $-\epsilon_1-\epsilon_r$ where $r=1$ in Type AI, $r=2$ in Type AII, and $r=n+1$ in  diagonal type.

\subsection{Algebra Structure}\label{section:algebrastructure}
The next result establishes relations satisfied by the $x_{ij}$.  For Type AI and AII this is just \cite{N}, Proposition 4.4 with a slight change in notation.  
 
 \begin{proposition}\label{proposition:relations_overview}  The   elements $x_{ij}$, $i,j=1, \dots, {\rm rank}(\mathfrak{g})$ satisfy   the following linear relations
\begin{itemize}
\item[(i)] ${x}_{ij} = q{x}_{ji}$ all $1\leq i<j\leq n$ in type AI
\item[(ii)] ${x}_{ij}=-q^{-1}{x}_{ji}$ for all $1\leq i<j\leq 2n$ and $x_{ii} = 0$ for all $i$ in type AII.
\item[(iii)] $x_{i,j+n} = x_{j+n,i}$ and $x_{ij}=x_{i+n,j+n}=0$ for all $i,j=1,\dots, n$ in the diagonal type
\end{itemize} and the quadratic  relations expressed in matrix form by   
\begin{align}\label{quad-reln}R_{\mathfrak{g}}{X}_1R_{\mathfrak{g}}^{t_1}{X}_2 ={X}_2R_{\mathfrak{g}}^{t_1}{X}_1R_{\mathfrak{g}}\end{align}
where $X$ is the ${\rm rank}(\mathfrak{g})\times {\rm rank}(\mathfrak{g})$ matrix with $ij$ entry $x_{ij}$.\end{proposition}
\begin{proof}By \cite{N}, Proposition 4.4, the  $x_{ij}$ satisfy the relations  obtained from (\ref{quad-reln}) in type AI and in type AII.  Here we are using the fact that $R'_{21}$ in the notation of \cite{N} is the same as $R_{12}$ in the notation of this paper. Moreover, since the diagonal type is also defined using a solution to the reflection equations, the proof in \cite{N} carries over to this family as well.   

Relations (i), (ii) are also part of \cite{N} Proposition 4.4; both can be deduced in a straightforward manner from the relations for the $t_{ij}$.  Relations in (iii) follow from  the fact that  $t_{i+n,j+n}$ commutes with $t_{kl}$  and  $t_{i,j+n} = t_{j+n,i}=0$ for all $i,j,k,l$ in $\{1,\dots, n\}$.  
\end{proof}

Let us take a closer look at relation (\ref{quad-reln}).   Recall that $X_1 = X\otimes I$.   Hence  $(X_1)^{ab}_{cd} = x_{ac}$ if $b=d$ and is $0$ otherwise. 
Similarly, $(X_2)^{ab}_{cd} = x_{bd}$ if $a=c$ and $0$ otherwise.
The $ld,hk$ entry of the left hand side of (\ref{quad-reln}) is  
\begin{align*}
(R_{\mathfrak{g}}{X}_1R_{\mathfrak{g}}^{t_1}{X}_2)^{ld}_{hk} = \sum_{j,s,a,b,u,v}(R_{\mathfrak{g}})^{ld}_{js}(X_1)^{js}_{ab}(R_{\mathfrak{g}})^{ub}_{av}(X_2)^{uv}_{hk}
 = \sum_{j,s,a,v}(R_{\mathfrak{g}})^{ld}_{js}x_{ja}(R_{\mathfrak{g}})^{hs}_{av}x_{vk}
\end{align*}
The $ld, hk$ entry of the right hand side of (\ref{quad-reln}) can be evaluated in the same way yielding 
\begin{align*}
({X}_2R_{\mathfrak{g}}^{t_1} {X}_1R_{\mathfrak{g}})^{ld}_{hk} =\sum_{j,s,a,b,u,v} (X_2)^{ld}_{js}(R_{\mathfrak{g}})^{as}_{jb}(X_1)^{ab}_{uv}(R_{\mathfrak{g}})^{uv}_{hk}=\sum_{s,a,b,u} x_{ds}(R_{\mathfrak{g}})^{as}_{lb}x_{au}(R_{\mathfrak{g}})^{ub}_{hk}
\end{align*}
Hence 
\begin{align}\label{full-relation1}
\sum_{j,s,a,v}(R_{\mathfrak{g}})^{ld}_{js}x_{ja}(R_{\mathfrak{g}})^{hs}_{av}x_{vk}=\sum_{s,a,b,u} x_{ds}(R_{\mathfrak{g}})^{as}_{lb}x_{au}(R_{\mathfrak{g}})^{ub}_{hk}
\end{align}
for all $l,d,h,k$. 

\begin{lemma} For diagonal type, relation (\ref{quad-reln}) is equivalent to the set of matrix relations $R\hat{X}_1 \hat{X}_2
=\hat{X}_2\hat{X}_1 R$ where $R$ is the matrix defined by (\ref{Rmatrix2}) and $\hat X$ is the $n\times n$ matrix with $ij$ entry $x_{i,j+n}$ for $i,j=1,\dots, n$.
\end{lemma}
\begin{proof}The matrix of relations defined by (\ref{quad-reln}) can be broken into four submatrices of relations consisting of all $l,d,h,k$ entries where $l$ and $d$ are each chosen to lie in either $\{1, \dots, n\}$ or $\{n+1,\dots, 2n\}$.  These submatrices turn out to correspond to either $R\hat{X}_1 \hat{X}_2
=\hat{X}_2\hat{X}_1 R$ or $R\hat{X}^t_1 \hat{X}^t_2
=\hat{X}^t_2\hat{X}^t_1 R$.   As explained in Section \ref{section:as}, these two relations are equivalent.  

We explain how to go from the submatrix consisting of those relations corresponding to the  $l,d,h,k$ entries of (\ref{quad-reln}) under the assumption $1\leq l,d\leq n$ to the relation $R\hat{X}_1 \hat{X}_2
=\hat{X}_2\hat{X}_1 R$.  The proof for the other submatrices is similar. 
Since $R_{\mathfrak{g}}$ is the block diagonal matrix $(R, I_{n^2\times n^2}, I_{n^2\times n^2}, R)$, we have $(R_{\mathfrak{g}})_{js}^{ld}\neq 0$
 implies that $\{j,s\} = \{l,d\}$, and, moreover, in this case $(R_{\mathfrak{g}})_{js}^{ld}=r_{js}^{ld}$.  It follows that all choices of $j$ and $s$ appearing in 
the left hand side of (\ref{full-relation1}) must satisfy $1\leq j,s\leq n$.  Moreover, $x_{ja}\neq 0$ implies that $n+1\leq a\leq 2n$.
This condition on $a$ combined with the fact that $1\leq s\leq n$ ensures that $(R_{\mathfrak{g}})^{hs}_{av}\neq 0$
if and only if $a=h$ and $v=s$, and, moreover, when this happens, we get $(R_{\mathfrak{g}})^{hs}_{av}=1$. This also implies the only values of $h$ and $v$ for which the left hand side is nonzero must  satisfy 
$n+1\leq h\leq 2n$ and $1\leq v\leq n$. Thus $x_{vk}\neq 0$ ensures that $n+1\leq k\leq 2n$.  Hence the original assumptions on $l$ and $d$, namely $1\leq l,d\leq n$, guarantee that the left hand side of (\ref{full-relation1}) is nonzero only if $n+1 \leq h,k\leq 2n$. 

The same type of analysis shows that we only get nonzero terms showing up in (\ref{full-relation1}) provided one of the following conditions hold:
\begin{itemize}
\item $1\leq l,d\leq n$ and $n+1 \leq h,k\leq 2n$
\item $1\leq l,k\leq n$ and $n+1 \leq h,d\leq 2n$
\item $1\leq h,d \leq n$ and $n+1 \leq l,k\leq 2n$
\item $1\leq h,k \leq n$ and $n+1 \leq l,d\leq 2n$
\end{itemize}

Lets return to the first case, namely $1\leq l,d\leq n$ and $n+1 \leq h,k\leq 2n$.  We showed above that  $(R_{\mathfrak{g}})^{hs}_{av}=1$ if $a=h$ and $v=s$ and $0$ otherwise. A similar argument yields  $(R_{\mathfrak{g}})_{as}^{lb}=1$
provided $a=l$, $s=b$ and is $0$ otherwise. On the other hand, again as explained above, $(R_{\mathfrak{g}})_{js}^{ld}=r_{js}^{ld}$ and a similar argument gives us $(R_{\mathfrak{g}})^{us}_{hk}=r^{us}_{hk}$ where here we are using the fact that $s=b$..
 Thus (\ref{full-relation1})  is \begin{align*}\sum_{j,s}r^{ld}_{js}x_{jh}x_{sk}=\sum_{s,u} x_{ds}x_{lu}r^{us}_{hk}
\end{align*}
  Hence the set of all such equalities together under these assumptions gives us $R\hat{X}_1 \hat{X}_2
=\hat{X}_2\hat{X}_1 R$. The arguments for the other cases are similar.
\end{proof}

We now turn our attention to Types AI and AII and so $R_{\mathfrak{g}}$ is just the $R$-matrix $R$ as in (\ref{Rmatrix2}) with entries $r^{ij}_{kl}$.  Recall that  $r^{ij}_{ij} = q^{\delta_{ij}}$ for all $i,j$, $r^{ij}_{ji} = (q-q^{-1})$ for $j<i$ and all other $r^{ab}_{cd}$ equal $0$. 
Hence, we can rewrite  (\ref{full-relation1}) as 
 \begin{align*}
&q^{\delta_{ld}+\delta_{hd}}x_{lh}x_{dk} +(q-q^{-1})\left( q^{\delta_{hl}}\delta_{d<l}x_{dh} x_{lk}+ q^{\delta_{ld}}\delta_{d<h}x_{ld}x_{hk} + \delta_{l<h}\delta_{d<l}(q-q^{-1}) x_{dl} x_{hk}\right)\cr
&=q^{\delta_{hk}+\delta_{lk}}x_{dk}x_{lh}+
(q-q^{-1})\left(q^{\delta_{hk}}\delta_{l<k}x_{dl}x_{kh}+q^{\delta_{lh}}\delta_{h<k}x_{dh}x_{lk}+\delta_{l<h}\delta_{h<k}(q-q^{-1})x_{dl}x_{hk}\right)
\end{align*}
Moving all but the first term of the left hand side to the right hand side, this relation can be rewritten as 
\begin{align}\label{full-relation} q^{\delta_{ld}+\delta_{hd}}x_{lh}x_{dk} &= q^{\delta_{hk}+\delta_{lk}}x_{dk}x_{lh} + (\delta_{h<k}-\delta_{d<l})q^{\delta_{lh}}(q-q^{-1}) x_{dh} x_{lk}\cr
&+\delta_{l<h}(\delta_{h<k}-\delta_{d<l})(q-q^{-1})^2x_{dl}x_{hk}+q^{\delta_{hk}}\delta_{l<k}(q-q^{-1})x_{dl}x_{kh}\cr&- q^{\delta_{ld}}\delta_{d<h}(q-q^{-1})x_{ld}x_{hk}.
\end{align}

In \cite{N}, the algebra generated by the $x_{ij}$ in Type AI  (resp. Type AII) is referred to as a quantized function algebra on the space of symmetric matrices (resp. skew symmetric matrices).   The next lemma gives explicit relations for these Type AI and Type AII quantum functions algebras.  It should be noted,  that in Type AI, the algebra generated by the $x_{ij}$ is the same as the quantum analog of the function algebra on symmetric matrices  studied  in \cite{Kamita} (see also \cite{B2}, Section 3).  


\begin{lemma}\label{lemma:explicit-relations} In Type AI, we have 
\begin{align*}
x_{lh}x_{dk}=\left\{\begin{matrix}q^{-2}x_{dk}x_{lh}  &{\rm \ if \ }d=l=k<h\cr
q^{-1}x_{dk}x_{lh}&{\rm\ if \ }d=l<k<h\cr
q^{-2}x_{dk}x_{lh}&{\rm\ if \ }d<l=k=h\cr
q^{-1}x_{dk}x_{lh}&{\rm\ if \ }d<l<k=h\cr
q^{-1}x_{dk}x_{lh}  -(q-q^{-1}) x_{dh} x_{lk} &{\rm \ if \ }d<l=k<h\cr
x_{dk}x_{lh}-(q-q^{-1})x_{dh}x_{lk} &{\rm \ if \ }d<l<k<h\cr
x_{dk}x_{lh} -q^{-1}(q^2-q^{-2}) x_{dh} x_{kl}&{\rm \ if \ }d=k<l=h\cr
x_{dk}x_{lh} -(q^2-q^{-2}) x_{dh} x_{kl}&{\rm \ if \ }d=k<l<h\cr
x_{dk}x_{lh} -q^{-1}(q^2-q^{-2}) x_{dh} x_{kl}&{\rm \ if \ }d<k<l=h\cr
x_{dk}x_{lh} -(q-q^{-1}) (q^{-1}x_{dh} x_{kl}+x_{dl}x_{kh})&{\rm \ if \ }d<k<l<h\cr
 x_{dk}x_{lh}&{\rm \ if \ }d<l<h<k\cr
x_{dk}x_{lh}&{\rm \ if \ }d<l=h<k\cr
\end{matrix}\right.
\end{align*}
and for Type AII, we have 
\begin{align*}
x_{lh}x_{dk}=\left\{\begin{matrix}
q^{-1}x_{dk}x_{lh}&{\rm\ if \ }d=l<k<h\cr
q^{-1}x_{dk}x_{lh}&{\rm\ if \ }d<l<k=h\cr
q^{-1}x_{dk}x_{lh} &{\rm \ if \ }d<l=k<h\cr
x_{dk}x_{lh}-(q-q^{-1})x_{dh}x_{lk} &{\rm \ if \ }d<l<k<h\cr
x_{dk}x_{lh} +(q-q^{-1}) (qx_{dh} x_{kl}-x_{dl}x_{kh}).&{\rm \ if \ }d<k<l<h\cr
x_{dk}x_{lh}&{\rm \ if \ }d<l<h<k\cr
\end{matrix}\right.
\end{align*}
\end{lemma}
\begin{proof} We prove the lemma in  Type AII. The proof in Type AI is similar, albeit somewhat more computationally involved.   

Since $x_{ii}=0$ for all $i$ in Type AII, we must have $l<h$ and $d<k$ in all cases.  
Note also that $x_{ji} = -qx_{ij}$ for all $i<j$.
Assume in addition that $d\leq l\leq k\leq h$. Formula (\ref{full-relation}) simplifies to 
 \begin{align} \label{TypeAIIreln} q^{\delta_{ld}}x_{lh}x_{dk} &= q^{\delta_{hk}+\delta_{lk}}x_{dk}x_{lh} -\delta_{d<l}(q-q^{-1})x_{dh} x_{lk}
-\delta_{d<l}(q-q^{-1})^2x_{dl}x_{hk}\cr&+\delta_{l<k}(q-q^{-1})x_{dl}x_{kh}- \delta_{d<h}(q-q^{-1})x_{ld}x_{hk}.
\end{align} 
If $d=l<k<h$, this further simplifies to
$qx_{lh}x_{dk} = x_{dk}x_{lh} $.  If $d<l<k=h$,  (\ref{TypeAIIreln}) becomes 
 \begin{align*} x_{lh}x_{dk} &= qx_{dk}x_{lh} -(q-q^{-1})x_{dh} x_{lk} = q^{-1}x_{dk}x_{lh}\end{align*} 
 which yields the same equality as in the previous case. For $d<l=k<h$, (\ref{TypeAIIreln}) reduces to
  \begin{align*}x_{lh}x_{dk} &= qx_{dk}x_{lh} 
+q(q-q^{-1})^2x_{dl}x_{kh}-q^2 (q-q^{-1})x_{dl}x_{kh}
\cr &=  qx_{dk}x_{lh} 
-(q-q^{-1})x_{dl}x_{kh}=q^{-1}x_{dk}x_{lh}
\end{align*} 
with the last equality following because $l=k$, thus verifying this case.
Now assume $d<l<k<h$.   Note that $q(q-q^{-1})^2 + (q-q^{-1})-q^2(q-q^{-1}) = 0.$ Hence, under this assumption, (\ref{TypeAIIreln}) is equivalent to
 \begin{align*} x_{lh}x_{dk} &= x_{dk}x_{lh} -(q-q^{-1})x_{dh} x_{lk}
+q(q-q^{-1})^2x_{dl}x_{kh}\cr&+(q-q^{-1})x_{dl}x_{kh}- q^2(q-q^{-1})x_{dl}x_{kh}
\cr&=x_{dk}x_{lh} -(q-q^{-1})x_{dh} x_{lk}.
\end{align*} 
which is the fourth entry of the relations for Type AII.

Now consider $d<k<l<h$.  For this case (\ref{full-relation}) can be rewritten as 
\begin{align*} x_{lh}x_{dk} &= x_{dk}x_{lh} +q(q-q^{-1}) x_{dh} x_{kl}
+q(q-q^{-1})^2x_{dl}x_{kh}- q^2(q-q^{-1})x_{dl}x_{kh}\cr
&=x_{dk}x_{lh} +q(q-q^{-1}) x_{dh} x_{kl}- (q-q^{-1})x_{dl}x_{kh}.
\end{align*}This takes care of the penultimate case.  For the final case, 
assume that $d<l<h<k$.   We simplify (\ref{full-relation}) with respect to these assumptions, yielding
\begin{align*}x_{lh}x_{dk} = x_{dk}x_{lh} + 
(q-q^{-1})x_{dl}x_{kh}- (q-q^{-1})x_{ld}x_{hk}.
\end{align*}
Using the facts tht $x_{kh} = -qx_{hk}$ and $x_{ld} = -qx_{dl}$, we see that the last two terms vanish which 
leaves us with $x_{lh}x_{dk} = x_{dk}x_{lh}$ as claimed.
\end{proof}

Define  subalgebras of $\mathscr{P}_{\theta}$ by
 $P_{(r-1)n+s} = \mathbb{C}(q)[x_{11}, x_{12}, \dots x_{rs}]$ for $r=1, \dots, n$ and $s=r, \dots, n$ in Type AI and $P_{(r-1)2n+s} = \mathbb{C}(q)[x_{12}, x_{13}, \dots x_{rs}]$ for $r=1, \dots, 2n$ and $s=r+1, \dots, 2n$ in Type AII.
It is well known that the quantized function algebra $\mathcal{O}_q(\Mat_N)$ can be expressed as an iterated Ore extension (see  for example \cite{GY}, Example 3.4)).  By Proposition \ref{proposition:relations_overview}, $\mathscr{P}_{\theta}$ in the diagonal setting satisfies the relations of $\mathcal{O}_q(\Mat_N)$ and hence $\mathscr{P}_{\theta}$ is an iterated Ore extension in this setting. A consequence of the next result is that  $\mathscr{P}_{\theta}$ can also be expressed as an iterated Ore extension in Types AI and AII.  

\begin{lemma} \label{lemma:homog}
 In Types AI and AII, the $x_{ij}$ satisfy the following 
$q$-type commuting property
\begin{align}\label{lexexp}
x_{lh}x_{dk} - q^{s(l,h,d,k)}x_{dk}x_{lh} \in P_{(l-1)m+h-1}
\end{align}
where  $m=n$ in Type AI, $m=2n$ in Type AII, $d\leq k$, $l\leq h$,  $(d,k)$ is less than $(l,h)$ in the lexicographic ordering, and $s(l,h,d,k)$ is a function  
from $\{1, 2, \dots, m\}^4$ to the integers.  
\end{lemma}
\begin{proof}   It is straightforward to check that (\ref{lexexp})  holds for all cases using the explicit relations given in Lemma \ref{lemma:explicit-relations}. 
\end{proof}

Recall the PBW basis for $\mathcal{O}_q(\Mat_N)$ as described in Section \ref{section:as}.  We use this basis to verify that certain sets of  monomials form PBW basis for $\mathscr{P}_{\theta}$ in each of the three cases. 

\begin{lemma} \label{lemma:PBW} The following monomials form a basis for $\mathscr{P}_{\theta}$:
\begin{itemize}
\item[(i)] Type AI: $$x_{11}^{m_{11}}x_{12}^{m_{12}}\cdots x_{1n}^{m_{1n}}x_{22}^{m_{22}}x_{23}^{m_{23}}\cdots x_{2n}^{m_{2n}}\cdots x_{n-1,n-1}^{m_{n-1,n-1}}x_{n-1,n}^{m_{n-1,n}}x_{nn}^{m_{nn}}$$
\item[(ii)] Type AII: 
$$x_{12}^{m_{12}}x_{13}^{m_{13}}\cdots x_{1,2n}^{m_{1,2n}}x_{23}^{m_{23}}x_{24}^{m_{24}}\cdots x_{2,2n}^{m_{2,2n}}\cdots x_{2n-1,2n}^{m_{2n-1,2n}}$$
\item[(iii)] Diagonal type: $$
x_{1,n+1}^{m_{11}}x_{12}^{m_{12}}\cdots x_{1,2n}^{m_{1n}}x_{2,n+1}^{m_{21}}x_{2,n+2}^{m_{22}}\cdots x_{2,2n}^{m_{2n}}\cdots x_{n,n+1}^{m_{n1}}\cdots x_{n,2n}^{m_{nn}}$$
\end{itemize}
as each $m_{ij}$ runs over nonnegative integers.  Moreover, we also get a basis if the order of the monomials in the terms above are reversed. \end{lemma}
\begin{proof} 

It follows from Lemma \ref{lemma:homog} that the above monomials span the algebra $\mathscr{P}_{\theta}$ in Type AI and AII.  By Proposition \ref{proposition:relations_overview}, the same is true for the diagonal type.  Hence, we need only check linear independence of the proposed basis elements.  


Consider first the Type AI case.   Set $\bar x_{ij} = \sum_{k\geq 2}t_{ik}t_{jk}$ for each $i,j$.   We proceed by induction on $n$, assuming that monomials as in (i) with $x_{ij}$ is replaced by $\bar{x}_{ij}$ form a basis for the subalgebra generated by $x_{ij}, i\geq 2, j\geq 2$.  Since there is only one monomial, $\bar{x}_{nn}$, for $i\geq n, j\geq n$, the base case is  is clearly true.

Let $\mathcal{S}$ be the subalgebra of $\mathcal{O}_q(\Mat_N)$ generated by 
$t_{ij}, i\geq 2, j\geq 1$.  It is straightforward to check from the relations of $\mathcal{O}_q(\Mat_N)$, that $t_{11}\notin \mathcal{S}$.  Moreover,  $\mathcal{O}_q(\Mat_N) =\sum_{r\geq 0}t_{11}^r \mathcal{S}$ and, as vector spaces,  this sum is isomorphic to a direct sum. Suppose that $Y=0$ where $Y$ is defined by
\begin{align}\label{comb}Y=\sum_{m}a_mx_{11}^{m_{11}}\cdots x_{nn}^{m_{nn}}
\end{align}
and the sum runs over  tuples $m=(m_{11}, \dots, m_{nn})$ so that each product in the $x_{ij}$ appearing in the right-hand side of (\ref{comb}) is a basis element as described in (i).
 Set $|m_1| = m_{11} + m_{12}+\cdots + m_{1n}$ and $M= \max_{m, a_m\neq 0} m_{11} + |m_1|$. 
 Set  \begin{align*}
Y':=\sum_{ m_{11}+|m_1|= M}a_m (t_{11}t_{11})^{m_{11}}(t_{11}t_{21})^{m_{12}} \cdots 
(t_{11}t_{n1})^{m_{nn}}x_{22}^{m_{22}} \cdots x_{{nn}}^{m_{nn}}
\end{align*}
and note that $Y\in Y' + \sum_{s<M}t_{11}^s\mathcal{S}$.
Hence $Y=0$ implies that $Y'=0$.

We can express $Y'$ as a sum of right $K_{\epsilon_1}$ eigenvectors with eigenvalues of the form $q^s$ where $s$ is an integer.   Let $M'= \min_{m, a_m\neq 0, m_{11} + |m_1| = M} |m_1|$ and let $\mathcal{S}'' = \{m, a_m\neq 0, m_{11} + |m_1| = M, |m_1| = M'\}$.   The eigenvector with smallest exponent in the eigenvalue is 
\begin{align*}
Y'' =\sum_{ m\in \mathcal{S}''}a_m (t_{11}t_{11})^{m_{11}}(t_{11}t_{21})^{m_{12}} \cdots 
(t_{11}t_{n1})^{m_{nn}}\bar{x}_{22}^{m_{22}} \cdots \bar{x}_{{nn}}^{m_{nn}}.  
\end{align*} Hence $Y=0$ implies that $Y''=0$.
Given an $n$-tuple $w$, write $\mathcal{S}''_w = \{m\in \mathcal{S}''| m_1 = w\}$.  We can express $Y''$ as a sum $\sum_wY''_w$ where 
\begin{align*}Y''_w = \sum_{w}(t_{11}t_{11})^{w_1}(t_{11}t_{21})^{w_2} \cdots 
(t_{11}t_{n1})^{w_n}\sum_{m\in \mathcal{S}''_w}a_m\bar{x}_{22}^{m_{22}} \cdots \bar{x}_{{nn}}^{m_{nn}}
\end{align*}
It follows from the form of the PBW basis for $\mathcal{O}_q(\Mat_N)$ that $Y''=0$ implies $Y''_w=0$ each $w$.  By the inductive assumption, we get $a_m=0$ for all $m\in \mathcal{S}''_w$, a contradiction.  This proves the linear independence of the basis elements listed in (i).

 For (ii), we assume that $Y$ is defined similarly to (\ref{comb}) where the sum runs  over  tuples $m=(m_{12}, m_{13}, \dots, m_{2n-1,2n})$ so that each product in the $x_{ij}$ appearing in the right-hand side of (\ref{comb}) is a basis element as described in (ii).  As in case (i), $\mathcal{S}$ is the subalgebra generated by the $t_{ij}$ with $i\geq 2$ and $j\geq 1$. We have $Y\in Y' + \sum_{s<M}t_{11}^s\mathcal{S}$ where
    \begin{align*}
Y':=\sum_{ |m_1|= M}a_m (t_{11}t_{22})^{m_{11}}(t_{11}t_{32})^{m_{12}} \cdots 
(t_{11}t_{2n,2})^{m_{2n,2}}x_{22}^{m_{22}} \cdots x_{m_{2n,2n}}^{m_{2n,2n}}
\end{align*}
and $M = \max_{m,a_m\neq 0} |m_1|.$ The argument now follows as in case (i) using the decomposition of $Y'$ into a sum of eigenvectors with respect to the right action of $K_{\epsilon_1}$.

 Now consider case (iii).  Suppose that $Y$ is defined as in (\ref{comb}) where the sum runs  over  tuples $m=(m_{11}, m_{12}, \dots, m_{nn})$ so that each product in the $x_{ij}$ appearing in the right-hand side of (\ref{comb}) is a basis element as described in (iii). Note that the $Y$ can be expressed as a sum of eigenvectors with respect to the right action of $K_{\epsilon_n}$. The eigenvector with eigenvalue $q^s$ where $s$ is maximum is   \begin{align}\label{secondsum}Y'= \sum_{m}a_m(t_{1n} t_{n+1,2n})^{m_{11}} \cdots (t_{nn}t_{2n,2n})^{m_{nn}}   \end{align}
  Hence $Y=0$ implies $Y'=0$.  
Recall that in this case, we are viewing the $t_{ij}, 1\leq i\leq n$ as generators for the first copy of $\mathcal{O}_q(\Mat_N)$ and 
$t_{i+n,j+n}, 1\leq i,j\leq n$ for the second copy. In particular, each $t_{ij}$ commutes with each $t_{k+n,l+n}$ for all  $1\leq i,j,k,l\leq n$.  Thus we can rewrite $Y'$ as 
 \begin{align*}Y'=  \sum_{m}a_mt_{1n}^{m_{11}} \cdots t_{nn}^{m_{nn}} t_{n+1,2n}^{m_{11}} \cdots t_{2n,2n}^{m_{nn}}
  \end{align*}
  It follows from the description of the PBW basis for $\mathcal{O}_q(\Mat_N)$ that $Y=0$ implies that each $a_m=0$, which establishes linear independence for case (iii).
  
  The final assertion is proved in the same way by reversing the order of the $t_{ij}$ terms in each of the above monomials.  \end{proof}

The next lemma provides a nice  algebraic description of the algebras $\mathscr{P}_{\theta}$ for all three cases. 
It should be noted that  this result is  implicit in \cite{N} in Types AI and AII.  In particular, one could prove this result  using specialization techniques combined with the similarity of the structure of these algebras as  $U_q(\mathfrak{gl}_n)$-bimodules and their classical counterparts.

\begin{proposition}\label{prop:Palgebra} The map sending each $\tilde{x}_{ij}$ to $x_{ij}$ defines an algebra surjection from   the free algebra $\mathbb{C}(q)\langle \tilde{x}_{ij}, 1\leq i,j\leq {\rm rank}(\mathfrak{g})\rangle$ to $\mathscr{P}_{\theta}$ with kernel equal to the ideal generated by the following elements  
\begin{itemize}
\item[(i)] $\tilde{x}_{ij} - q\tilde{x}_{ji}$ all $i<j$ in Type AI
\item[(ii)] $\tilde{x}_{ij}+q^{-1}\tilde{x}_{ji}$ for all $i<j$ and $\tilde{x}_{ii}$ for all $i$ in Type AII.
\item[(iii)] $\tilde{x}_{i,j+n} - \tilde{x}_{j+n,i}$ and $\tilde{x}_{ij},\tilde{x}_{n+i,n+j}$  for all $i,j=1, \dots, n$ in diagonal type.
\item[(iv)] the matrix entries of  $R_{\mathfrak{g}}\tilde{X}_1R_{\mathfrak{g}}^{t_1}\tilde{X}_2 - \tilde{X}_2R_{\mathfrak{g}}^{t_1} \tilde{X}_1R_{\mathfrak{g}}$ 
 where $\tilde{X}$ is the ${\rm rank}(\mathfrak{g})\times {\rm rank}(\mathfrak{g})$ matrix with $ij$ entry $\tilde{x}_{ij}$.
 \end{itemize}
 Moreover, in the diagonal case,  (iv) is equivalent to the set of 
 matrix entries of  $R{\hat{\tilde{X}}}_1\hat{\tilde{X}}_2 -\hat{\tilde{X}}_2\hat{\tilde{X}}_1R$  where $\hat{\tilde{X}}'$ is the $n\times n$ matrix with $ij$ entry $x_{i,j+n}$ for $i,j=1, \dots, n.$
\end{proposition}
\begin{proof} 
 It follows from Lemmas \ref{lemma:PBW} and \ref{lemma:homog} that there are no additional relations satisfied by the $x_{ij}$ beyond those listed in Proposition \ref{proposition:relations_overview}. The final assertion follows from Lemma \ref{lemma:diagxequalst}.
\end{proof}

Recall that in defining the $x_{ij}$, we took a particular choice of the $n$-tuple $a$, namely $a= (1, \dots, 1)$ and set 
each $x_{ij}=x_{ij}(1,\dots, 1)$. It should be noted that if we picked another choice, the resulting elements would still statisfy the same relations as given in the above proposition. This is because different choices correspond to algebra automorphisms of $\mathscr{P}$. 

Recall also that the antialgebra automorphism $*$ sending each $t_{ij}$ to $\partial_{ji}$ has the following impact on the $x_{ij}$: $(x_{ij}(a))^* = d_{ji}(a)$ for all $n$-tuples $a$.   This does not mean that $x_{ij}^*=d_{ij}$ since $x_{ij}$ and $d_{ij}$ are defined using different, though related, $n$-tuples.  However, just as in the case of the $x_{ij}(a)$, we have that the algebra generated by the $d_{ij}$ is isomorphic to the algebra generated by $d_{ij}(a)$ for all $n$-tuples $a$.  Thus, the map $x_{ij}$ to $d_{ji}$ for all $i,j$ defines an algebra antiautomorphism from $\mathscr{P}_{\theta}$ to $\mathscr{D}_{\theta}$. 
 Using this antiautomorphism, it is  straightforward  to deduce the following result  for $\mathscr{D}_{\theta}$ directly from  
 Proposition \ref{prop:Palgebra}.

\begin{proposition}\label{prop:Dalgebra} The map sending each $\tilde{d}_{ij}$ to $d_{ij}$ defines an algebra surjection from  the free algebra $\mathbb{C}(q)\langle \tilde{d}_{ij}, 1\leq i,j\leq n\rangle$  to $\mathscr{D}_{\theta}$  with kernel equal to the ideal of relations generated by the following elements
\begin{itemize}
\item[(i)] $\tilde{d}_{ij} - q^{-1}\tilde{d}_{ji}$ all $i<j$ in type AI
\item[(ii)] $\tilde{d}_{ij}+q\tilde{d}_{ji}$ for all $i<j$ and $\tilde{d}_{ii} $ for all $i$ in type AII.
\item[(iii)] $\tilde{d}_{i,j+n} -\tilde{d}_{j+n,i}$ and $\tilde{d}_{ij}, \tilde{d}_{i+n,j+n}$ for all $i,j=1, \dots, n$ in  diagonal type.
\item[(iv)] the matrix entries of  $R_{\mathfrak{g}}\tilde{D}_2R_{\mathfrak{g}}^{t_1}\tilde{D}_1 - \tilde{D}_1R_{\mathfrak{g}}^{t_1} \tilde{D}_2R_{\mathfrak{g}}$  in types AI and AII where $\tilde{D}$ is the ${\rm rank}(\mathfrak{g})\times {\rm rank}(\mathfrak{g})$ matrix with $ij$ entry $\tilde{d}_{ij}$.
\end{itemize}
Moreover, in the diagonal setting, (iv) is equivalent to the set of  matrix entries of  $R\hat{\tilde{D}}_2\hat{\tilde{D}}_1= \hat{\tilde{D}}_1\hat{\tilde{D}}_2R$ where $\hat{\tilde{D}}$ is the $n\times n$ matrix with $ij$ entry $\tilde{d}_{i,j+n}.$
\end{proposition}

In the diagonal setting, using  Proposition \ref{prop:Palgebra}, we have that $\mathscr{P}_{\theta}$ is isomorphic  as an algebra to $\mathcal{O}_q(\Mat_n)$ via the map sending $x_{i,j+n}$ to $t_{ij}$.  Lemma \ref{lemma:left_action} shows that   this is an isomorphism of left $U_q(\mathfrak{gl}_n)$-modules using the first copy of $U_q(\mathfrak{gl}_n)$.  The next result shows that this map is actually a $U_q(\mathfrak{gl}_n)$-bimodule algebra isomorphism where the right  action of $U_q(\mathfrak{gl}_n)$ on $\mathcal{O}_q(\Mat_n)$ corresponds to the left action of the second copy of $U_q(\mathfrak{gl}_n)$ on $\mathscr{P}_{\theta}$.

\begin{lemma}\label{lemma:diagxequalst} For diagonal  type, the map $\psi$ defined by 
  $\psi(t_{ ij})=x_{i,j+n} $ all $i,j$ with $1\leq i,j\leq n$  is a $U_q(\mathfrak{gl}_n)$-bimodule algebra isomorphism  from $\mathcal{O}_q(\Mat_n)$ onto $\mathscr{P}_{\theta}$ where 
\begin{align}\label{adotformulas}
\psi(a \cdot t_{ij}) =a\cdot  x_{i,j+n} {\rm \ and \ } \psi(t_{ji}\cdot a^{\natural}) = \gamma(a)\cdot x_{i,j+n}
\end{align}
for all $i,j\in \{1, \dots, n\}$, where $\gamma$ is the map from the first copy of $U_q(\mathfrak{gl}_n)$ to the second defined by $\gamma(E_r) = E_{n+r}, 
\gamma(F_r)=F_{n+r},$ and $\gamma(K_{\epsilon_s}) = K_{\epsilon_{n+s}}$ for all $r,s,\in  \{1, \dots, n\}$.
\end{lemma}
\begin{proof} By the discussion proceeding the proposition, we only need to show that the right action of $U_q(\mathfrak{gl}_n)$ on $\mathcal{O}_q(\Mat_n)$ corresponds to the left action of  $\gamma(U_q(\mathfrak{gl}_n))$ on $\mathscr{P}_{\theta}$ using the second equality in (\ref{adotformulas}). This follows from the method for converting the right action of $U_q(\mathfrak{gl}_n)$ on $\mathcal{O}_q(\Mat_n)$
into a left one using Lemmas \ref{lemma:actions}.
\end{proof}

Note that the above argument can be easily tweaked to apply to $\mathscr{D}_{\theta}$.  More precisely, in the diagonal case, $\mathscr{D}_{\theta}$ is isomorphic as a bimodule algebra  to $\mathcal{O}_q(\Mat_n)^{op}$ via the map sending $\partial_{ij}$ to $d_{i,j+n}$ for all $i,j$ satisfying $1\leq i,j\leq n$.  

\section{Graded Weyl Algebras for Matrices}\label{section:graded-weyl}
\subsection{Twisted Tensor Products}\label{section:twisted}
We consider here a particular type of twisted tensor products for bialgebras, though the end result does not necessarily have a bialgebra structure.  This twisted tensor product is similar, but not the same as the construction for Drinfeld doubles as presented 
in \cite{KS}, Section 8. First, we recall the notions of twisted tensor products and dual pairings.

 Let $A$ and $B$ be algebras over a field.  The twisted tensor product, as defined  in \cite{CSV}, is an algebra $C$ with multiplication map $m_C$ along with two inclusion maps $\iota_A$ and $\iota_B$ such that 
$m\circ(\iota_A\otimes \iota_B)$ defines an isomorphism as vector spaces from $A\otimes B$ to $C$.  The twisted tensor product comes equipped with a twisting map $\tau$  which is a linear map from $B\otimes A$ to $A\otimes B$  that satisfies 
\begin{itemize}
\item[(i)]$\tau(1\otimes a) = a\otimes 1{\rm \ and \ }
\tau(b\otimes 1) = 1\otimes b$
\item[(ii)] $m_C = (m_A\otimes m_B)\circ (Id_A\otimes \tau\otimes 
Id_B)$
\item[(iii)]$\tau\circ (m_B\otimes m_A) = m_C\circ (\tau \otimes \tau)\circ (Id_B\otimes \tau\otimes Id_A)$
\end{itemize}
where  $m_A$ denotes multiplication for $A$, $m_B$ denotes multiplication for $B$, $Id_B$ is the identity map on $B$ and $Id_A$ is the identity map on $A$.  Moreover, the existence of a twisting map is essential here.  In other words, given a twisting map $\tau$  (i.e. a linear map that satisfies (i) and the map $m_C$ defined by (ii) that satisfies the constraints in (iii))  one can form a twisted tensor product $A\otimes_{\tau}B$ with multiplication $m_C$ as defined in  (ii).

\subsection{Dual Pairing Construction}
In the next lemma, we use dual pairings as in the construction for Drinfeld doubles and other forms of twisted tensor products in   \cite{KS}, Section 8.  A dual pairing  is a bilinear map  $\langle \cdot, \cdot \rangle$ from $A\times B$ to the scalars where $A,B$ are bialgebras such that 
\begin{align*}
\langle a,1\rangle = \epsilon(a) {\rm \ and \ }\langle 1,b\rangle = \epsilon(b)
\end{align*}
along with the following compatibility formulas with respect to  comultiplication:
\begin{align}\label{pairing-ext}
\langle \Delta_{A}(a), b_1\otimes b_2 \rangle = \langle a, b_1b_2\rangle 
{\rm \quad  and \quad} \langle a_1\otimes a_2, \Delta_{B}(b) \rangle = \langle a_1a_2, b\rangle
\end{align}

\begin{lemma}\label{twisted-tensors} Let $A$ and $B$ be two bialgebras with two pairings: ${\bf u}\langle\cdot,\cdot\rangle$ is a dual pairing of $A^{op}$ and $B$ and ${\bf v}\langle\cdot,\cdot\rangle$ is a dual pairing of  $A$ and $B^{op}$.  Then $A\otimes B$ becomes a twisted tensor product  with twisting map defined by 
\begin{align*}
\tau(b\otimes a) = \sum a_{(2)} \otimes b_{(2)}  {\bf v}\langle a_{(1)},b_{(1)}\rangle{\bf u}\langle a_{(3)}, b_{(3)}\rangle.
\end{align*}
\end{lemma}
\begin{proof}  
 For property (i), note  that 
\begin{align*}
\tau(b\otimes 1) = \sum 1\otimes b_{(2)}{\bf v}\langle 1,b_{(1)}\rangle{\bf u}\langle 1,b_{(3)}\rangle = \sum 1\otimes b_{(2)} \epsilon(b_{(1)}) \epsilon(b_{(3)}) = 1\otimes b
\end{align*}
for all $b\in B$. 
A similar argument yields $\tau(1\otimes a) = a\otimes 1$ for all $a\in A$.  As explained in \cite{CSV}, Section 2.2, we make take (ii) as the definition for the multiplication map $m_C$ since we have already verified that $\tau$ satisfies (i). For  property (iii), note that by \cite{CSV}, Proposition/Definition 2.3,
this condition is equivalent to  associativity. Moreover, associativity follows as in the proof of \cite{KS}, Section 8.2.1, Proposition 8.  Indeed, \cite{KS}, Section 8.2.1, Proposition 8 focuses on a product whose only difference from the one here is the extra assumption that ${\bf u}$ and ${\bf v}$ are convolution inverses.  This additional assumption is not needed for the proof of associativity.  \end{proof}

 Let $\zeta$ be an $s$-dimensional representation of $U_q(\mathfrak{gl}_N)$ and set $R_{\zeta} = (\zeta\otimes \zeta)(\mathcal{R})$ as in Section  \ref{section:FRT1}.  Let $A(R_{\zeta})$ denote the FRT algebra defined by $\zeta$ realized as a quotient of tensor algebra $T(M)$ over the $r^2$ dimensional vector space $M$ (as defined in Section \ref{section:FRT1}).  Let $\xi$ be another $U_q(\mathfrak{gl}_N)$ representation of dimension $s$ and write $A(R_{\xi})$ for the FRT bialgebra defined by $\xi$ where $R_{\xi} = (\xi\otimes \xi)(\mathcal{R})$.
 For the construction of $A(R_{\xi})$,   we use the vector space $M'$ spanned by the variables $m'_{ij}$, $1\leq i,j\leq s$ in order to distinguish 
$A(R_{\xi})$ from $A(R_{\zeta})$.   Set $ R_{\zeta,\xi}=(\zeta\otimes \xi)(\mathcal{R})$.

Note that a  bilinear pairing $\langle \cdot, \cdot \rangle$ on  $M\times M'$ extends uniquely to a dual pairing on  $T(M)$ and $ T(M')^{op}$ by insisting that 
\begin{align*}
\langle m_{ij}, 1\rangle = \delta_{ij} {\rm \ and \ }\langle 1, m'_{kl}\rangle = \delta_{kl}
\end{align*}
for all $i,j,k,l$ 
and  that (\ref{pairing-ext}) holds (using induction and properties of the coproduct which ensure associativity). 
The next lemma provides us with dual pairings that can be used to form twisted tensor products. Its proof follows closely the proof of  \cite{KS}, 10.1.7.
\begin{lemma}\label{lemma:dualpairing} 
Assume that 
$(\xi\otimes \xi)(\mathcal{R}) = \flip\circ (\zeta\otimes \zeta)(\mathcal{R})$.  If  $(\zeta\otimes \xi)(\mathcal{R}) =\flip\circ ((\zeta\otimes \xi)(\mathcal{R} ))$ then the bilinear map ${\bf y}\langle \cdot, \cdot \rangle$ defined by 
\begin{align*}
{\bf y}\langle m_{ij}, m'_{kl}\rangle =[ (\zeta\otimes \xi)(\mathcal{R})]^{ik}_{jl}
\end{align*}
can be uniquely extended to  a dual pairing of the bialgebra  $A(R_{\zeta})$ and $A(R_{\xi})$.  Moreover, the same result holds for $\mathcal{R}$ replace by $\mathcal{R}_{21}^{-1}$.
\end{lemma}
\begin{proof}
 Define the bilinear map ${\bf y}$ from $M\times M'$ to $\mathbb{C}(q)$
via 
\begin{align*}
{\bf y}\langle m_{ij},m'_{kl}\rangle = (R_{\zeta,\xi})^{ik}_{jl}\quad {\rm and}\quad
{\bf y}\langle 1,m'_{ij}\rangle = {\bf y}\langle m_{ij},1\rangle = \delta_{ij}
\end{align*}
The bilinear form ${\bf y}\langle \cdot, \cdot \rangle$ extends uniquely  to a dual pairing, which we also denote by ${\bf y}\langle \cdot, \cdot \rangle$, of
 $T(M)$ with $T(M')$  as explained before the lemma.  
 We now argue as in the proof of \cite{KS}, Theorem  10.1.7, that   ${\bf y}\langle \cdot, \cdot \rangle$ induces the appropriate form on the corresponding FRT bialgebras.  To do this, we need to show that this form vanish on the relations of $A(R_{\xi})$ and $A(R_{\zeta})$.  
 Note that uniqueness is forced on us since we have specified ${\bf y}\langle \cdot, \cdot \rangle$ on scalars and terms in $M'$ and $M$ and the remaining values  follow from  induction using (\ref{pairing-ext}) just as they do on the tensor algebra level. 

 The ideal of relations $\mathcal{I} = \mathcal{I}_{\zeta}$ for $A(R_{\zeta})$ realized as a quotient of $T(M)$ is generated by the elements $\mathcal{I}^{ld}_{ab}$, for all $l,d,a,b$, where  
\begin{align*}
\mathcal{I}^{ld}_{ab} =  \sum_{j,h}(R_{\zeta})^{ld}_{jh} m_{ja}m_{hb} - m_{dh}m_{lj}(R_{\zeta})^{jh}_{ab}
\end{align*}
For each $l,d,a,b, r,s$ we have 
\begin{align*}
{\bf y}\langle \mathcal{I}^{ld}_{ab}, m'_{rs}\rangle&= \sum_{j,h}(R_{\zeta})^{ld}_{jh}{\bf y}\langle m_{ja}m_{hb},  m'_{rs} \rangle - {\bf y}\langle m_{dh}m_{lj},m'_{rs} \rangle(R_{\zeta})^{jh}_{ab}\cr
&=\sum_{j,h,k}(R_{\zeta})^{ld}_{jh}{\bf y}\langle m_{ja},m'_{rk}\rangle {\bf y}\langle m_{hb},m'_{ks}\rangle - {\bf y}\langle m_{dh}, m'_{rk}\rangle{\bf y}\langle m_{lj},m'_{ks}\rangle (R_{\zeta})^{jh}_{ab}\cr
&=\sum_{j,h,k}(R_{\zeta})^{ld}_{jh}(R_{\zeta,\xi})^{jr}_{ak}(R_{\zeta,\xi})^{hk}_{bs} - (R_{\zeta,\xi})^{dr}_{hk}(R_{\zeta,\xi})^{lk}_{js} (R_{\zeta})^{jh}_{ab}
\end{align*}
Hence ${\bf y}\langle\mathcal{I}^{ld}_{ab}, m'_{rs} \rangle$ equals the $ldr, abs$ entry of the matrix
\begin{align*}
(R_{\zeta})_{12}(R_{\zeta,\xi})_{13}(R_{\zeta,\xi})_{23} - (R_{\zeta,\xi})_{23}(R_{\zeta,\xi})_{13}(R_{\zeta})_{12}.
\end{align*}
This matrix equals 
$(\zeta\otimes \zeta\otimes \xi)(\mathcal{R}_{12}\mathcal{R}_{13}\mathcal{R}_{23}-\mathcal{R}_{23}\mathcal{R}_{13}\mathcal{R}_{12}),$
which is just  the image of the  Quantum Yang-Baxter Equation  under $(\zeta\otimes \zeta\otimes \xi)$. Hence, this matrix must equal the $0$ matrix because 
 $\mathcal{R}$ is a universal $R$-matrix.
Thus, ${\bf y}\langle\mathcal{I}^{ld}_{ab}, m'_{rs}\rangle=0$ for all $l,d,a,b, r, s$ and so  ${\bf y}\langle\mathcal{I}, M'\rangle = {\bf y}\langle\mathcal{I}, T(M')\rangle = 0$.

The ideal $\mathcal{J} = \mathcal{J}_{\xi}$ defining the  relations for $A(R_{\xi})$ is generated by the elements 
 \begin{align*}
\mathcal{J}^{ld}_{ab} =  \sum_{j,s}(R_{\xi})^{ld}_{js} m'_{ja}m'_{sb}- m'_{ds}m'_{lj}(R_{\xi})^{js}_{ab} 
\end{align*}
for all $l,d,a,b$.
Arguing as above, we see that ${\bf y}\langle m_{rt},\mathcal{J}^{ld}_{ab}\rangle$ equals
 \begin{align*}
&\sum_{j,s,k}(R_{\xi})^{ld}_{js}(R_{\zeta,\xi})^{rs}_{kb}(R_{\zeta,\xi})^{kj}_{ta} -(R_{\zeta,\xi})^{rl}_{kj}(R_{\zeta,\xi})^{kd}_{ts}(R_{\xi})^{js}_{ab}
\end{align*}
By assumption, $R_{\xi} = \flip(R_{\zeta})$ and $R_{\zeta,\xi} = \flip(R_{\zeta,\xi})$.  Hence, we can rewrite ${\bf y}\langle m_{rt},\mathcal{J}^{ld}_{ab}\rangle$ as 
 \begin{align*}
&\sum_{j,s,k}(R_{\zeta})^{dl}_{sj}(R_{\zeta,\xi})^{sr}_{bk}(R_{\zeta,\xi})^{jk}_{at} -(R_{\zeta,\xi})^{lr}_{jk}(R_{\zeta,\xi})^{dk}_{st}(R_{\xi})^{sj}_{ba}
\cr&= ((R_{\zeta})_{12}(R_{\zeta,\xi})_{13}(R_{\zeta,\xi})_{23}-(R_{\zeta,\xi})_{23}(R_{\zeta,\xi})_{13}(R_{\zeta})_{12})^{dlr}_{bat}
\end{align*} 
Using the Quantum Yang-Baxter Equation again, this reduces to $0$ and so ${\bf y}\langle M, \mathcal{J}\rangle = {\bf y}\langle T(M), \mathcal{J}\rangle = 0$.  This completes the proof for $\mathcal{R}$.  Since $\mathcal{R}_{21}^{-1}$ is also a universal $R$-matrix, the same analysis holds when $\mathcal{R}$ is replaced by $\mathcal{R}_{21}^{-1}$.
\end{proof}

\subsection{Four Twisted Tensor Products}\label{section:four}

Recall  that when $\zeta=\rho$, the FRT algebra $A(R_{\zeta}) = \mathcal{O}_q(\Mat_N)$ (see Section \ref{section:as}).  Similarly, as explained in Section \ref{section:mr}, for $\xi=\rho\circ \natural\circ S$, we get $A(R_{\xi})=\mathcal{O}_q(\Mat_N)^{op}$.  In the next proposition, we construct four twisted tensor products of $\mathcal{O}_q(\Mat_N)$ and $\mathcal{O}_q(\Mat_N)^{op}$ using these choices of $\xi$ and $\zeta$.  Recall the matrix  $R$ given in  formula   (\ref{Rmatrix2}).  Set $R_{0} = R^{t_2}$ and $R_{1} = ((R_{21})^{-1})^{t_2}$.

\begin{proposition}\label{prop:twisted} For each choice of $\upsilon$ and $\sigma$ in $\{0,1\}$, there exist (unique) dual pairings ${\bf v}_{\sigma}\langle \cdot, \cdot\rangle$  on $ \mathcal{O}_q(\Mat_N)$ and $\mathcal{O}_q(\Mat_N)^{op}$
and ${\bf u}_{\upsilon}\langle \cdot, \cdot\rangle$  on $ \mathcal{O}_q(\Mat_N)^{op}$ and $\mathcal{O}_q(\Mat_N)$  such that 
\begin{align*} {\bf v}_{\sigma}\langle t_{ij}, \partial_{kl}\rangle = [R_{\sigma}]^{ik}_{jl}{\rm \quad and \quad}
{\bf u}_{\upsilon}\langle t_{ij}, \partial_{kl}\rangle = [R_{\upsilon}]^{ik}_{jl}.\end{align*}
  Moreover, the twisting map $\tau_{\upsilon,\sigma}$  defined by ${\bf u}_{\upsilon}\langle \cdot, \cdot\rangle$ and ${\bf v}_{\upsilon}\langle \cdot, \cdot \rangle$ as in Lemma \ref{twisted-tensors}
satisfies 
\begin{align*}
\tau_{\upsilon,\sigma}(\partial_{ea}\otimes t_{fb}) = \sum_{j,k,d,l}(R_{\sigma})^{dl}_{fe} (R_{\upsilon})^{jk}_{ba}t_{dj}\otimes\partial_{lk}
\end{align*} 
and restricts to a linear isomorphism of $\sum_{i,j,k,l}\mathbb{C}(q)\partial_{ij}\otimes t_{kl}$ onto $\sum_{i,j,k,l}\mathbb{C}(q)t_{kl}\otimes \partial_{ij}$.
\end{proposition}
\begin{proof} By Lemmas \ref{Rimage} and \ref{moreRimage}, we have 
$ (\rho\otimes \rho)(\mathcal{R}) = R$,
$ (\rho\circ \natural\circ S)\otimes ( \rho\circ \natural \circ S)(\mathcal{R})=R_{21}=\flip(R)$,  
   $ ( ( \rho\circ \natural \circ S)\otimes \rho) (\mathcal{R}) =( R_{21}^{-1})^{t_2}$ and  $ ( ( \rho\circ\natural \circ S)\otimes \rho) (\mathcal{R}_{21}^{-1}) =R^{t_2}$.  It is straightforward to check from the explicit formula (\ref{Rmatrix2}) for $R$  that 
 $\flip( (R_{21}^{-1})^{t_2}) =  (R_{21}^{-1})^{t_2}$ and $\flip(R^{t_2}) = R^{t_2}$.  Thus $\zeta=\rho\circ\natural\circ S$ and $\xi=\rho$ satisfy the conditions of Lemma
 \ref{lemma:dualpairing} with respect to the forms ${\bf u}_{\upsilon}\langle \cdot, \cdot \rangle, \upsilon \in \{0,1\}$. 
 Moreover, by Lemma \ref{moreRimage}, we see that
 \begin{align*}
 {\bf u}_0\langle t_{ij}, \partial_{kl}\rangle = [(\xi\otimes \zeta)(\mathcal{R}_{21}^{-1})]^{ik}_{jl} = (R^{t_2})^{ik}_{jl} = (R_0)^{ik}_{jl}
 \end{align*}
 and 
 \begin{align*}
 {\bf u}_1\langle t_{ij}, \partial_{kl}\rangle = [(\xi\otimes \zeta)(\mathcal{R})]^{ik}_{jl} = ((R_{21}^{-1})^{t_2})^{ik}_{jl} = (R_1)^{ik}_{jl}.
 \end{align*}
 
  A similar argument yields that $\zeta=\rho$ and $\xi'=\rho\circ \natural\circ S^{-1}$ satisfies the conditions of Lemma \ref{lemma:dualpairing} with respect to the forms ${\bf v}_{\upsilon}, \upsilon\in \{0,1\}$. Thus, by Lemma \ref{lemma:dualpairing}, the bilinear forms ${\bf v}_{\upsilon}\langle \cdot,\cdot\rangle$ can be extended uniquely to the stated dual pairings for $\upsilon=0,1$.  Furthermore,  it follows from Lemma \ref{moreRimage}   that \begin{align*}
 {\bf v}_0\langle t_{ij}, \partial_{kl}\rangle = [(\zeta\otimes \xi')(\mathcal{R}_{21}^{-1})]^{ik}_{jl} = (R^{t_1})^{ik}_{jl} = (R_0)_{ik}^{jl}
 \end{align*}
 and 
 \begin{align*}
 {\bf v}_1\langle t_{ij}, \partial_{kl}\rangle = [(\zeta\otimes \xi')(\mathcal{R})]^{ik}_{jl} = ((R_{21}^{-1})^{t_1})^{ik}_{jl} = (R_1)_{ik}^{jl}.
 \end{align*}

By the definition of the twisted tensor product $\mathcal{O}_q(\Mat_N)\otimes_{\tau_{\upsilon,\sigma}}\mathcal{O}_q(\Mat_N)^{op}$ in Lemma \ref{twisted-tensors}, we have 
\begin{align*}
\tau_{\upsilon,\sigma}(\partial_{ea}\otimes t_{fb})&= \sum_{d,j,l,k} t_{dj}\otimes\partial_{lk}{\bf v}_{\sigma}\langle t_{fd},\partial_{el}\rangle{\bf u}_{\upsilon}\langle t_{jb},\partial_{ka}\rangle.
\end{align*}
The lemma now follows by plugging in the values for ${\bf u}_{\upsilon}\langle \cdot,\cdot\rangle$ and ${\bf v}_{\sigma}\langle \cdot,\cdot\rangle$ as described above and noting that both $R_0$ and $R_1$ are invertible matrices. 
\end{proof}

Given $\upsilon,\sigma\in \{0,1\}$, 
 define algebras $\mathcal{A}_{\upsilon,\sigma}$ as twisted tensor products
\begin{align*}\mathcal{A}_{\upsilon,\sigma} = \mathcal{O}_q(\Mat_N)\otimes_{\tau_{\upsilon,\sigma}}\mathcal{O}_q(\Mat_N)^{op}
\end{align*}
where $\tau_{\upsilon,\sigma}$ are the twisting maps from Proposition \ref{prop:twisted}.   The twisting map for $\mathcal{A}_{\upsilon,\sigma}$ gives us the  following equalities
\begin{align}\label{twistdefn}
\partial_{ea} t_{fb} = \sum_{j,k,d,l}(R_{\sigma}^{t_2})^{dl}_{fe} (R_{\upsilon}^{t_2})^{jk}_{ba}t_{dj}\partial_{lk}
\end{align}
for all $e,a,f,b,$ where $R_0 = R$ and $R_1 = R_{21}^{-1}$.
Note that this equality combined with the embeddings of $\mathcal{O}_q(\Mat_N)$ and $\mathcal{O}_q(\Mat_N)^{op}$ inside $\mathcal{A}_{\upsilon,\sigma}$ define multiplication on this twisted tensor product.  Just as was done for $\mathcal{O}_q(\Mat_N)$ and $\mathcal{O}_q(\Mat_N)^{op}$ (see (\ref{FRTmatrix}) and (\ref{FRTmatrixop})), relations (\ref{twistdefn}) can be put into matrix form as 
\begin{align}\label{twistdefn-matrix}
P_2T_1 = R_{\sigma}^{t_1}T_1P_2R_{\upsilon}^{t_2}.
\end{align}


Recall that both $\mathcal{O}_q(\Mat_N)$ and $\mathcal{O}_q(\Mat_N)^{op}$ are $U_q(\mathfrak{gl}_n)$-bimodule algebras.   The nexts result shows that these bimodule algebra structures extend to the twisted tensor products of the above proposition.

\begin{proposition}\label{prop:bimodule}
For each $\upsilon,\sigma\in \{0,1\}$, the twisted tensor product $\mathcal{A}_{\upsilon,\sigma}$ inherits a $U_q(\mathfrak{gl}_N)$-bimodule algebra structure from $\mathcal{O}_q(\Mat_N)$ and $\mathcal{O}_q(\Mat_N)^{op}$.  \end{proposition}

\begin{proof} Recall that as an algebra, $\mathcal{O}_q(\Mat_N)$ is a quotient of the tensor algebra $T(V\otimes W)$ modulo the relations coming from the FRT construction.  Similarly, $\mathcal{O}_q(\Mat_N)^{op}$ is a quotient of the tensor algebra $T(V^*\otimes W^*)$ modulo FRT relations.  Thus $\mathcal{A}_{\upsilon,\sigma}$ can be viewed as a quotient of $T(V\otimes W)\otimes T(V^*\otimes W^*)$ modulo three types of relations: the relations for $\mathcal{O}_q(\Mat_N)$, the relations for $\mathcal{O}_q(\Mat_N)^{op}$, and the relations that come from the twisting map as in (\ref{twistdefn}).  We have already shown that the first two types of relations are invariant under the left and right action 
of $U_q(\mathfrak{gl}_N)$.   So we only need to check that the extension of these actions preserve the relations coming from the twisted tensor product. 

The relations coming from the twisted tensor product can be lifted to   the level of the tensor algebra $T(V\otimes W)\otimes T(V^*\otimes W^*)$ as 
\begin{align*}
(v^*_e\otimes w^*_a)\otimes (v_f\otimes w_b)- \sum_{j,k,d,l}(R_{\sigma}^{t_2})^{dl}_{fe} (R_{\upsilon}^{t_2})^{jk}_{ba}(v_d\otimes w_j)\otimes (v^*_l\otimes w^*_k).
\end{align*}These relations  correspond to the mappings of vector spaces
\begin{align}\label{VWform}
(V^*\otimes W^*)\otimes (V\otimes W)\longrightarrow (R^{t_2}_{\sigma})_{13}((V\otimes W)\otimes (V^*\otimes W^*))(R_{\upsilon}^{t_1})_{24}.
\end{align}

By  Lemma \ref{moreRimage} and the facts that $\flip(R^{t_2}) = R^{t_2}$ and $\flip((R_{21}^{-1})^{t_2}) =( R_{21}^{-1})^{t_2}$, we have 
\begin{align*}R_0= R^{t_2} =((\rho\circ \natural\circ S)\otimes \rho)(\mathcal{R}_{21}^{-1}) =(\rho\otimes (\rho\circ \natural\circ S))(\mathcal{R}_{12}^{-1}) 
\end{align*}
and 
\begin{align*}R_1= (R_{21}^{-1})^{t_2} =((\rho\circ \natural\circ S)\otimes \rho)(\mathcal{R}) = (\rho\otimes (\rho\circ \natural\circ S))(\mathcal{R}_{21}) 
\end{align*}
Another application of  Lemma \ref{moreRimage} yields
$R_0^{t_1t_2}= R^{t_1} =(\rho\otimes (\rho\circ \natural\circ S^{-1}))(\mathcal{R}_{21}^{-1}) $
and 
$R_1^{t_1t_2}= (R_{21}^{-1})^{t_1} =(\rho\otimes (\rho\circ \natural\circ S^{-1}))(\mathcal{R})$.
Hence, we can rewrite (\ref{VWform}) as 
\begin{align}\label{VWform2}
(V^*\otimes W^*)\otimes (V\otimes W)\longrightarrow \flip(\mathcal{R}_{\sigma})_{13}\cdot((V\otimes W)\otimes (V^*\otimes W^*))\cdot (\mathcal{R}_{\upsilon})_{24}
\end{align}
where  $\mathcal{R}_0 = \mathcal{R}_{21}^{-1}$ and $\mathcal{R}_{1} = \mathcal{R}$.
By (\ref{firstiso}), the map $
V^*\otimes V \rightarrow R_{\sigma}(V\otimes V^* )= \flip(\mathcal{R}_{\sigma})\cdot (V\otimes V^*)$
is an  isomorphism of left $U_q(\mathfrak{gl}_N)$-modules for both $\sigma = 0$ and $\sigma = 1$. Similarly, by (\ref{secondiso}), the map
$
W^*\otimes W \rightarrow (W\otimes W^* )R_{\upsilon}^{t_1t_2}= (W\otimes W^*)\cdot (\mathcal{R}_{\upsilon})$
is an isomorphism of right $U_q(\mathfrak{gl}_N)$-modules for $\upsilon\in \{0,1\}$. 
Thus (\ref{VWform2}) is a bimodule map with respect to the left and right actions of $U_q(\mathfrak{gl}_N)$ which means that the relations coming from the twisting map are preserved by both the left 
and right action of $U_q(\mathfrak{gl}_N)$ as desired.
\end{proof}

A linear map $f:M\rightarrow M'$ of $U_q(\mathfrak{gl}_N)$-modules is a $U_q(\mathfrak{gl}_N)$-module map provided that
$f(u\cdot m) = u\cdot f(m)$
for all $m\in M$ and $u\in U_q(\mathfrak{gl}_N)$. Right module and bi-module maps are defined similarly.  The
next result shows that the twisting map is a module map with respect to both the left and right action of $U_q(\mathfrak{gl}_N)$.

\begin{corollary}\label{corollary:invar} For each $\upsilon,\sigma\in \{0,1\}$, the twisting map $\tau_{\upsilon,\sigma}$  is a $U_q(\mathfrak{gl}_N)$ bi-module map from $\sum_{i,j,k,l}\mathbb{C}(q)\partial_{ij}\otimes t_{kl}$ to
$\sum_{i,j,k,l}\mathbb{C}(q)t_{kl}\otimes \partial_{ij}$.
\end{corollary}
\begin{proof}
It follows from Proposition \ref{prop:bimodule} and its proof  that 
\begin{align*}
\sum_{i,j,k,l}\mathbb{C}(q)(\partial_{ij}\otimes t_{kl} -\tau_{\upsilon,\sigma}(\partial_{ij}\otimes t_{kl}))
\end{align*}
is a $U_q(\mathfrak{gl}_N)$ sub-bimodule of \begin{align*}\sum_{i,j,k,l}\mathbb{C}(q) \partial_{ij}\otimes t_{kl} + \sum_{i,j,k,l}\mathbb{C}(q) t_{kl}\otimes \partial_{ij}.\end{align*}  This means that 
$u\cdot (\partial_{ij}\otimes t_{kl}) -u\cdot \tau_{\upsilon,\sigma}(\partial_{ij}\otimes t_{kl})$
is an element of this sub-bimodule  for all $u\in U_q(\mathfrak{gl}_N)$.  Note that $u$ acting on the left defines a linear map on $\mathbb{C}(q)\partial_{ij}\otimes t_{kl}$  while $\tau_{\upsilon,\sigma}$ is a linear isomorphism of $\mathbb{C}(q)\partial_{ij}\otimes t_{kl}$ onto $\mathbb{C}(q)t_{kl}\otimes \partial_{ij}.$  Hence 
$u\cdot (\partial_{ij}\otimes t_{kl}) -\tau_{\upsilon,\sigma}(u\cdot (\partial_{ij}\otimes t_{kl}))$
is the unique element of this sub-bimodule that is also contained in $u\cdot (\partial_{ij}\otimes t_{kl}) + \sum_{i,j,k,l}\mathbb{C}(q)t_{kl}\otimes \partial_{ij}$.  It follows that 
$u\cdot \tau_{\upsilon,\sigma}(\partial_{ij}\otimes t_{kl})= \tau_{\upsilon,\sigma}(u\cdot (\partial_{ij}\otimes t_{kl}).$
This shows that  $\tau_{\upsilon,\sigma}$ is a left module map.  The proof for left replaced by right is the same using right actions instead of left ones.
\end{proof}

The four twisted tensor products $\mathcal{A}_{\upsilon,\sigma}, \upsilon, \sigma\in \{0,1\}$ can be viewed as graded quantum analogs of the Weyl algebra.  Indeed, these algebras have an obvious grading 
using the fact that all the relations 
are homogeneous with respect to the degree function defined by $\deg(t_{ij})=\deg(\partial_{ij}) = 1$ for all $i,j$.  Moreover, it is straightforward to see that their relations specialize to those of the graded Weyl algebra at $q=1$ (i.e. the constant terms are dropped).  In Section \ref{section:PBWdef}, we show how to transform two of these graded algebras into non-graded ones which, in turn, can be viewed as quantum analogs of the Weyl algebra.  For now, we note that these four algebras fall into two classes via $\mathbb{C}$-algebra isomorphisms.

Write $\bar{a}$ for the image of $a\in \mathbb{C}(q)$ under the $\mathbb{C}$-automorphism of $\mathbb{C}(q)$ sending $q$ to $q^{-1}$. 
  
  \begin{proposition} \label{prop:iso} The map sending each scalar $a$ to $\bar{a}$, each $t_{ij}$ to $t_{N-i,N-j}$ and each $\partial_{kl}$ to $\partial_{N-k,N-l}$ defines a  $\mathbb{C}$-algebra isomorphism from $\mathcal{A}_{00}$ to $\mathcal{A}_{11}$ and defines a $\mathbb{C}$-algebra isomorphism from $\mathcal{A}_{01}$ to $\mathcal{A}_{10}$.  
  \end{proposition}
  \begin{proof} Note that the map sending $i$ to $N-i$ has the effect of switching order:  $i<j$ becomes $N-j<N-i$.   Note further that the set of relations in Section \ref{section:as} for $\mathcal{O}_q(\Mat_N)$ 
  is equivalent to the same relations in (i), the same first relation of (ii) and the final relation of (ii) replaced by 
     \begin{align*}t_{ij}t_{kl} -t_{kl}t_{ij} = (q-q^{-1})t_{il}t_{kj}{\rm \ for \ }i<k;j<l
  \end{align*}
  since by the first relation of (ii), $t_{il}t_{kj} = t_{kj}t_{il}$ for $i<k;j<l.$
 Hence,  applying the map sending $a$ to $\bar{a}$ and each $t_{ij}$ to $t_{N-i,N-j}$ to the relations for $\mathcal{O}_q(\Mat_N)$ as written out in (i) and (ii) in Section \ref{section:as}, yields an equivalent set of  relations.   Thus, this map extends to  a $\mathbb{C}$-algebra isomorphism on $\mathcal{O}_q(\Mat_N)$.  A similar argument shows that the map sending each scalar $a$ to $\bar{a}$ and each $\partial_{ij}$ to $\partial_{N-i,N-j}$
 extends to  a $\mathbb{C}$-algebra isomorphism on  $\mathcal{O}_q(\Mat_N)^{op}$.  
 
 It remains to show that the combination of these two  maps, which is just the map described in the proposition, sends the relations defined by the twisting map of $\tau_{00}$ to that of $\tau_{11}$ and the relations defined by the twisting map of $\tau_{01}$ to that of $\tau_{10}$. This can be seen easily from the relations (\ref{twistdefn}) derived from these twisting maps and the fact that applying the map $r\rightarrow \bar{r}$ to the entries of the  matrix ${R^{t_2}}$ yields  $((R_{21}^{-1})^{t_2})^{t_1t_2}$. \end{proof}

For $\mathcal{A}_{00}$, the relations coming from the twisting map can be expanded out as follows.  For all $a,b,c,d$, we have 
\begin{itemize}
\item[(i)] $\partial_{cb}t_{da} =t_{da} \partial_{cb}$ if $b\neq a $ and $c\neq d$.  
\item[(ii)] $\partial_{cb}t_{ca} =qt_{ca} \partial_{cb}+ \sum_{c'>c} (q-q^{-1})t_{c'a} \partial_{c'b}$ if $b\neq a $ and $c=d$.
\item[(iii)] $\partial_{ca}t_{da} = qt_{da} \partial_{ca}+\sum_{a'>a}(q-q^{-1})t_{da'} \partial_{ca'}$ if $b=a$ and $c\neq d$.
\item[(iv)]$\partial_{ca}t_{ca} = q^2t_{ca} \partial_{ca} +q \sum_{c'>c}(q-q^{-1})t_{c'a} \partial_{c'a}+q\sum_{a'>a}(q-q^{-1}) t_{ca'} \partial_{ca'}$  

  $+ \sum_{a'>a}\ \sum_{c'>c} (q-q^{-1})^2t_{c'a'} \partial_{c'a'}$ if $b=a$ and $c= d$.
\end{itemize}
Using Proposition \ref{prop:iso}, it is easy to translate these relations into ones for $\mathcal{A}_{11}$.  For the other two cases, only one of the two inequalities $a'>a$ and $c'>c$ is changed and the powers of $q$ showing up before the various summands are modified appropriately.

\subsection{Comparison with other Constructions}\label{section:comp} It is natural to ask whether the twisted tensor products of this section  correspond to a standard construction such as a quantum double or, more generally, a  double cross product (see for example \cite{KS}, Chapters 8 and 10). 
For starters, as pointed out in its proof, 
Lemma \ref{twisted-tensors}  is very similar to Proposition 8 in Section 8.2.1 of \cite{KS} used to define quantum doubles.  
Moreover, the twisted tensor product of Lemma \ref{twisted-tensors} resembles the one for double cross product bialgebras (see Proposition 26 of Section 10.2.5 in \cite{KS}). 

Despite these similarities, the algebras $\mathcal{A}_{\upsilon,\sigma}$ are not double cross product bialgebras.  This can be verified in a straightforward manner using the fact that the double crossed product  admits a tensor product coalgebra structure (see \cite{KS} Proposition 26).  In contrast, the map sending $\Delta':yx \rightarrow \sum y_{(1)}x_{(1)}\otimes y_{(2)}x_{(2)}$ is not an algebra isomorphism of $\mathcal{A}_{\upsilon,\sigma}$ to 
$\mathcal{A}_{\upsilon,\sigma}\otimes \mathcal{A}_{\upsilon,\sigma}$ where $y\in \mathcal{O}_q(\Mat_N)^{op}$ and $x\in \mathcal{O}_q(\Mat_N)$ for any choice of $\upsilon,\sigma$ in $\{0,1\}$. 
For instance, consider the case where $\sigma=\upsilon=0$,  $y=\partial_{cb}$, $x=t_{ca}$, and $a\neq b$.  Assume that $\Delta'$ is an algebra homomorphism. We have 
\begin{align*}
\Delta'(\partial_{cb}t_{ca}) &= \sum_{j,k} \partial_{cj} t_{ck} \otimes \partial_{jb}t_{ka}\cr &\in\sum_{j,k} q^{1+2\delta_{jk}}t_{ck}\partial_{cj}\otimes t_{ka}\partial_{jb}+
\sum_{(c',k',j', k'', j'')\in C}\mathbb{C}(q) t_{c'k'}\partial_{c'j'}\otimes t_{k'a}\partial_{j'b}
\end{align*}
where $C$ is the set of $5$-tuples $(c',k',j',k'',j'')$ satisfying $c'\geq c$, $k'\geq k$, $j'\geq j$, $k''\geq k, j''\geq j$ with at least one of these inequalities strict (plus other conditions such as $k'=k$ and $j'=j$ unless $k=j$, etc.).  
On the other hand,
\begin{align*}
\Delta'(qt_{ca})\Delta'(\partial_{cb}) &+ \sum_{c'>c}(q-q^{-1})\Delta'(t_{c'a})\Delta'(\partial_{c'b})
\cr&\in \sum_{j,k} qt_{ck}\partial_{cj}\otimes t_{ka}\partial_{jb}+
\sum_{c'>c}\mathbb{C}(q) t_{c'k}\partial_{c'j}\otimes t_{ka}\partial_{jb}
\end{align*}
By relation (ii) of Section \ref{section:four} and the assumption that $\Delta'$ is an algebra homomorphism, these two values should be equal. But, since
\begin{align*}
\sum_{j,k} q^{1+2\delta_{jk}}t_{ck}\partial_{cj}\otimes t_{ka}\partial_{jb}\neq  \sum_{j,k} qt_{ck}\partial_{cj}\otimes t_{ka}\partial_{jb}
\end{align*}
these two values are not equal.  Hence $\Delta'$ is not an algebra homomorphism and does not define a comultipication for $\mathcal{A}_{00}$.  A similar argument yields the same negative result for  $\mathcal{A}_{11}$; a somewhat more complicated argument estasblishes this result for the other two possibilities $\mathcal{A}_{10}$ and $\mathcal{A}_{01}$.

\section{Graded Weyl Algebras for Homogeneous Spaces}\label{section:graded-wafhs}
\subsection{Inverting a Matrix Related to $R$}
 The next two computational lemmas show how to invert $R^{t_2}$ and related matrices.  These results will allow us to use the reflection equations in the construction of certain twisting maps.
  \begin{lemma}\label{lemma:zero} Suppose that $\alpha,\upsilon\in \{0,1\}$ and $\alpha+\upsilon=1$.  For all $g, u\in \{1, \dots, n\}$ with $g\neq u$, we have 
\begin{align}\label{zeroidentity} \sum_{k=1}^nq^{-2k}(S_{\upsilon}^{t_2})^{gg}_{kk}(S_{\alpha}^{t_1})_{uu}^{kk}= 0
\end{align}
where $S_{\gamma} = R_{\mathfrak{g}}$ if $\gamma=0$ and $S_{\gamma} = (R_{\mathfrak{g}})^{-1}_{21}$ if $\gamma=1$.
\end{lemma}

\begin{proof} We prove the lemma for $\alpha=1$ and $\upsilon = 0$; the argument is easily modified for the  case $\alpha=0$ and $\upsilon=1$. 
Since $R_{\mathfrak{g}}$ is block diagonal with diagonal entries $(R,I_{n^2},I_{n^2},R)$, it is straightforward to reduce to the case where $R_{\mathfrak{g}}= R$.

Note that  $(R^{t_2})^{gg}_{kk}=r^{gk}_{kg}$. Using the formulas for $r^{ij}_{kl}$   given in Section \ref{section:images}, we have  $(R^{t_2})^{gg}_{kk} = q$ for $k=g$, $(R^{t_2})^{gg}_{kk} = (q-q^{-1})$ for $k<g$, and $(R^{t_2})^{gg}_{kk} = 0$ for $k>g$.  One checks from these explicit formulas for the entries of $R$, that  $R_{21}^{-1} = \bar{R}^{t_1t_2}$ where $a\rightarrow\bar{a}$ is the $\mathbb{C}$-algebra isomorphism of $\mathbb{C}(q)$ sending $q$ to $q^{-1}$. Hence
$((R_{21}^{-1})^{t_1})_{uu}^{kk}  = \bar{r}^{ku}_{uk}=q^{-1}$ for $k=u$,  $((R_{21}^{-1})^{t_1})_{uu}^{kk}  = -(q-q^{-1})$ for $k>u$, and $((R_{21}^{-1})^{t_1})_{uu}^{kk}  = 0$ for $k<u$.  It follows that (\ref{zeroidentity}) holds 
for $u>g$.   For $u<g$, we have 
\begin{align*}
\sum_{k=1}^n&q^{-2k}(R^{t_2})^{gg}_{kk}((R^{-1}_{21})^{t_1})_{uu}^{kk}=-q^{-2u}(q^{-2}-1) -q^{-2g}(q^{2}-1) + \sum_{u<k<g}q^{-2k}(q^{2}-1)(q^{-2}-1)
\cr&=-q^{-2u}(q^{-2}-1) +q^{-2g+2}(q^{-2}-1) - \sum_{k=u+1}^{g-1}q^{-2k}(q^{-2}-1) +\sum_{k=u+1}^{g-1}q^{-2k+2} (q^{-2}-1)
\end{align*}
which simplifies to
\begin{align*}
 - \sum_{k=u}^{g-1}q^{-2k}(q^{-2}-1) +\sum_{k=u+1}^{g}q^{-2k+2} (q^{-2}-1)=0.
\end{align*}
\end{proof}

 Let $G$ be the $({\rm rank}\ \mathfrak{g})\times ({\rm rank}\ \mathfrak{g}$) diagonal matrix with 
$k^{th}$ entry $G_{kk} = q^{-2k}$ for all $1\leq k\leq {\rm rank}\ \mathfrak{g}$.  Set $G_1 = G\times I$ and $G_2 = I\otimes G$ where here $I$ is the $({\rm rank}\ \mathfrak{g})\times ({\rm rank}\ \mathfrak{g})$ identity matrix.

 \begin{lemma}\label{lemma:allthree} We have  $(R_{\mathfrak{g}}^{t_2})^{-1} = G_1((R_{\mathfrak{g}})_{21}^{-1})^{t_1} )G_1^{-1} =G_2((R_{\mathfrak{g}})_{21}^{-1})^{t_1} )G_2^{-1}$. A similar assertion holds for  $R_{\mathfrak{g}}$ replaced by $(R_{\mathfrak{g}})^{-1}_{21}$. \end{lemma} 
\begin{proof} 
Using the formulas for the entries of $R$ (Section \ref{section:images}), we see  that $(R_{\mathfrak{g}}^{t_2})^{ab}_{ce}\neq 0$ provided one of the following two conditions hold
\begin{itemize}
\item $c=e$ and $a=b$ 
\item $c\neq e$ and $(a,b)=(c,e)$. 
\end{itemize} Moreover, in the latter case,  $(R_{\mathfrak{g}}^{t_2})^{ab}_{ce} = 1$.  
Hence, using Lemma \ref{lemma:zero}, we have \begin{align*}
[R_{\mathfrak{g}}^{t_2}&G_1((R_{\mathfrak{g}})_{21}^{-1})^{t_1}]^{ab}_{rs} = \sum_{ce}[R_{\mathfrak{g}}^{t_2}]^{ab}_{ce}[G_1]^{ce}_{ce}[((R_{\mathfrak{g}})_{21}^{-1})^{t_1}]^{ce}_{rs}\cr &=\delta_{ar}\delta_{bs}q^{-2a}[R_{\mathfrak{g}}^{t_2}]^{ab}_{ab}
[((R_{\mathfrak{g}})_{21}^{-1})^{t_1}]^{ab}_{ab} + \delta_{ab}\delta_{rs}\delta_{a\neq r}\sum_{c}[R_{\mathfrak{g}}^{t_2}]^{aa}_{cc}[((R_{\mathfrak{g}})_{21}^{-1})^{t_1}]^{cc}_{rr}
\cr&=\left\{\begin{matrix} q^{-2a}&{\rm \ if \ }a=r, b=s \cr 0&{\rm \ otherwise\ }\end{matrix}\right.
\end{align*}
It follows that $R_{\mathfrak{g}}^{t_2}G_1((R_{\mathfrak{g}})_{21}^{-1})^{t_1} = G_1$ which proves the first equality.  The other equality as well as the versions with $R_{\mathfrak{g}}$ replaced by $(R_{\mathfrak{g}})_{21}^{-1}$ follow using a similar argument.\end{proof}

 \subsection{Extended Graded Weyl Algebras}\label{section:diagonal}
We construct here graded quantum Weyl Algebras for each of the three families of symmetric pairs where the polynomial part is $\mathscr{P}_{\theta}$.  The starting point is the formation of algebras which can be viewed as extended versions of the graded Weyl algebras of Section \ref{section:four}.  In particular, they are closely related to the algebras $\mathcal{A}_{\upsilon,\sigma}$. 

Let $\mathcal{O}$ be the algebra generated by two copies of $\mathscr{P}$, the first generated by $t_{ij}$ and the second  by $t'_{ij}$, with \emph{no relations} between the $t_{ij}$ and the $t'_{ij}$.  In other words, $\mathcal{O}$ is  a quotient of the tensor algebra 
$T(M)\otimes T(M')= T(M\oplus M')$ by the ideal generated by the ideals  corresponding to  the relations of the first and second copy of 
$\mathscr{P}$.  Similarly, let $\mathcal{O}^{op}$ be the algebra generated by two copies of  $\mathscr{D}=\mathscr{P}^{op}$, the first generated by $\partial_{ij},$ and the second generated by $\partial'_{ij}$ with no relations between the $\partial_{ij}$ and the $\partial'_{ij}$.  

Given $\alpha,\beta,\upsilon, \sigma\in \{0,1\}$, define the map $\tau_{\alpha,\beta,\upsilon,\sigma}$ from $ \mathcal{O}^{op}  \otimes{\mathscr{P}}$  to  ${\mathscr{P}}\otimes \mathcal{O} ^{op}$ by 
\begin{align*}\tau_{\alpha,\beta,\upsilon,\sigma}(\partial_{ea}\otimes t_{fb}) = \tau_{\alpha,\beta}(\partial_{ea}\otimes t_{fb})
{\rm \quad and \quad }\tau_{\alpha,\beta,\upsilon,\sigma}(\partial'_{ea}\otimes t_{fb}) = \tau_{\upsilon,\sigma}(\partial'_{ea}\otimes t_{fb})
\end{align*}
for all $e,a,f,b $ in $\{1,\dots, n\}$. By Section \ref{section:four}, the maps  $\tau_{\alpha,\beta}$ and $\tau_{\upsilon,\sigma}$ define twisting maps on  $ \mathcal{O}^{op}  \otimes{\mathscr{P}}$.  Hence, $\tau_{\alpha,\beta,\upsilon,\sigma}$ is a twisting map on   $ \mathcal{O}^{op}  \otimes{\mathscr{P}}$.  Therefore, we can form the twisted tensor product 
${\mathcal{A}}_{\alpha,\beta,\upsilon,\sigma}= {\mathscr{P}}  \otimes_{\tau_{\alpha,\beta,\upsilon,\sigma} }\mathcal{O}^{op}.$
The twisting relations can be put into matrix format in analogy to the definitions of $\mathcal{A}_{\upsilon,\sigma}$ in Section \ref{section:four_twisted}.  In particular, we have  \begin{align}\label{anotherreln5}
P_2T_1 = S_{\alpha}^{t_1}T_1P_2S_{\beta}^{t_2} {\rm \ and \ }{P}'_2{T}_1 = S_{\upsilon}^{t_1}{T}_1{P}'_2S_{\sigma}^{t_2}
\end{align}
where $T$ is the matrix with entries $t_{ij}$, $P$ is the matrix with entries $\partial_{ij}$, $P'$ is the matrix with entries $\partial'_{ij}$, $S_{\gamma} = R_{\mathfrak{g}}$ if $\gamma=0$ and $ S_{\gamma} = (R_{\mathfrak{g}})_{21}^{-1}$
if $\gamma=1$.

Similarly, the map $\tau'_{\alpha,\beta,\upsilon,\sigma}$ is a twisting map where
\begin{align*}\tau'_{\alpha,\beta,\upsilon,\sigma}(\partial_{ea}\otimes t_{fb}) = \tau_{\alpha,\beta}(\partial_{ea}\otimes t_{fb})
{\rm \quad and \quad }\tau'_{\alpha,\beta,\upsilon,\sigma}(\partial_{ea}\otimes t'_{fb}) = \tau_{\upsilon,\sigma}(\partial_{ea}\otimes t'_{fb}).
\end{align*}
Hence we can form the twisted tensor product 
$
{\mathcal{A}}'_{\alpha,\beta,\upsilon,\sigma}:=\mathcal{O}\otimes_{\tau'_{\alpha,\beta,\upsilon,\sigma}}{\mathscr{D}}.
$ These relations can also be put into matrix format  as
\begin{align}\label{anotherreln6}
P_2T_1 = S_{\alpha}^{t_1}T_1P_2S_{\beta}^{t_2} {\rm \ and \ }{P}_2{T}'_1 = S_{\upsilon}^{t_1}{T}'_1{P}_2S_{\sigma}^{t_2}
\end{align}
using the same notation as above where $P'$ is the matrix with  entries $\partial'_{ij}$.

We are going to use the   algebras to ${\mathcal{A}}_{\alpha,\beta,\upsilon,\sigma}$ and ${\mathcal{A}}'_{\alpha,\beta,\upsilon,\sigma}$ in order to identify a 
twisting map on $\mathscr{D}_{\theta}\otimes \mathscr{P}_{\theta}$.  To do this, we take a step back and consider the ${\rm rank}(\mathfrak{g})\times {\rm rank}(\mathfrak{g})$ matrices $\tilde{D}$ and $\tilde X$ with 
$ij$ entries $\tilde{d}_{ij}$ and $\tilde{x}_{ij}$ respectively where the $\tilde{d}_{ij}$ and $\tilde{x}_{ij}$ can be viewed as independent non-commuting variables. Alternatively, in the discussion below, we will be taking $\tilde{d}_{ij}$ to be either $d'_{ij}$ or $d_{ij}$ and a similar statement holds for $\tilde{x}_{ij}$.  The following lemma allows us to express what will be come the desired twisting map both on the element level and in matrix form. 

\begin{lemma} \label{lemma:matrixform} The set of equalities 
\begin{align}\label{firsteqn}
\tilde{d}_{ab}\tilde{x}_{ef} = \sum_{r,w,p,q,x,y,m,l}(S_{\alpha}^{t_2})_{xq}^{wr}(S_{\alpha}^{t_2})_{ma}^{pq}(S_{\upsilon}^{t_2})^{xy}_{fl}(S_{\upsilon}^{t_2})^{ml}_{eb}\tilde{x}_{pw}\tilde{d}_{ry}
\end{align}
for all $a,b,e,f$ in $\{1, \dots, n\}$ is equivalent to the matrix equality 
\begin{align}\label{thirdeqn}
\tilde{D}_2((S_{\upsilon}^{t_1})^{-1})^{t_2}\tilde{X}_1=S_{\alpha}^{t_1}\tilde{X}_1S_{\alpha}\tilde{D}_2S_{\upsilon}^{t_2}.
\end{align}
where $S_{\gamma} = R_{\mathfrak{g}}$ if $\gamma=0$ and $S_{\gamma} = (R_{\mathfrak{g}})_{21}^{-1}$ if $\gamma = 1$.
\end{lemma}
\begin{proof}
Reordering the right hand side  of (\ref{firsteqn}) yields
\begin{align*}
\tilde{d}_{ab}\tilde{x}_{ef} &=\sum_{r,w,p,q,x,y,m,l} (S_{\upsilon}^{t_1})_{ml}^{eb}(S_{\alpha}^{t_1})^{ma}_{pq}(\tilde{X}_1)^{pq}_{wq}(S_{\alpha})^{wq}_{xr}(\tilde{D}_2)^{xr}_{xy}(S_{\upsilon}^{t_2})^{xy}_{fl}\cr&=\sum_{m,l} (S_{\upsilon}^{t_1})_{ml}^{eb}[S_{\alpha}^{t_1}\tilde{X}_1S_{\alpha}\tilde{D}_2S_{\upsilon}^{t_2})]^{ma}_{fl}\end{align*}
Applying $ (S_{\upsilon}^{t_1})^{-1}$ to both sides gives us
\begin{align}\label{secondeqn}
\sum_{e,b}((S_{\upsilon}^{t_1})^{-1})_{eb}^{ml}\tilde{d}_{ab}\tilde{x}_{ef}=
[S_{\alpha}^{t_1}X_1S_{\alpha}D'_2S_{\upsilon}^{t_2})]^{ma}_{fl}
\end{align}
We can rewrite the left hand side as 
\begin{align*}
\sum_{eb}((S_{\upsilon}^{t_1})^{-1})_{eb}^{ml}\tilde{d}_{ab}\tilde{x}_{ef}=\sum_{eb}(\tilde{D}_2)_{mb}^{ma}((S_{\upsilon}^{t_1})^{-1})_{eb}^{ml}(\tilde{X}_1)^{el}_{fl} = [\tilde{D}_2((S_{\upsilon}^{t_1})^{-1})^{t_2}\tilde{X}_1]^{ma}_{fl}
\end{align*}
Plugging the last term on the right of the above equality into the left hand side of (\ref{secondeqn}) yields the desired matrix form (\ref{thirdeqn}).\end{proof}

Both ${O}$ and ${O}^{op}$ inherit the structure of a $U_q(\mathfrak{g})$-bimodule algebra from  $\mathscr{P}$ and $\mathscr{D}$. Here, we assume the action of $U_q(\mathfrak{g})$ on $\mathscr{P}'$ and $\mathscr{D}'$  is exactly the same as for the first copies with $t_{ij}$ replaced by $t'_{ij}$ everywhere and similarly, each $\partial_{ij}$ replaced by $\partial'_{ij}$.   Since the twisting maps $\tau_{\alpha,\beta}$, $\alpha,\beta\in \{0,1\}$ induce relations 
that are bi-invariant with respect to these actions, the same is true for the twisting maps $\tau_{\alpha,\beta,\upsilon,\sigma}$ and $\tau'_{\alpha,\beta,\upsilon,\sigma}$, 
$\alpha,\beta,\upsilon,\sigma\in \{0,1\}$.  Hence, the algebras ${\mathcal{A}}_{\alpha,\beta,\upsilon,\sigma}$ and  ${\mathcal{A}}'_{\alpha,\beta,\upsilon,\sigma}$ inherit a $U_q(\mathfrak{g})$-bimodule structure from the subalgebras used to construct them.

Note that the algebra $\mathscr{P}_{\theta}$ embeds inside ${\mathcal{A}}_{\alpha,\beta,\upsilon,\sigma}$  via the inclusion of $\mathscr{P}_{\theta}$ inside $\mathscr{P}$ and $\mathscr{D}_{\theta}$ embeds inside of 
${\mathcal{A}'}_{\alpha,\beta,\upsilon,\sigma}$ via the inclusion of   $\mathscr{D}_{\theta}$ inside $\mathscr{D}$ for each choice of $\alpha,\beta,\upsilon,$ and $\sigma$.  We also consider  subalgebras  of invariants  inside ${\mathcal{O}}$ and ${\mathcal{O}}^{op}$ with respect to the action of $\mathcal{B}_{\theta}$ that are closely related to $\mathscr{P}_{\theta}$ and  $\mathscr{D}_{\theta}$.  In particular, 
define
 $x_{ij}'\in {\mathcal{O}}$ and $d_{ij}'\in {\mathcal{O}}^{op}$ by 
 \begin{align}\label{xprime-defn1} x'_{ij} = \sum_{r,s}t'_{ir}J_{r,s}t_{js}
 {\rm \  and \ }
 d'_{ij} = \sum_{r,s}q^{-2\hat{s}}\partial_{ir}J_{r,s}\partial'_{js}
 \end{align}
 where $\hat{s} = s$ in Types AI and AII; for diagonal type  $\hat{s} =s$ for $s\leq n$, and   $\hat{s} = s - n$  when $s\geq n+1$. 
 Let $\mathscr{P}'_{\theta}$ be the subalgebra of ${\mathcal{O}}$ generated by the $x'_{ij}$ and let $\mathscr{D}'_{\theta}$ be the subalgebra of ${\mathcal{O}}^{op}$  generated by the $d'_{ij}$.  
 
 Let $X'$ be the ${\rm rank}(\mathfrak{g})\times {\rm rank}(\mathfrak{g})$ matrix with $ij$ entry equal to $x'_{ij}$ and let $D'$ be the ${\rm rank}(\mathfrak{g})\times {\rm rank}(\mathfrak{g})$ matrix with $ij$ entry equal to $d'_{ij}$.  It is straightforward to check that $X'=T'JT^t$ while $D' = PGJ(P')^t$  for all three families where $J=J(1) $ is defined as in Section \ref{section:ie} and $G$ is the diagonal matrix defined in Section \ref{section:diagonal}.

It follows from  the arguments in Section \ref{section:ie} that each $x'_{ij}$ and each $d'_{ij}$ is right invariant with respect to the action of $\mathcal{B}_{\theta}$. However, since there are no relations satisfied between the $t_{ij}$ and the $t'_{ij}$, we see that  the algebra $\mathscr{P}'_{\theta}$ generated by the $t'_{ij}$, $1\leq i,j\leq n$, is a free algebra with these generators.  The same assertion holds for  $\mathscr{D}'_{\theta}$ with the $t'_{ij}$ replaced by the $d'_{ij}$.

Recall the reflection equations (\ref{reflection-equation}).  It is straightforward to check that $J$ is a nonzero scalar multiple of $J^{-1}$ for all three families.  Hence, we also have
\begin{align}\label{reflection-equation2}
J_2R_{\mathfrak{g}}J_1 R_{\mathfrak{g}}^{t_1} =  R_{\mathfrak{g}}^{t_1} J_1 R_{\mathfrak{g}}J_2
\end{align}
Given a matrix $M$ with entries in $\mathbb{C}(q)$, write $\bar M$ for the image of $M$ under the $\mathbb{C}$-algebra map sending $q$ to $q^{-1}$.  It is also straightforward to check that $\bar J^t $ is a nonzero scalar multiple of $J$ and ${\bar{R}_{\mathfrak{g}}}^{t_1t_2} = (R_{\mathfrak{g}})_{21}^{-1}$. Hence, applying the transpose and bar map to both sides of (\ref{reflection-equation}) and (\ref{reflection-equation2}) yields the same equations with $R_{\mathfrak{g}}$ replaced by $(R_{\mathfrak{g}})_{21}^{-1}$ everywhere.

\begin{proposition}\label{prop:twotwistingmaps} Set  $S_{\gamma} = R_{\mathfrak{g}}$ if $\gamma = 0$ and $S_{\gamma} = (R_{\mathfrak{g}})_{21}^{-1}$ if $\gamma = 1$. If $\beta+\sigma= 1$, then the subalgebra of ${\mathcal{A}}_{\alpha,\beta,\upsilon,\sigma}$ generated by $\mathscr{P}_{\theta}$ and $\mathscr{D}'_{\theta}$ satisfies the relations 
\begin{align*} 
d'_{ab}x_{ef} = \sum_{r,w,p,q,x,y,m,l}(S_{\alpha}^{t_2})_{xq}^{wr}(S_{\alpha}^{t_2})_{ma}^{pq}(S_{\upsilon}^{t_2})^{xy}_{fl}(S_{\upsilon}^{t_2})^{ml}_{eb}x_{pw}d'_{ry}
\end{align*}
for all $a,b,e,f$.
Similarly, if $\beta+\sigma = 1$, then the subalgebra of ${\mathcal{A}}'_{\alpha,\beta,\upsilon,\sigma}$ generated by $\mathscr{D}_{\theta}$ and $\mathscr{P}'_{\theta}$ satisfies the relations 
\begin{align*} 
d_{ab}x'_{ef} =\sum_{r,w,p,q,x,y,m,l}(S_{\alpha}^{t_2})_{xq}^{wr}(S_{\alpha}^{t_2})_{ma}^{pq}(S_{\upsilon}^{t_2})^{xy}_{fl}(S_{\upsilon}^{t_2})^{ml}_{eb}x'_{pw}d_{ry}
\end{align*}
for all $a,b,e,f$.
\end{proposition}
\begin{proof}  
Applying  $(S_{\upsilon}^{t_1})^{-1}$ to both sides of the second equation of  (\ref{anotherreln5}) yields 
\begin{align*}
(S_{\upsilon}^{t_1})^{-1}{P'_2}{T_1} = {T_1}{P'_2}S_{\sigma}^{t_2}
\end{align*}
which is equivalent to 
\begin{align}\label{anotherreln10}
{(P')_2^t}((S_{\upsilon}^{t_1})^{-1})^{t_2}{T'_1} = T_1S_{\sigma}{(P')_2^t}
\end{align}
where $(P')^t_2 = Id\otimes (P')^t$.
Note that applying the  map $t=t_1t_2$ to both sides of  the second equality of (\ref{anotherreln5}) becomes 
\begin{align}
\label{anotherreln11}
(P')_2^tT_1^t = S_{\sigma}^{t_1}T^t_1{(P')^t_2}S_{\epsilon}^{t_2}
\end{align}
where here we are using the fact that $T^t_1$ and ${(P')^t_2}$ commute with each other.

We prove the proposition by using the matrix form of the equations given by Lemma \ref{lemma:matrixform}.
In particular, expanding out the left hand side of (\ref{thirdeqn}) yields
\begin{align*} D'_2((S_{\upsilon}^{t_1})^{-1})^{t_2}X_1 &= \left(P_2G_2J_2(P')_2^t\right) ((S_{\upsilon}^{t_1})^{-1})^{t_2}\left(T_1J_1T_1^t \right)
\cr &=P_2G_2J_2\left((P')^t_2 ((S_{\upsilon}^{t_1})^{-1})^{t_2}T_1\right)J_1T_1^t
\end{align*}
Replacing the middle term by the right hand side of (\ref{anotherreln10}) and using the fact that $T'_1$ commutes with $G_2J_2$ gives us 
\begin{align*}
D'_2((S_{\upsilon}^{t_1})^{-1})^{t_2}X_1 
&= P_2T_1G_2J_2 S_{\sigma}J_1(P')_2^tT_1^t
\end{align*}
Using (\ref{anotherreln5}) and (\ref{anotherreln11})  to rewrite the first two and last two matrices yields
\begin{align*}
D'_2((S_{\upsilon}^{t_1})^{-1})^{t_2}X_1&=\left(S_{\alpha}^{t_1}T_1P_2S_{\beta}^{t_2}\right)G_2J_2 S_{\sigma}J_1\left(S_{\sigma}^{t_1}{T_1}^t(P')_2^tS_{\upsilon}^{t_2}\right)\cr
&=S_{\alpha}^{t_1}T_1P_2S_{\beta}^{t_2}G_2\left(J_2 S_{\sigma}J_1S_{\sigma}^{t_1}\right){T_1}^t(P')_2^tS_{\upsilon}^{t_2}
\end{align*}
Using the reflection equation as presented before the proposition, we can replace $J_2 S_{\sigma}J_1S_{\sigma}^{t_1}$ with $S_{\sigma}^{t_1}J_1 S_{\sigma}J_2$
which yields 
\begin{align*} D'_2((S_{\upsilon}^{t_1})^{-1})^{t_2}X_1 &= S_{\alpha}^{t_1}T_1P_2S_{\beta}^{t_2}G_2\left(S_{\sigma}^{t_1}J_1 S_{\sigma}J_2\right)T_1^t(P')_2^tS_{\upsilon}^{t_2}.
\end{align*}
Since $\beta +\sigma = 1$, Lemma \ref{lemma:allthree} ensures that $S_{\beta}^{t_2}G_2S_{\sigma}^{t_1} = G_2$ and so the above becomes
\begin{align*} D'_2((S_{\upsilon}^{t_1})^{-1})^{t_2}X_1 &= S_{\alpha}^{t_1}T_1P_2G_2J_1 S_{\sigma}J_2{T_1}^t(P')_2^tS_{\upsilon}^{t_2}
=S_{\alpha}^{t_1}T_1J_1P_2G_2 S_{\sigma}{T_1}^tJ_2(P')_2^tS_{\upsilon}^{t_2}
\end{align*}
where  the last equality follows from the fact that $J_1$ commutes with $P_2G_2$.

The key to finishing the proof is showing that 
 \begin{align}\label{key}
G_2S_{\sigma} = ((S_{\beta}^{t_2})^{-1})^{t_1}G_2
 \end{align} when $\beta+\sigma = 1$. 
By Lemma \ref{lemma:allthree}, we have  $(R_{21}^{-1})^{t_2}G_2R^{t_1} = G_2$.
Multiplying both sides by $((R_{21}^{-1})^{t_2})^{-1}$ produces $G_2R^{t_1} = ((R_{21}^{-1})^{t_2})^{-1}G_2$.  Applying $t_1$ to both sides yields the equality
$G_2R= ((((R_{21}^{-1})^{t_2})^{-1})^{t_1}G_2 .$   This establishes (\ref{key}) when $S_{\sigma} = R$ and $S_{\beta} = R_{21}^{-1}$.  This immediately extends to the diagonal type using the fact that $R(\mathfrak{g}) = (R,I_{n^2},I_{n^2},R)$ in that setting. The other case ($\sigma = 1$ and $\beta=0$) is similar with the roles of $R$ and $R_{21}^{-1}$ interchanged.

Now using (\ref{key}) we see that 
\begin{align*}
S_{\alpha}^{t_1}T_1J_1P_2((S_{\beta}^{t_2})^{-1})^{t_1}G_2 {T_1}^tJ_2{P'_2}^tS_{\upsilon}^{t_2}=S_{\alpha}^{t_1}T_1J_1T_1^t S_{\alpha}P_2G_2J_2{P'_2}^tS_{\upsilon}^{t_2}
=S_{\alpha}^{t_1}X_1S_{\alpha}D_2'S_{\upsilon}^{t_2}
\end{align*}
as desired.  This proves the first set of equalities.  The argument for the second set is very similar using two copies of the polynomial part and only one copy of the partial  part.\end{proof}

\subsection{Four Twisted Tensor Products for Homogeneous Spaces}\label{section:four_twisted}
We introduce a family of twisting maps, and refer to each of them by $\tau^{\theta}_{\alpha,\upsilon}$ for $\alpha,\upsilon\in \{0,1\}$, as follows.  Let $\tau^{\theta}_{\alpha,\upsilon}$ be the map from $\sum_{i,j,k,l}\mathbb{C}(q)\tilde{d}_{ij}\otimes \tilde{x}_{kl}$ to $\sum_{i,j,k,l}\mathbb{C}(q)\tilde{x}_{kl}\otimes\tilde{d}_{ij}$ defined by 
\begin{align} \label{twistingagain}
{\tau}^{\theta}_{\alpha,\upsilon}(\tilde{d}_{ab}\otimes \tilde{x}_{ef}) = \sum_{r,w,p,q,x,y,m,l}(S_{\alpha}^{t_2})_{xq}^{wr}(S_{\alpha}^{t_2})_{ma}^{pq}(S_{\upsilon}^{t_2})^{xy}_{fl}(S_{\upsilon}^{t_2})^{ml}_{eb}\tilde{x}_{pw}\otimes \tilde{d}_{ry}
\end{align}
for all $a,b,e,f$.  
Here, we let $\tilde{d}_{ij}$ be either $d'_{ij}$ or $d_{ij}$ and similarly, $\tilde{x}_{ij}$ is either $x_{ij}$ or $x'_{ij}$.   Since both $\mathscr{D}_{\theta}'$ and $\mathscr{P}_{\theta}'$ are free algebras, the map ${\tau}^{\theta}_{\alpha,\upsilon}$ extends to a twisting map from 
$\mathscr{D}'_{\theta}\otimes \mathscr{P}'_{\theta}$ to $\mathscr{P}'_{\theta}\otimes \mathscr{D}'_{\theta}$.  Using this twisting map, we obtain a twisted tensor product 
$\mathscr{P}'_{\theta}\otimes_{{{\tau}^{\theta}_{\alpha,\upsilon}}}\mathscr{D}'_{\theta}$.

Assume that $\beta+\sigma=1$.  Recall that by construction,  $\mathcal{A}_{\alpha,\beta,\upsilon,\sigma}$ is a twisted tensor product.  Since $\mathscr{P}_{\theta}$ embeds in the first component and $\mathscr{D}'_{\theta}$ in the second, we have that multiplication induces an injection 
\begin{align*}
\mathscr{P}_{\theta}\otimes \mathscr{D}'_{\theta} \hookrightarrow \mathcal{A}_{\alpha,\beta,\upsilon,\sigma}.
\end{align*}
By Proposition \ref{prop:twotwistingmaps}, the subalgebra of $\mathcal{A}_{\alpha,\beta,\upsilon,\sigma}$ generated by $\mathscr{P}_{\theta}$ and $\mathscr{D}'_{\theta}$ is isomorphic to the  twisted tensor product 
$\mathscr{P}_{\theta}\otimes_{\tau^{\theta}_{\alpha,\upsilon}}\mathscr{D}'_{\theta}$.  Similarly, the subalgebra of $\mathcal{A}'_{\alpha,\beta,\upsilon,\sigma}$ generated by $\mathscr{D}_{\theta}$ and $\mathscr{P}'_{\theta}$ is isomorphic to the  twisted tensor product 
$\mathscr{P}'_{\theta}\otimes_{\tau^{\theta}_{\alpha,\upsilon}}\mathscr{D}_{\theta}$.

Since $\mathscr{P}'_{\theta}$ is a free algebra generated by the $x'_{ij}$, we can realize $\mathscr{P}_{\theta}$ as a quotient of $\mathscr{P}'_{\theta}$.  The defining ideal $\mathcal{I}_{\theta}$ is generated by the elements described in Proposition \ref{prop:Palgebra} where each $\tilde{x}_{ij}$ is replaced by $x'_{ij}$.  Similarly, $\mathscr{D}_{\theta}$ can be realized as a quotient 
$\mathscr{D}'_{\theta}/\mathcal{J}_{\theta}$ where $\mathcal{J}_{\theta}$ is  the ideal of relations generated by elements read off of Proposition \ref{prop:Dalgebra}.

\begin{proposition}\label{prop:twisting-homcase} Let $\tau$ be one of the twisting maps $\tau^{\theta}_{\alpha,\upsilon}$ for some choice of $\alpha,\upsilon\in \{0,1\}$.  Define two algebra maps by
\begin{itemize}
\item  $\varphi_1: \mathscr{P}'_{\theta}\otimes_{\tau}\mathscr{D}'_{\theta} \rightarrow \mathscr{P}_{\theta}\otimes_{\tau}\mathscr{D}'_{\theta}$  is the identity on $\mathscr{D}'_{\theta}$ and sends each $x'_{ij}$ to $x_{ij}$
\item   $\varphi_2: \mathscr{P}'_{\theta}\otimes_{\tau}\mathscr{D}'_{\theta} \rightarrow \mathscr{P}'_{\theta}\otimes_{\tau}\mathscr{D}_{\theta}$ sends $d'_{ij}$ to $d_{ij}$ and is the identity on $\mathscr{P}'_{\theta}$.
\end{itemize} The kernel  of $\varphi_1$ is generated by $\mathcal{I}_{\theta}$ and the kernel of $\varphi_2$ is generated by $\mathcal{J}_{\theta}$. 
Moreover, both  maps are  left $U_q(\mathfrak{g})$-module and right $\mathcal{B}_{\theta}$-module maps and all elements under consideration are invariant with respect to the right action of $\mathcal{B}_{\theta}$.  
\end{proposition}
\begin{proof} The fact that these two maps are algebra homomorphisms with the desired kernels follows immediately from the discussion preceding the lemma on various twisted tensor products.   The last assertion follows from the fact that $\mathcal{A}_{\alpha,\beta,\upsilon,\sigma}$ and 
 $\mathcal{A}'_{\alpha,\beta,\upsilon,\sigma}$ are  both two-sided $U_q(\mathfrak{g})$-modules and $\mathscr{P}_{\theta}, \mathscr{P}'_{\theta}, \mathscr{D}_{\theta}$, and $\mathscr{D}'_{\theta}$ are all left $U_q(\mathfrak{g})$-submodules and right $\mathcal{B}_{\theta}$-submodules with the latter action trivial.   Moreover, by construction, these module structures for each of the algebras under consideration are defined in the same way and so compatible with these algebra maps. \end{proof}
 
 A consequence of the above proposition is that the ideal $\mathcal{I}_{\theta}$ ``commutes" with elements in $\mathscr{D}'_{\theta}$.  More precisely, we have 
$d'\mathcal{I}_{\theta} \subset \mathcal{I}_{\theta} \mathscr{D}'_{\theta}$
for all $d'\in \mathscr{D}'_{\theta}$ and, moreover, 
$\mathscr{D}'_{\theta} \mathcal{I}_{\theta} = \mathcal{I}_{\theta}\mathscr{D}'_{\theta}$.
This is because the twisting map $ \tau^{\theta}_{\alpha,\upsilon}$ defines an isomorphism of $\mathcal{D'}_{\theta}\otimes \mathscr{P}'_{\theta}$ onto $\mathscr{P}'_{\theta}\otimes \mathscr{D}'_{\theta}$ which induces an isomorphism by the twisting map of  $\mathcal{D'}_{\theta}\otimes \mathscr{P}_{\theta}$ onto $\mathscr{P}_{\theta}\otimes \mathscr{D}'_{\theta}$.  Hence
$\tau(\mathscr{D}'_{\theta}\otimes \mathcal{I}_{\theta}) = \mathcal{I}_{\theta}\otimes \mathscr{D}'_{\theta}$.
Analogous results hold for $\mathcal{J}_{\theta}$.  In particular, we have 
$\mathcal{J}_{\theta}\mathscr{P}'_{\theta} = \mathscr{P}'_{\theta} \mathcal{J}_{\theta}.$
It follows that the ideal in $\mathscr{P}'_{\theta}\otimes_{ \tau^{\theta}_{\alpha,\upsilon}}\mathscr{D}'_{\theta}$ generated by $\mathcal{I}_{\theta}$ and $\mathcal{J}_{\theta}$ takes the form 
$\mathcal{I}_{\theta}\mathcal{D'}_{\theta} + \mathscr{P}'_{\theta}\mathcal{J}_{\theta}$
and is isomorphic via multiplication as vector spaces to 
$
\mathcal{I}_{\theta}\otimes \mathscr{D}'_{\theta} + \mathscr{P}'_{\theta}\otimes\mathcal{J}_{\theta}.  
$   
This ensures that the quotient $\mathscr{P}'_{\theta}\otimes_{\tau^{\theta}_{\alpha,\upsilon}} \mathscr{D}'_{\theta}$ by the ideal $\mathcal{I}_{\theta}\mathcal{D'}_{\theta} + \mathscr{P}'_{\theta}\mathcal{J}_{\theta}$
 is itself a twisted tensor product of $\mathscr{P}_{\theta}$ and $\mathscr{D}_{\theta}$.  Thus we have the following result on twisted tensor products of $\mathscr{D}_{\theta}$ and $\mathscr{P}_{\theta}$.

\begin{theorem}\label{theorem:graded-hom}
For each $\alpha,\upsilon\in \{0,1\}$, the  map $\tau^{\theta}_{\alpha,\upsilon}$ 
defined by 
\begin{align*} {\tau}^{\theta}_{\alpha,\upsilon}(d_{ab}\otimes x_{ef}) = \sum_{r,w,p,q,x,y,m,l}(S_{\alpha}^{t_2})_{xq}^{wr}(S_{\alpha}^{t_2})_{ma}^{pq}(S_{\upsilon}^{t_2})^{xy}_{fl}(S_{\upsilon}^{t_2})^{ml}_{eb}x_{pw}\otimes d_{ry}
\end{align*}
is
 a twisting map from $\mathscr{D}_{\theta}\otimes \mathscr{P}_{\theta}$ to $\mathscr{P}_{\theta}\otimes \mathscr{D}_{\theta}$ where $S_0 = R_{\mathfrak{g}}$ and $S_1 = (R_{\mathfrak{g}})_{21}^{-1}$. 
Moreover, $\mathscr{P}_{\theta}\otimes_{{\tau^{\theta}_{\alpha,\upsilon}}}\mathscr{D}_{\theta}$ inherits the structure of a left $U_q(\mathfrak{g})$-module algebra and trivial right $\mathcal{B}_{\theta}$-module algebra from the subalgebras 
$\mathscr{P}_{\theta}$ and $\mathscr{D}_{\theta}$.
\end{theorem}
\begin{proof} This follows from the discussion above combined with the fact  that the twisting map $\tau^{\theta}_{\alpha,\upsilon}$  preserves the left $U_q(\mathfrak{g})$-module and right $\mathcal{B}_{\theta}$-module structures of $\mathcal{I}_{\theta}$ and $\mathcal{J}_{\theta}$.
\end{proof}

Note that Theorem \ref{theorem:graded-hom} gives  us four graded Weyl algebras for each of  the three types of homogeneous spaces $\mathscr{P}_{\theta}$.  This is similar to the situation for  the original quantum graded Weyl algebras $\mathcal{A}_{\alpha, \upsilon}$ associated to $\mathscr{P}$. In analogy to the original case, we set $\mathcal{A}^{\theta}_{\alpha,\upsilon}
=\mathscr{P}_{\theta}\otimes_{\tau^{\theta}_{\alpha,\upsilon}}\mathscr{D}_{\theta}$ for each $\alpha,\upsilon\in \{0,1\}$. The relations relating the $x_{ij}$ and the $d_{ij}$ can easily be read off of the definition of the twisting map $\tau^{\theta}_{\alpha,\upsilon}$ in Theorem \ref{theorem:graded-hom}.  Namely, we have 
\begin{align}\label{relnsforxd} d_{ab}x_{ef} = \sum_{r,w,p,q,x,y,m,l}(S_{\alpha}^{t_2})_{xq}^{wr}(S_{\alpha}^{t_2})_{ma}^{pq}(S_{\upsilon}^{t_2})^{xy}_{fl}(S_{\upsilon}^{t_2})^{ml}_{eb}x_{pw}d_{ry}
\end{align}
for all $a,b,e,f\in \{1, \dots, {\rm rank}(\mathfrak{g})\}$  and where $ S_0 = R_{\mathfrak{g}}$ and $S_1 = (R_{\mathfrak{g}})_{21}^{-1}$.
Using Lemma \ref{lemma:matrixform}, these relations can be put in matrix form in a way that resembles the reflection equations, or, more precisely, the relations for $\mathscr{P}_{\theta}$.  
For example, when $\alpha=\upsilon=0$, we have 
\begin{align*}
D_2((R_{\mathfrak{g}}^{t_1})^{-1})^{t_2} X_1 R_{\mathfrak{g}}^{t_2} = R_{\mathfrak{g}}^{t_1}X_1R_{\mathfrak{g}}D_2
\end{align*}

In the next result, we show that for the diagonal type, we recover the graded Weyl algebras $\mathcal{A}_{\alpha,\upsilon}$, $\alpha,\upsilon\in \{0,1\}$.
 
 \begin{corollary}\label{diagiso}In the diagonal type, the subalgebra generated by $\mathscr{P}_{\theta}$ and 
$\mathscr{D}_{\theta}$ inside  of  $\mathcal{A}_{\alpha,\beta,\upsilon,\sigma}$, for $\beta+\sigma=1$, is isomorphic to the  graded quantum Weyl algebra 
$\mathcal{A}_{\alpha,\upsilon}$ via the map  $\psi$ sending each $x_{i,j+n}$ to $t_{ij}$ and each $d_{i,j+n}$ to $\partial_{ij}$.  Moreover  this map is a  $U_q(\mathfrak{g})$-bimodule algebra  isomorphism where
\begin{align*}\psi(a \cdot u) =a\cdot \psi(u) {\rm \ and \ } \psi(u\cdot a^{\natural}) = \gamma(a)\cdot \psi(u)
\end{align*} 
for all $a\in U_q(\mathfrak{gl}_n)$ and $u$ in the subalgebra generated by $\mathscr{P}_{\theta}$ and 
$\mathscr{D}_{\theta}$ where $\gamma$ is the map from the first copy of $U_q(\mathfrak{gl}_n)$ to the second defined by $\gamma(E_r) = E_{n+r}, 
\gamma(F_r)=F_{n+r},$ and $\gamma(K_{\epsilon_s}) = K_{\epsilon_{n+s}}$ for all $r,s\in  \{1, \dots, n\}$.  \end{corollary}
\begin{proof} By Proposition \ref{prop:Palgebra},  the map sending $x_{i,n+j}$ to $t_{ij}$ is an isomorphism of $\mathscr{P}_{\theta}$ onto $\mathcal{O}_q(\Mat_n)$. Similarly, by Proposition \ref{prop:Dalgebra}, the map sending $d_{i,n+j}$ to $\partial_{ij}$ is an isomorphism of $\mathscr{D}_{\theta}$ onto $\mathcal{O}_q(\Mat_n)^{op}$.   

We show below that 
\begin{align}\label{dxreln}
d_{a, b+n}x_{e,f+n}=\sum_{r,m,j,l} (R_{\alpha}^{t_2})^{rj}_{ea}(R_{\upsilon}^{t_2})^{ml}_{fb}x_{r,m+n}d_{j,l+n}\end{align} for all $a,b,e,f$ in $\{1,\dots, n\}$.  In other words, 
the elements $d_{e,a+n}, x_{f,b+n}$ satisfy the same twisted tensor product rules as $\partial_{ea}$ and $t_{fb}$.   Thus, 
 the  map defined by sending $\partial_{ij}$ to $d_{i,j+n}$ and each $t_{ij}$ to $x_{i,j+n}$  is an algebra homomorphism.  To see that  this is an isomorphism, we note that by Theorem \ref{theorem:graded-hom}, the map from 
$\mathscr{P}_{\theta}\otimes \mathscr{D}_{\theta}$ to $\mathcal{A}^{\theta}_{\alpha,\upsilon}$ induced by multiplication is a vector space isomorphism.   Therefore, multiplication induces a vector space isomorphism from $\mathscr{P}_{\theta}\otimes \mathscr{D}_{\theta}$ into  
$\mathcal{A}_{\alpha,\beta,\upsilon,\tau}$.   Since  (\ref{dxreln}) corresponds to the defining twisting map for the twisted tensor product $\mathcal{A}_{\alpha,\upsilon}$, there are no additional relations then those generated by the ideal including the 
relations satisfied by $\mathscr{P}_{\theta}$, those satisfied by $\mathscr{D}_{\theta}$ and the above relation relating $d_{ea}$ and $x_{fb}$.  Thus the proposition follows once we establish (\ref{dxreln}).

By Theorem \ref{theorem:graded-hom}, we have 
\begin{align}\label{dxreln10}
d_{a, b+n}x_{e,f+n}=\sum_{r,w,p,q,x,y,m,l} (S_{\alpha}^{t_2})^{wr}_{xq}(S_{\alpha}^{t_2})^{pq}_{ma}(S_{\upsilon}^{t_2})^{xy}_{f+n,l}(S_{\upsilon}^{t_2})^{ml}_{e,b+n}x_{pw}d_{ry}\end{align} for all $a,b,e,f$ in $\{1,\dots, n\}$ where  the sum is over elements $r,w,p,q,x,y,m,l$ in $\{1, \dots, 2n\}$.  Since $e,b$ are both in $\{1, \dots, n\}$, we must have $m=e$ and $l=b+n$ in order for $(S_{\upsilon}^{t_2})^{m,l}_{e,b+n}$ to be nonzero. Moreover, it follows from the definition of $R_{\mathfrak{g}}$ in the diagonal type that  $(S_{\upsilon}^{t_2})^{e,b+n}_{e,b+n}=I^{e,b+n}_{e,b+n} = 1$. Hence (\ref{dxreln10}) becomes
\begin{align}\label{dxreln20}
d_{a, b+n}x_{e,f+n}=\sum_{r,w,x,p,q,y} (S_{\alpha}^{t_2})^{wr}_{xq}(S_{\alpha}^{t_2})^{pq}_{ea}(S_{\upsilon}^{t_2})^{xy}_{f+n,b+n}x_{pw}d_{ry}\end{align}
Since $e\in \{1,\dots, n\}$, $x_{e,w}$ is nonzero if and only if $w\in \{n+1,\dots, 2n\}$.  On the other hand, $(S_{\upsilon}^{t_2})^{x,y}_{f+n,b+n}x_{e,w}\neq 0$ implies that both $x$ and $y$ are also in $\{n+1,\dots, 2n\}$.  Hence $d_{ry}\neq 0$ implies that $r\in \{1, \dots, n\}$.  With $w\in \{n+1,\dots, 2n\}$ and $r\in \{1,\dots, n\}$, we get $(S_{\alpha}^{t_2})^{w,r}_{x,q}\neq 0$ if and only if $w=x$, $r=q$, and $(S_{\alpha}^{t_2})^{w,r}_{w,r}\neq 1$.  Finally, since both $e$ and $a$ are in $\{1, \dots, n\}$, the same must be true for $p$ and $q$ in order for $(S_{\alpha}^{t_2})^{p,q}_{e,a}$ to be nonzero.  Hence (\ref{dxreln20}) becomes
\begin{align}\label{dxreln22}
d_{a, b+n}x_{e,f+n}=\sum_{w',p,r,y'} (S_{\alpha}^{t_2})^{pr}_{ea}(S_{\upsilon}^{t_2})^{w'+n,y'+n}_{f+n,b+n}x_{p,w'+n}d_{r,y'+n}\end{align}
Now $(S_{\upsilon}^{t_2})^{w'+n,y'+n}_{f+n,b+n} = (R_{\upsilon}^{t_2})^{w',y'}_{f,b}$ and 
$(S_{\alpha}^{t_2})^{pq}_{ea}=(R_{\alpha}^{t_2})^{pq}_{ea}$ for all values of 
$a,e,p,q,w',y',f,b$ in $\{1,\dots, n\}$.  Thus (\ref{dxreln22}) is the same as (\ref{dxreln}) up to a change of variables.

The bimodule isomorphism follows from Lemma \ref{lemma:diagxequalst} and its  analog for  $\mathscr{D}_{\theta}$.
\end{proof}

The next result shows that the $\mathbb{C}$-algebra isomorphisms among the $\mathcal{A}_{\upsilon,\sigma}$ of Proposition \ref{prop:iso} extend to this setting for all three families. Recall that $\bar{a}$ denotes the image of $a\in \mathbb{C}(q)$ under the $\mathbb{C}$-automorphism of $\mathbb{C}(q)$ sending $q$ to $q^{-1}$. 

\begin{proposition} 
Set $N={\it rank}(\mathfrak{g})$. The map sending each scalar $a$ to $\bar{a}$,  $x_{ij}$ to $x_{N-i,N-j}$ (resp. $x_{i,j+n}$ to $x_{n-i,2n-j}$)  and $d_{ij}$ to $d_{N-i,N-j}$ (resp. $d_{ij}$ to $d_{n-i,2n-j}$)  defines  a $\mathbb{C}$-algebra isomorphism from $\mathcal{A}^{\theta}_{00}$ to $\mathcal{A}^{\theta}_{11}$ and from $\mathcal{A}^{\theta}_{10}$ to $\mathcal{A}^{\theta}_{01}$ in Types AI   and AII  (resp.  diagonal type).
\end{proposition}
\begin{proof}
The diagonal case follows immediately from Corollary \ref{diagiso} and Proposition \ref{prop:iso}.  For Types AI and AII, it is straightforward to check directly from the relations  that this map defines an isomorphism of $\mathscr{P}_{\theta}$ onto itself.  Using the $\mathbb{C}(q)$ antiautomorphism defined by $x_{ij}\mapsto d_{ij}$ (see the discussion preceding Proposition \ref{prop:Dalgebra}) we see that the same result holds for $\mathscr{D}_{\theta}$.  It remains to show that the relations defined by the twisting map $\tau^{\theta}_{00}$ becomes that of $\tau^{\theta}_{11}$ and the ones defined by the twisting map $\tau^{\theta}_{10}$ becomes those of $\tau^{\theta}_{01}$ under this mapping.  The argument follows as in the proof of Proposition \ref{prop:iso} using the explicit form of these twisting maps given above.
\end{proof}

Using the formulas for the entries of $R$ and $R_{21}^{-1}$,   one can expand out the relations defined by the twisting 
map in Theorem \ref{theorem:graded-hom} in Types AI and AII.  We won't do a complete expansion here, but instead, note important properties.

\begin{corollary}\label{cor:relnsgradedWeyl} 
The following inclusions hold for the quantum graded Weyl algebra $\mathcal{A}_{00}^{\theta}$:  \begin{align*}
d_{ab}x_{ef} - q^{\delta_{af}+\delta_{ae}+\delta_{bf} +\delta_{be}} x_{ef}d_{ab} \in  \sum_{(e',f',a',b')>(e,f,a,b)} \mathbb{C}(q)x_{e'f'}d_{a'b'}
\end{align*}
for all $a,b,e,f\in \{1, \dots, {\rm rank}(\mathfrak{g})\}$ where
\begin{itemize}
\item $a\leq b$ and $e\leq f$ in Type AI
\item $a<b$ and $e<f$ in Type AII
\item $a\leq n<b$ and $e\leq n<f$ in  diagonal type
\end{itemize}
and $(e',f',a',b')>(e,f,a,b)$ if and only if $e'\geq e,f'\geq f, a'\geq a, b'\geq b$ and at least one of these inequalities is strict.   \end{corollary}
\begin{proof} The corollary follows in the diagonal case using the explicit relations given at the end of Section \ref{section:four}.  Hence, we just consider Types AI and AII and so   $R_{\mathfrak{g}} = R$ where $N=n$ in Type AI  and $N=2n$ in Type AII. An examination of the formulas from Section \ref{section:as} yields $(R^{t_2})^{lk}_{ji}\neq 0$ implies that $k\geq i$ and $l\geq j$. 
It follows that 
\begin{align*}
(R^{t_2})_{xq}^{wr}(R^{t_2})^{pq}_{ma}(R^{t_2})^{xy}_{fl}(R^{t_2})^{ml}_{eb}\neq 0
\end{align*}
implies that  ${y\geq l\geq b}$, $ {p\geq m\geq e}$, ${w\geq x\geq f}$, and ${r\geq q\geq a}$.
Hence 
\begin{align*}
d_{ab}x_{ef} = (R^{t_2})^{fa}_{fa}(R^{t_2})^{ea}_{ea}(R^{t_2})^{fb}_{fb}(R^{t_2})^{eb}_{eb}
x_{ef}d_{ab} +X
\end{align*}
where
\begin{align*} X\in \sum_{(e',f',a',b')>(e,f,a,b)} \mathbb{C}(q)x_{e'f'}d_{a'b'}.
\end{align*}
The corollary now follows from the fact that $(R^{t_2})^{ji}_{ji}  = q^{\delta_{ij}}$.
\end{proof} 

Note that a version of Corollary \ref{cor:relnsgradedWeyl} holds for  $\mathcal{A}_{11}^{\theta}$ with the coefficient $q^{\delta_{af}+\delta_{ae}+\delta_{bf} +\delta_{be}}$ replaced by $q^{-(\delta_{af}+\delta_{ae}+\delta_{bf} +\delta_{be})}$ and the sum runs over tuples $(e',f',a',b')$ satisfying 
the opposite inequality $(e,f,a,b)>(e',f',a',b')$.  

\begin{remark}  As explained in \cite{GY}, $\mathcal{O}_q(\Mat_N)$ is a  CGL extension and hence a  quantum nilpotent algebra.   One checks using Lemma \ref{lemma:explicit-relations} that the same holds for $\mathscr{P}_{\theta}$ for all three families. It turns out that $\mathcal{A}^{\theta}_{00}$ is an example of a symmetric CGL extension (see \cite{GY}, Section 3.3) for each family.  Certainly the previous result suggests that this is true.  This property can be verified via a complete expansion of the relations for $\mathcal{A}^{\theta}_{00}$ arising from the twisting map that defines these algebras.  Note that it is further shown in \cite{GY} that $\mathcal{O}_q(\Mat_N)$ is an example of a quantum cluster algebra. We suspect the same is true for $\mathscr{P}_{\theta}$ as well as $\mathcal{A}^{\theta}_{00}$ for all three symmetric pair families.
\end{remark} 




\section{Quantum Weyl Algebras}\label{section:QWA}
 
\subsection{An invariant bilinear form} \label{section:inv-bi-form}
The starting point for lifting the graded quantum Weyl algebras is to determine how to introduce constant terms in the relations coming from the twisting map.  This will be accomplished using  a $U_q(\mathfrak{gl}_N)$ bi-invariant   bilinear form.

\begin{lemma} \label{lemma:bilinear-form}The bilinear form $\langle \cdot, \cdot \rangle$ on  $\left(\sum_{i,j}\mathbb{C}(q)\partial_{ij} \right)\times \left(\sum_{i,j}\mathbb{C}(q)t_{ij} \right)$ defined by $$\langle\partial_{ij}, t_{sk}\rangle = \delta_{is}\delta_{jk}
{\rm \ for\  all\  }i,j,s,k$$
is $U_q(\mathfrak{gl}_N)$ bi-invariant with respect to the $U_q(\mathfrak{gl}_N)$-module structures on $\sum_{i,j}\mathbb{C}(q)\partial_{ij}$ and $\sum_{i,j}\mathbb{C}(q)t_{ij}$.
\end{lemma}
\begin{proof} It is sufficient to check the properties of left and right invariance with respect to the generators of 
 $U_q(\mathfrak{gl}_N)$.
For the left action, we have
\begin{align*}
\langle K_{\epsilon_r}\cdot\partial_{ij}, K_{\epsilon_r}\cdot t_{sk}\rangle &= q^{-\delta_{ir}}q^{\delta_{sr}}\langle \partial_{ij}, t_{sk}\rangle =  \delta_{is}\delta_{jk}
= \epsilon(K_{\epsilon_r})\langle \partial_{ij}, t_{sk}\rangle .
\end{align*} for all $r, i,j,s,k$.
Also, for all $r,i,j,s,k$ we have
\begin{align*}
\langle E_r\cdot\partial_{ij}, t_{sk}\rangle + \langle K_r\cdot\partial_{ij}, E_r\cdot t_{sk}\rangle &= -q^{-1}\delta_{ir}\langle  \partial_{i+1,j}, t_{sk}\rangle + q^{-\delta_{ir}+\delta_{i,r+1}}\delta_{r,s-1}\langle \partial_{ij} , t_{s-1,k}\rangle\cr
&= -q^{-1}\delta_{ir}\delta_{i+1,s}\delta_{jk} + q^{-\delta_{ir}}\delta_{r,s-1}\delta_{i,s-1}\delta_{jk}
\cr&=\delta_{ir}\delta_{i+1,s}\delta_{j,k}(-q^{-1}  + q^{-1})=0= \epsilon(E_r) \langle \partial_{ij},t_{sk}\rangle.
\end{align*}
Similarly, for all $r,i,j,s,k$ we have
\begin{align*}
\langle F_r\cdot \partial_{ij}, K_r^{-1}\cdot t_{sk}\rangle + \langle \partial_{ij}, F_r\cdot t_{sk}\rangle &= -qq^{-\delta_{sr}+\delta_{s,r+1}}\delta_{i-1,r}\langle\partial_{i-1,j}, 
t_{sk}\rangle +\delta_{sr} \langle \partial_{ij},t_{s+1,k}\rangle
\cr &=\delta_{i-1,s}\delta_{jk}(-qq^{-\delta_{sr}}\delta_{i-1,r}+\delta_{sr} )
\cr &=\delta_{i-1,s}\delta_{jk}\delta_{sr}(-qq^{-1} + 1 ) = 0 = \epsilon(F_r) \langle \partial_{ij},t_{sk}\rangle.
\end{align*} This establishes left invariance.  A similar computation yields right invariance.
\end{proof}

The bilinear form $\langle \cdot, \cdot \rangle$ can be translated into a linear map  $\mu:\sum_{i,j,k,l}\mathbb{C}(q)\partial_{ij}\otimes t_{kl}\rightarrow\mathbb{C}(q)$ defined  by 
$\mu(\partial_{ij}\otimes t_{kl}) = \langle \partial_{ij}, t_{kl}\rangle$ for all $i,j,k,l$.   An immediate consequence of the above lemma is that 
\begin{align}\label{u-invar}
\mu(u\cdot (\partial_{ij}\otimes t_{sk}) )= u\cdot \mu(\partial_{ij}\otimes t_{sk}) = \epsilon(u) \mu(\partial_{ij}\otimes t_{sk}).
\end{align}
for all all $u\in U_q(\mathfrak{gl}_n)$ and all $i,j,s,k$.  The analogous equality holds for the right action.  Note that (\ref{u-invar}) and its right hand version ensure that $\mu$ is a 
$U_q(\mathfrak{gl}_n)$ bi-invariant map from $\sum_{i,j,s,k} \mathbb{C}\partial_{ij}\otimes t_{sk}$ to $\sum_{i,j,s,k} \mathbb{C}t_{sk}\otimes \partial_{ij}\oplus \mathbb{C}(q).$


 For each choice of $\upsilon,\sigma\in\{0,1\}$ and each $a,b,e,f$, we write $W^{\upsilon,\sigma}_{a,b,e,f}$ for the relation associated to $a,b,e,f$ so that 
\begin{align*}
W^{\upsilon,\sigma}_{a,b,e,f} =  \partial_{ea}\otimes t_{fb}-\tau_{\upsilon,\sigma}(\partial_{ea}\otimes t_{fb}) = \partial_{ea}\otimes t_{fb}-\sum_{j,k,d,l} (R_{\sigma})^{dl}_{fe}(R_{\upsilon})^{jk}_{ba}t_{dj}\otimes \partial_{lk}\end{align*}for all $a,b,e,f$.  By Proposition \ref{prop:twisted}, the vector 
 space spanned by the $W^{\upsilon,\sigma}_{a,b,e,f}$ for $a,b,e,f\in \{1, \dots, n\}$ is 
 a $U_q(\mathfrak{gl}_n)$ sub-bimodule of $\sum_{i,j,k,l}\mathbb{C}(q)\partial_{ij} \otimes t_{kl} +\sum_{i,j,k,l}\mathbb{C}(q) t_{kl}\otimes \partial_{ij}$ with respect to the bimodule structure defined by Lemmas \ref{lemma:actions} and \ref{lemma:opposite relation}. 
   We see that the same holds when we add a scalar to each relation coming from the bilinear form.

\begin{proposition}\label{prop:Wdefn}
 For each $\upsilon,\sigma$, the vector space spanned by  
\begin{align}\label{Wdefn}
\{{W}^{\upsilon,\sigma}_{e,a,f,b} - \mu( \partial_{ea}\otimes t_{fb})|\ a,b,e,f\in \{1,\dots, n\}\}
\end{align}
 is a $U_q(\mathfrak{gl}_N)$-sub-bimodule of $\sum_{i,j,k,l}\mathbb{C}(q) \partial_{ij} \otimes t_{kl}+\sum_{i,j,k,l}\mathbb{C}(q) t_{kl}\otimes \partial_{ij} + \mathbb{C}(q)$.   
\end{proposition}
\begin{proof}
Since the choice of $\upsilon, \sigma$ does not impact the proof, we drop the superscript from $W$ in 
the argument below. Note that for each $e,a,f,b$ we have 
\begin{align*}
W_{e,a,f,b}\in \partial_{ea}\otimes t_{fb} + \sum_{i,j,k,l}\mathbb{C}(q)\partial_{ij}\otimes t_{kl}.
\end{align*}
Since 
 $$\left(\sum_{i,j,k,l}\mathbb{C}(q)\partial_{ij}\otimes t_{kl}\right)\cap 
\left(\sum_{i,j,k,l}\mathbb{C}(q)t_{ij}\otimes \partial_{kl}\right) = 0,$$ it suffices to show that the vector space spanned by 
$\{\partial_{ea}\otimes t_{fb} - \mu (\partial_{ea}\otimes t_{fb})|\ a,b,e,f\in \{1,\dots, n\}\}$
is a $U_q(\mathfrak{gl}_n)$-bimodule. 
By (\ref{u-invar}), 
\begin{align*}
u\cdot (\partial_{ea}\otimes t_{fb} 
- \mu (\partial_{ea}\otimes t_{fb}) )= u\cdot  (\partial_{ea}\otimes t_{fb} ) -
 \mu( u\cdot (\partial_{ea}\otimes t_{fb}))
\end{align*} 
for all $u\in U_q(\mathfrak{gl}_n)$ and all $e,a,f,b$.  This proves invariance with respect to the left action of $U_q(\mathfrak{gl}_n)$.
The same argument works for the right action. 
\end{proof}


\subsection{A Criteria for PBW deformations} \label{section:PBWdef}

Recall the notion of PBW deformations introduced in \cite{BG} for quadratic algebras and generalized to other algebras
in \cite{WW2014}, and defined as follows.
\begin{definition} Let $D=\cup_{i\in \mathbb{N}}\mathcal{F}_i(D)$ be an filtered algebra and let $E$ be a $\mathbb{N}$-graded algebra.  The algebra $D$ is called a PBW deformation of $E$ if there is 
a filtered  map from $E$ to $D$ that defines 
an isomorphism of 
$E$ onto ${\rm gr}_{\mathcal{F}}D = \bigoplus_{i\geq 0}\mathcal{F}_i(D)/\mathcal{F}_{i-1}(D)$  as  $\mathbb{N}$-graded algebras.
\end{definition}

  We are interested in the following scenario. Assume that $A= T(Y)/\langle I\rangle$ and $B= T(Z)/\langle J\rangle$ where $I$ is a subspace of $Y\otimes Y$ and  $J$ is a subspace of $Z\otimes Z$. Note that both $A$ and $B$ are graded quadratic algebras. We further assume that both $A$ and $B$ are Koszul algebras (see \cite{BG}, Section 3, for a precise definition of Koszul).

  Let $\tau_{(1,1)}$ be a linear map  that sends $Z\otimes Y$ to $Y\otimes Z$.  Note that we can use $\tau_{(1,1)}$ to inductively define maps $\tau_{(m,s)}$  from $Z^{\otimes m}\otimes Y^{\otimes s}$
  by 
  \begin{align*} \tau_{(m,s)} = (Id \otimes \tau_{m-1,s-1}\otimes Id)(\tau_{m-1,1}\otimes \tau_{1, s-1})(Id^{\otimes m-1})\otimes \tau_{1,1}\otimes (Id^{\otimes s-1}).
  \end{align*} 
Now define a linear map   $\tau$  from $T(Z)\otimes T(Y) $ to $T(Y)\otimes T(Z)$ by insisting that $\tau(1\otimes a) = a\otimes 1$, $\tau(b\otimes 1) = 1\otimes b$ and 
   $\tau(c\otimes d) = \tau_{(m,s)}(c\otimes d)$ for all $c\in Z^{\otimes m}$ and $d\in Y^{\otimes s}$. We can define multiplication on $T(Z)\otimes T(Y)$  using the second property of twisting maps (see Section \ref{section:twisted}).  
   It is straightforward to check that  $\tau$ defines a twisting map from $T(Z)\otimes T(Y)$ to $T(Y)\otimes T(Z)$.
   
   Assume further that $\tau$  becomes a twisting map 
   from $B\otimes A$ to $A\otimes B$ when passing from $T(Z)\otimes T(Y)$ to $B\otimes A$ and denote this induced twisting map by $\tau$ as well. 
Set $E=A\otimes_{\tau}B$ and note that $E$ is also a Koszul algebra (\cite{WW2018}, Proposition 1.8).   Furthermore, we may identify $E$ with $T(Z\oplus Y)/\langle I + J + K\rangle$ 
where $K$ is the subspace of $Z\otimes Y + Y \otimes Z$ consisting of the elements $w-\tau(w)$ for $w\in Z\otimes Y$.

Let $\mu$ be a linear map from $I+J+K$ to $\mathbb{C}(q)$
and define the algebra $E_{\mu}$ by 
\begin{align}\label{defn:emu} E_{\mu} = T(Z\oplus Y)/\langle r-\mu(r), r\in I+J+K\rangle.
\end{align} By \cite{WW2018} Theorem 2.4 (c'), $E_{\mu}$ is a PBW deformation of $E$ if and only if $\mu\otimes Id= Id \otimes \mu$ on $$((I+J+K)\otimes(Z+Y))\cap ((Z+ Y)\otimes (I+J+K)).$$ (In the notation of \cite{WW2018} Theorem 2.4, $\mu$ plays the role of $\kappa^C$ while $\kappa^L=0$).

Now assume that $\mu$ is defined as follows.  Start with  
a linear map $\mu$  from $Z\otimes Y$ to $\mathbb{C}(q)$.  Extend $\mu$
to a linear map from $(Z+Y)\otimes (Z+Y)$ to $\mathbb{C}(q)$ by insisting that $\mu$ is identically $0$ on both $Z\otimes Z$ and $Y\otimes (Z+Y)$.  Note that  $\mu$ restricts to a linear map on
$I+J+K$, since the latter is  a subset of $(Z+Y)\otimes (Z+Y)$. Thus, we may take advantage of the above criteria for PBW deformations based on the construction of $E_{\mu}$  above.

 Since both $\mu\otimes Id$ and $Id \otimes \mu$  vanish on $Z\otimes Z\otimes Z+Y\otimes Y\otimes Y$,  the  above PBW deformation conditions becomes
$E_{\mu}$ is a PBW deformation of $E$ if and only if $\mu\otimes Id= Id \otimes \mu$ on 
\begin{align*}
((J\otimes Y + K\otimes Z)\cap (Z\otimes K + Y\otimes J) ) 
\end{align*}
and 
\begin{align*}((I\otimes Z + K\otimes Y)\cap (Z\otimes I + Y\otimes K))\end{align*}
The next result adapts  \cite{WW2018}, Theorem 2.4 (c'),  providing a particularly useful  criteria for when $E_{\mu}$ is a PBW deformation of $E$ in various settings of this paper.
  
  \begin{lemma}\label{lemma:PBWdef}   Let $\mu$ be a linear map from $(Y+Z)\otimes (Y+Z)$ to $\mathbb{C}(q)$ that is zero on $Y\otimes Y+ Y\otimes Z + Z\otimes Z$.   The algebra $E_{\mu}$ is a PBW deformation of $A\otimes_{\tau}B$ if   \begin{align}\label{cond1}\left(\mu\otimes Id+(Id\otimes \mu)(\tau\otimes Id)\right)(Z\otimes I)=0   \end{align}
  and 
  \begin{align}\label{cond2}\left(Id\otimes \mu+(\mu\otimes Id)(Id\otimes \tau)\right)(J\otimes Y)=0. \end{align}
    \end{lemma}

\begin{proof} By definition, 
$\mu\otimes Id$ vanishes on $\tau(Z\otimes I)$ since this space is a subspace of $Y\otimes Y\otimes Z$.  Similarly, $\mu\otimes Id$ vanishes 
on $(\tau\otimes Id)(Z\otimes I)$ because it is a subspace of $Y\otimes Z\otimes Y$.
Hence 
\begin{align*}
(\mu\otimes Id)(Z\otimes I) = ({\mu}\otimes Id)\circ (Id +\tau -\tau\otimes Id)(Z\otimes I).
\end{align*}
Similarly, 
\begin{align*}
(Id\otimes \mu)((\tau\otimes Id)(Z\otimes I) )= (Id\otimes \mu)\circ (-Id -\tau +\tau\otimes Id)(Z\otimes I).\end{align*}
Hence (\ref{cond1})   is equivalent to 
${\mu}\otimes Id - Id \otimes {\mu}$ vanishes on 
$(Id+ \tau -\tau\otimes Id)(Z\otimes I)$.  A similar analysis shows that (\ref{cond2})  
is equivalent to 
${\mu}\otimes Id - Id \otimes {\mu}$ vanishes on $(Id+ \tau -Id\otimes \tau)(J\otimes Y).$  

 Let $a\in(K\otimes Y+ I\otimes Z)\cap (Z\otimes I+Y\otimes K)$.  
 and write $a = a_1 + a_2 + a_3$ where $a_1\in Y\otimes Y\otimes Z, a_2 \in Y\otimes Z\otimes Y, $ and $a_3\in Z\otimes Y\otimes Y$. Since $a\in (K\otimes Y+ I\otimes Z)$,  we must have $a_3 + a_2 \in K\otimes Y$ with $a_2 = -(\tau\otimes Id)(a_3)$.  Similarly, $a\in (Z\otimes I+Y\otimes K)$
 ensures that $a_2 + a_1\in Y\otimes K$ with $a_1 = -(Id\otimes \tau)(a_2)$.  These two conditions together yield
$a_1= (Id\otimes \tau)(\tau\otimes Id)(a_3)=\tau(a_3)$.
Hence 
$a = (\tau -(\tau\otimes Id) + Id)a_3$.
In other words, $$(K\otimes Y+ I\otimes Z)\cap (Z\otimes I+Y\otimes K)\subseteq (Id +\tau-\tau\otimes Id)(Z\otimes I).$$ A similar argument shows that 
$$(J\otimes Y+K\otimes Z)\cap (Z\otimes K+ Y\otimes J) \subseteq (Id +\tau-Id\otimes \tau)(I\otimes Z).$$  The proof now follows by the discussion preceding the lemma.
\end{proof}

\subsection{PBW Deformations for Matrices}\label{section:PBWmatrices}
We now turn our attention to specific twisted tensor product algebras  introduced earlier.   In particular, we consider $A=\mathcal{O}_q(\Mat_N)$, $B=\mathcal{O}_q(\Mat_N)^{op}$, 
$E=\mathcal{A}_{\upsilon\upsilon} = A\otimes_{\tau}B$ where $\tau=\tau_{\upsilon,\upsilon}$ for $\upsilon \in\{0,1\}$ and $\mu$ is defined by the bilinear form of Lemma \ref{lemma:bilinear-form}.  As in Section \ref{section:FRT}, we have that $A$ is a quotient of the tensor algebra $T(Y)$ where $Y=V\otimes W$ and $B$ is a quotient of the tensor algebra $T(Z)$ where $Z=V^*\otimes W^*.$ 

Abusing notation somewhat, we will take $t_{ij}$ as a basis for $Y$ where we identify   $t_{ij} $ with $ v_i\otimes w_j$ for all $i,j$ and 
$\partial_{ij}$ for a basis of $Z$ where $\partial_{ij} $ is identified with $v_i^*\otimes w_j^*$ for all $i,j$.  Note that the set of defining relations $I$ in $Y\otimes Y$ can be read off of (i) and (ii) of Section \ref{section:as}.  Recall the bilinear map $\mu$ defined by the bilinear form $\langle \cdot, \cdot \rangle$ of Section \ref{section:inv-bi-form}.
The algebra $T(Y)$ is a $U_q(\mathfrak{gl}_n)$-bimodule via the action of Lemma \ref{lemma:actions}  and 
 $T(Z)$ is a $U_q(\mathfrak{gl}_n)$-bimodule via the action of Lemma \ref{lemma:opposite relation}.
  
\begin{theorem}\label{thm:deformation}
For both $\upsilon=0$ and $\upsilon=1$, the algebra $E_{\mu}$ is a PBW deformation of $E=\mathcal{A}_{\upsilon,\upsilon}= A\otimes_{\tau}B$ where $\tau=\tau_{\upsilon,\upsilon}$, $A=\mathcal{O}_q(\Mat_N)$ and $B=\mathcal{O}_q(\Mat_N)^{op}$.  Moreover, $E_{\mu}$ inherits a $U_q(\mathfrak{gl}_N)$-bimodule structure from $A$ and $B.$
\end{theorem}
\begin{proof} By Lemmas \ref{lemma:actions} and \ref{lemma:opposite relation}, both $I$ and $J$ are $U_q(\mathfrak{gl}_N)$-sub-bimodules of $T(Y)\otimes T(Z)$. By Proposition \ref{prop:Wdefn}, the space $K$ defined by $\tau$ and $\mu$ is a $U_q(\mathfrak{gl}_N)$ sub-bimodule of $T(Y)\otimes T(Z)$.  This proves the last assertion of the theorem.  For the first assertion, we show that $E_{\mu}$ satisfies the criteria of Lemma \ref{lemma:PBWdef}.

 By Corollary \ref{corollary:invar} and 
(\ref{u-invar}), both  $\tau$ and $\mu$ are $U_q(\mathfrak{gl}_N)$ bi-invariant.  Hence so is the map $\mu\otimes Id +(Id\otimes \mu)(\tau\otimes Id)$.  Hence
$(\mu\otimes Id +(Id\otimes \mu)(\tau\otimes Id)) (z\otimes y) = 0$
implies the same is true for $z\otimes y$ replaced by $u\cdot (z\otimes y)$  and by $ (z\otimes y)\cdot u$ for all $u\in U_q(\mathfrak{gl}_N)$.

 For the remainder of the proof, we assume that $\upsilon=0$. The proof for $\upsilon=1$ is the same up to  easy modifications such as  swapping $\partial_{NN}$ with $\partial_{11}$. Now suppose that 
\begin{align}\label{main-objective}
\left(\mu\otimes Id +(Id\otimes \mu)(\tau\otimes Id)\right) (\partial_{NN}\otimes y) = 0
\end{align}
for all $y\in I$. Hence 
$\left(\mu\otimes Id +(Id\otimes \mu)(\tau\otimes Id)\right) (F_{N-1}\cdot (\partial_{NN}\otimes y) )= 0$
for all $y\in I$.
Note that 
\begin{align*} F_{n-1}\cdot (\partial_{NN}\otimes y)&= (F_{N-1}\cdot \partial_{NN}) \otimes K_i^{-1}\cdot y +\partial_{NN}\otimes F_{N-1}\cdot y\cr 
&=-q\partial_{N-1,N} \otimes K_i^{-1}\cdot y + \partial_{NN}\otimes F_{N-1}\cdot y.
\end{align*}
Since $K_i$ acts semisimply on 
$Y\otimes Y$ and $I$ is an $U_q(\mathfrak{gl}_N)$ invariant subspace, we have $I = K_i^{-1}\cdot I$.  Hence, it follows that 
$\left(\mu\otimes Id+(Id\otimes \mu)(\tau\otimes Id)\right) (\partial_{N-1,N}\otimes  y) = 0$
for all $y\in I$.  By induction, 
$\left(\mu\otimes Id+(Id\otimes \mu)(\tau\otimes Id)\right) (\partial_{jN}\otimes  y) = 0$
for all $1\leq j\leq N$ and all $y\in I$.  Using the right action with $F_j$ replaced by $E_j$ yields 
$\left(\mu\otimes Id +(Id\otimes \mu)(\tau\otimes Id)\right) (\partial_{jk}\otimes  y) = 0$
for all $j,k$ and all $y\in I$.  Thus it is sufficient to establish (\ref{main-objective}) in order to prove (\ref{cond1}) with $\mu=\mu$ and $\tau=\tau_{00}$.

Note that the space  $I$ is spanned by the elements of $\sum_{i,j,k,l}\mathbb{C}(q)t_{ij}\otimes t_{kl}$  corresponding to the relations (i) and (ii)  of $\mathcal{O}_q(\Mat_N)$  presented in Section \ref{section:as}. 
It follows from the formulas for the entries of $R$ in Section \ref{section:images} that  $(R^{t_2})^{jk}_{bN} = r^{jN}_{bk} = q^{\delta_{bN}}\delta_{jb}\delta_{Nk}$
and $(R^{t_2})^{dl}_{fN} =r^{dN}_{fl} = q^{\delta_{fN}}\delta_{Nl}\delta_{df}$.   Hence, since $\upsilon=0$,    (\ref{twistdefn}) becomes 
\begin{align*}
\tau_{00}(\partial_{NN}\otimes t_{fb}) =  q^{\delta_{fN}+\delta_{bN}}t_{fb}\otimes \partial_{NN}
\end{align*}
at $(e,a)= (N,N)$ for all $f$ and $b$ in the set $\{1, \dots, N\}$.  It follows that the only term of the form $\partial_{ij}$ that shows up in 
the expression for $(\tau\otimes Id) (\partial_{NN}\otimes y)$ is $\partial_{NN}$ for all $y\in I$.  Hence we need only show that $\mu\otimes Id -(Id\otimes \mu)(\tau\otimes Id)$ vanishes
on the subset of $Z\otimes I$ consisting of elements 
$\partial_{NN}\otimes y$ where $y$ is any element in the set 
\begin{align*} 
\{t_{Ni}\otimes t_{NN} - qt_{NN}\otimes t_{Ni} , t_{iN}\otimes t_{NN} - qt_{NN}\otimes t_{iN}|\ i=1, \dots, N-1\}.
\end{align*}
We have 
\begin{align*}
(\mu\otimes Id) (\partial_{NN}\otimes t_{N,i}\otimes t_{NN} - q\partial_{NN}\otimes t_{NN}\otimes t_{Ni} )=-q(1\otimes t_{Ni}) 
\end{align*}
This final term is simply $-qt_{N,i}$ viewed as an element in $T(Z)$.  Also
\begin{align*}
&(Id\otimes \mu)(\tau_{00}\otimes Id)(\partial_{NN}\otimes t_{Ni}\otimes t_{NN} - q\partial_{NN}\otimes t_{NN}\otimes t_{Ni} )
\cr &=(Id\otimes \mu)(qt_{N,i}\otimes \partial_{NN}\otimes t_{NN} - q^3 t_{NN}\otimes \partial_{NN}\otimes t_{Ni} ) = q(t_{Ni}\otimes 1) = qt_{Ni}
\end{align*}
Thus $(\mu\otimes Id) + (Id\otimes \mu)(\tau_{00}\otimes Id)$ vanishes on all elements of the form $\partial_{NN}\otimes t_{N,i}\otimes t_{NN} - q\partial_{NN}\otimes t_{NN}\otimes t_{N,i}$.
The same argument with $t_{N,i}$ replaced by $t_{i,N}$ terms shows the analogous result  is true for all elements of the form $\partial_{NN}\otimes t_{i,N}\otimes t_{NN} - q\partial_{NN}\otimes t_{NN}\otimes t_{i,N}$.
This proves (\ref{main-objective}) and hence (\ref{cond1}) for $\tau=\tau_{00}$.  Criteria (\ref{cond2}) for $\tau=\tau_{11}$ is established using a similar argument with $\partial_{NN}$ replaced by $t_{NN}$.  \end{proof}

A natural question is whether a version of Theorem \ref{thm:deformation} holds for the other two algebras $\mathcal{A}_{01}$ and $\mathcal{A}_{10}$.  Unfortunately, such a deformation does not work using the invariant bilinear form $\mu$. To get an idea of why this construction fails, consider $(\upsilon,\sigma) = (0,1)$.  Formula (\ref{twistdefn}) becomes
\begin{align*}
\tau_{01}(\partial_{1,N}\otimes t_{fb}) =  q^{-\delta_{f1}+\delta_{bN}}t_{fb}\otimes \partial_{1N}
\end{align*}
for $(e,f)= (1,N)$ and any choice of $f,b$.  We have 
\begin{align*}
(\mu\otimes Id) (\partial_{1N}\otimes t_{1i}\otimes t_{1N} - q\partial_{1N}\otimes t_{1N}\otimes t_{1i} )=-qt_{1i} 
\end{align*}
but 
\begin{align*}
&(Id\otimes \mu)(\tau_{01}\otimes Id)(\partial_{1N}\otimes t_{1i}\otimes t_{1N} - q\partial_{1N}\otimes t_{1N}\otimes t_{1i} )
\cr &=(Id\otimes \mu)(q^{-1}t_{1,i}\otimes \partial_{1N}\otimes t_{1N} - q t_{1N}\otimes \partial_{1N}\otimes t_{1i} ) =  q^{-1}t_{1i}
\end{align*}
and clearly $-qt_{1i} \neq q^{-1}t_{1i}$.  This means that when we replace  (\ref{twistdefn}) with the deformed version using $\mu$, the element $t_{1i}$ ends up in the image of the  ideal generated by the relations for the $t_{kl}$.  In other words,  $t_{1i}$ must be $0$ in this deformation of the twisted tensor product $\mathcal{A}_{01}$.  Clearly, the end result is not a PBW deformation.  Despite this failure, we will find the twisted tensor products $\mathcal{A}_{01}$ and $\mathcal{A}_{10}$ to be essential in the construction of quantum Weyl algebras for homogeneous spaces.

Write $\mathcal{W}_{00}$ for the deformation of $\mathcal{A}_{00}$ and $\mathcal{W}_{11}$ for the deformation of  $\mathcal{A}_{11}$ as described by  Theorem \ref{thm:deformation}.  Note that Proposition can be easily extended to show that $\mathcal{W}_{00}$ and $\mathcal{W}_{11}$ are isomorphic as algebras.  It turns out that $\mathcal{W}_{00}$ is isomorphic to the quantum Weyl algebra $\mathscr{PD}_q(\Mat_N)$ studied in \cite{B} that is a normalized version of the algebra  ${\rm Pol}(\Mat_N)_q$ introduced in \cite{VSS} in the context of quantum bounded symmetric domains.  The algebra $\mathscr{PD}_q(\Mat_N)$  is generated by $z_{ij},z^*_{ij}, 1\leq i,j\leq N$ such that 
the $z_{ij}$ generate a subalgebra isomorphic to $\mathcal{O}_q(\Mat_N)$,
the $z^*_{ij}$ generate a subalgebra isomorphic to $\mathcal{O}_q(\Mat_N)^{op}$, 
and 
\begin{align}\label{Weylrelns}
z^*_{ea}z_{fb} = \sum_{j,n,d,l} q^2R(e,f,l,d)R(a,b,k,j)z_{dj}z^*_{lk} + \delta_{ef}\delta_{ab}
\end{align}
for all $e,a,f,b$ where 
\begin{itemize}
\item $ R(i,i,i,i) = 1,\ 
R(i,j,i,j)= q^{-1}$
for all $i,j$ with $i\neq j$.
\item 
$ R(i,i,j,j)= -(q^{-2}-1)$
for all $i<j$.
\item 
$R(i,j,k,l) = 0$ for all other choices of $j,i,j',i'$.
\end{itemize}
Note that (\ref{Weylrelns}) is just  equality (6) of \cite{B} (up to a change in variable name) but with the scalar term missing a coefficient of $(1-q^2)$. As explained in \cite{B}, one can rescale the terms $\partial_{ea}$ and $t_{dj}$ so as to remove this coefficient. We have performed this rescaling  in (\ref{Weylrelns}). 

\begin{proposition}\label{prop:PBWdef} The   algebra  $\mathcal{W}_{00}$ is isomorphic to  the algebra $\mathscr{PD}_q(\Mat_N)$.
\end{proposition}
  \begin{proof}  A comparison of the entries of the matrix $R$ as stated in Section \ref{section:images} and the elements $R(a,b,c,d)$ defined above yields $q^{-1}(R^{t_2})^{dl}_{fe}  = R(e,f,l,d)$
  for all $e,f,l,d$. Hence   (\ref{Weylrelns}) agrees with the relations for $\mathcal{W}_{00}$ obtained from (\ref{twistdefn}) with the deformation using $\mu$.   It follows that the map $\mathcal{W}_{00}\rightarrow \mathcal{W}$ sending $t_{ij}, \partial_{ij}$ to $z_{ij}, z^*_{ij}$  is an algebra isomorphism. 
  \end{proof}
It should be noted that the $U_q(\mathfrak{gl}_N)$ left module structure studied in this paper differs from that of \cite{VSS} and \cite{B} and, moreover, these references do not include a right module structure. Now by \cite{VSS} (see Proposition 2.1),
 the map $\mathcal{O}_q(\Mat_N)\otimes \mathcal{O}_q(\Mat_N)^{op}\rightarrow \mathscr{PD}_q(\Mat_N)$ defined by multiplication is  an isomorphism of vector spaces.   This fact directly implies that $\mathscr{PD}_q(\Mat_N) =\mathcal{W}_{00}$ is a PBW deformation of $\mathcal{A}_{00}$.  So Theorem \ref{thm:deformation} is really a new proof for a  result   in \cite{VSS}.  However,  the proof here has a number of advantages. It is less computational, it can be used to easily show that the other algebras $\mathcal{A}_{\upsilon,\sigma}, \upsilon\neq \sigma$ do not admit such a PBW deformation (as explained above),  and, perhaps most importantly, it  extends in a straightforward manner to the setting of quantum homogeneous spaces. 

\subsection{Invariant bilinear forms for homogeneous spaces}
 Here we introduce invariant bilinear forms so as to deform the algebras $\mathcal{A}^{\theta}_{\alpha,\upsilon}$ to the noncommutative setting (when $\alpha=\upsilon$) in analogy to how we handled the 
algebras $\mathcal{A}_{\alpha,\upsilon}$. For the regular quantum Weyl algebra, invariant meant with respect to the bimodule action of $U_q(\mathfrak{gl}_N)$.  For the three families of homogeneous spaces, by invariant, we mean with respect to the left action of $U_q(\mathfrak{g})$.   In the diagonal setting, we can use  Corollary  \ref{cor:relnsgradedWeyl} and the discussion following its proof to convert
 $\mathcal{A}^{\theta}_{\alpha,\upsilon}$ into $\mathcal{A}_{\alpha,\upsilon}$ as $U_q(\mathfrak{gl}_n)$-bimodules by moving the action of one of the copies of $U_q(\mathfrak{gl}_n)$ from the left to the right.  Hence, for most of the proofs below, the diagonal type follows immediately from the regular one.  
 
 Recall the quantum Weyl algebra $\mathcal{W}_{00}$, which is a deformation of  the graded quantum Weyl algebra $\mathcal{A}_{00}$, as introduced in Section \ref{section:PBWmatrices}. By construction, both $\mathscr{P}$ and $\mathscr{D}$ are $U_q(\mathfrak{g})$-sub-bimodule algebras of $\mathcal{W}_{00}$.  Hence, both $\mathscr{P}_{\theta}$ and $\mathscr{D}_{\theta}$ are left $U_q(\mathfrak{g})$  submodule algebras of $\mathcal{W}_{00}$.
 
 Let $\mathcal{L}$ denote the left ideal of $\mathcal{W}_{00}$ generated by the $\partial_{ij}$.   Since $\sum_{i,j}\mathbb{C}(q)\partial_{ij}$ is a $U_q(\mathfrak{g})$-bimodule, so is $\mathcal{L}$.  Note that  $\mathcal{W}_{00} = \mathcal{L}\oplus \mathscr{P}$.  Moreover, 
if $y\in \mathscr{D}$ and $w\in \mathscr{P}$ are homogeneous terms in the $\partial_{ij}$ and $t_{ij}$ respectively 
and of the same degree, then $yw \in \mathcal{L} + \mathbb{C}(q)$.
Thus, we can define a bilinear form 
\begin{align}\label{bilinear}
\left(\sum_{i,j}\mathbb{C}(q)d_{ij} \right)\times \left(\sum_{i,j}\mathbb{C}(q)x_{ij} \right)\rightarrow \mathbb{C}(q)
\end{align}
by the formula
\begin{align}\label{invformula}
d_{ij}x_{uv} -  \langle d_{ij}, x_{uv}\rangle \in\mathcal{L}
\end{align}
where $ \langle d_{ij}, x_{uv}\rangle\in \mathbb{C}(q)$ for each $i,j,u,v$.   
\begin{lemma} The biliinear form defined by (\ref{invformula}) is left $U_q(\mathfrak{g})$ and right $\mathcal{B}_{\theta}$ invariant.
\end{lemma}
\begin{proof}Note that $\mathcal{L}$, the scalars $\mathbb{C}(q)$, and 
$\sum_{i,j, u,v}\mathbb{C}(q)d_{ij} x_{uv}$
are left $U_q(\mathfrak{g})$-modules.   Hence, the above map is a map of left $U_q(\mathfrak{g})$-modules.  Since elements of $U_q(\mathfrak{g})$ act trivially (i.e. by the augmentation map) on $\mathbb{C}(q)$, we have  
\begin{align*}
a\cdot d_{ij}x_{uv} - \epsilon(a)\cdot \langle d_{ij}, x_{uv}\rangle \in a\cdot \mathcal{L} \subseteq \mathcal{L}
\end{align*}
for all $a\in U_q(\mathfrak{g})$.
Now $a\cdot d_{ij}x_{uv} = (\sum a_{(1)}\cdot d_{ij})(a_{(2)}\cdot x_{uv})$ Hence, by definition, we have $\langle \sum a_{(1)}\cdot d_{ij},a_{(2)}\cdot x_{uv}\rangle=\epsilon(a)\cdot \langle d_{ij}, x_{uv}\rangle$ as desired.

Note that both $\sum_{i,j}\mathbb{C}(q)d_{ij}$ and $\sum_{i,j}\mathbb{C}(q)x_{ij}$ are trivial right $\mathcal{B}_{\theta}$-modules.  Arguing as in the proof of Lemma \ref{lemma:B-invariance} and using the definition of invariant bilinear form, yields that this bilinear form is right $\mathcal{B}_{\theta}$ invariant.
\end{proof}

 Recall that a basis for $\sum_{i,j}\mathbb{C}(q)x_{ij}$ (resp. $\sum_{i,j}\mathbb{C}(q)d_{ij}$) consists of  those $x_{ij}$ (resp. $d_{ij}$) with $i\in \{1, \dots, n\}$ and $j\geq i$ in Type AI, $j>i$ in Type AII, and $j\geq n+1$ in  diagonal type.

\begin{proposition}\label{prop:bilinearAIAII} There exists a unique left $U_q(\mathfrak{g})$   and right $\mathcal{B}_{\theta}$ invariant form (up to a nonzero scalar multiple) such that 
\begin{align*}
\langle d_{ij}, x_{vu}\rangle=q^{-\delta_{uv}} \delta_{iv}\delta_{ju}
\end{align*}  for all $i\leq j$, $v\leq u$  in Type AI,
 for all  $i< j$, $v< u$ in Type AII, and for all $i<n$ and $j\geq n+1$ in diagonal type.
\end{proposition}
\begin{proof} By the discussion following Lemma \ref{lemma:left_action}, $\sum_{i,j}\mathbb{C}(q)x_{ij}$ is isomorphic to a simple $U_q(\mathfrak{g})$-module, 
which we refer to as $L$
and $\sum_{i,j}\mathbb{C}(q)d_{ij}$ is isomorphic to its dual $L^*$.  Hence $\sum_{ij}\mathbb{C}(q)d_{ij}\otimes \sum_{ij}\mathbb{C}(q)x_{ij}\cong L^*\otimes L\cong {\rm End}\ L$ 
as  left $U_q(\mathfrak{g})$-modules. It follows that there is a unique left one-dimensional $U_q(\mathfrak{g})$-submodule of $\sum_{i,j,k,l}\mathbb{C}(q)d_{ij}\otimes x_{kl}$.   Since both $\sum_{ij}\mathbb{C}(q)x_{ij}$  and $\sum_{ij}\mathbb{C}(q)d_{ij}$ are  trivial right $\mathcal{B}_{\theta}$-modules, this one-dimensional submodule must also be  a trivial right $\mathcal{B}_{\theta}$-module.

The map defined by (\ref{bilinear}) can be rewritten as a projection from $\sum_{i,j,u,v}\mathbb{C}(q)d_{ij}\otimes x_{uv}$ onto the scalars defined by $d_{ij}\otimes x_{uv} \mapsto \langle d_{ij}, x_{uv}\rangle$.  This map can further be viewed as a composition of the projection of the left hand side onto its one-dimensional submodule followed by an isomorphism from this one-dimensional submodule to the one defined by the bilinear form.  This guarantees that the bilinear form is unique up to nonzero scalar. 

We now turn to establishing the explicit formula for the bilinear form.  In the diagonal setting, this follows from  Corollary  \ref{cor:relnsgradedWeyl} and  Lemma \ref{lemma:bilinear-form} and the fact that $\delta_{uv}=0$ since $u\leq n<v$. For Types AI and AII, we do this by viewing $\mathscr{D}_{\theta}$ and $\mathscr{P}_{\theta}$ as subalgebras of $\mathcal{W}_{00}$ and 
compute the value of $d_{ij}x_{vu}$ module $\mathcal{L}$ for each $i,j,v,u$ using the relations presented at the end of Section \ref{section:four}.

First consider Type AI.  Assume that $i\leq j$ and $v\leq u$.  Recall that $x_{vu} = q^{1-\delta_{vu} }x_{uv}$.  
It follows from the relations in $\mathcal{W}_{00}$ that 
$\partial_{i,k}\partial_{jk} t_{us}t_{vs} \in \mathcal{L}$
for $k\neq s$. Moreover
\begin{align*}\partial_{us}t_{us} = 1 + \sum_{u'\geq u, s'\geq s}\mathbb{C}(q)t_{u's'}\partial_{u's'} \end{align*}
Hence, for $i< u'$ we have $\partial_{i,s}t_{u's'} \in \mathcal{L}$ and for  $v< u'$ we have  $\partial_{u's'}t_{vs}\in \mathcal{L}$.  Hence, if $i\leq j$ and $v\leq u$
then
$\partial_{i,s}\partial_{js}t_{us}t_{vs} \in \delta_{ju} \delta_{iv} + \mathcal{L}.$
Thus for $i\leq j$ and $v\leq u$ we have
\begin{align*}d_{ij}x_{vu} &=q^{1-\delta_{vu}}d_{ij} x_{uv} =  (\sum_{k=1}^nq^{-2k}\partial_{ik}\partial_{jk}) (\sum_{s=1}^nt_{us}t_{vs})\cr
 &=  q^{1-\delta_{vu}}\sum_{k=1}^nq^{-2k}\partial_{ik}\partial_{jk}t_{uk}t_{vk}+\mathcal{L}=(\sum_{k=1}^nq^{-2k+1})q^{-\delta_{vu}}\delta_{iv}\delta_{ju}. + \mathcal{L}
\end{align*}
which proves the lemma in Type AI.  

For Type AII, we have $-qx_{vu} = x_{uv}$ for $u<v.$ Also, since $x_{uv}=0$ whenever $u=v$, we may assume that $\delta_{uv}=0$ and so $q^{-\delta_{uv}}=1$ in the formula for the bilinear form. Note that for $i<j$ and $v<u$, we have 
$\partial_{i,2k-1}\partial_{j,2k}t_{us}t_{vr} =\delta_{ju}\delta_{iv} \delta_{2k,s}\delta_{2k-1,r} +\mathcal{L}$.
For $i<j$ and $v<u$, we have  $d_{ij}x_{vu}$ equals
\begin{align*}
-q^{-1}d_{ij}x_{uv} 
&=-\left(\sum_{k=1}^nq^{-4k+1}(\partial_{i,2k-1}\partial_{j,2k}-q^{-1}\partial_{i,2k}\partial_{j,2k-1})\right) \left(\sum_{s=1}^n(t_{u,2s-1}t_{v,2s}-qt_{u,2s}t_{v,2s-1})\right)\cr
&=-\sum_{k=1}^n(q^{-4k+1}(\partial_{i,2k-1}\partial_{j,2k}-q^{-1}\partial_{i,2k}\partial_{j,2k-1})(t_{u,2k-1}t_{v,2k}-qt_{u,2k}t_{v,2k-1})+\mathcal{L}\cr
&=\sum_{k=1}^n(-q-q^{-1})q^{-4k+1}\delta_{ju}\delta_{iv} +\mathcal{L}\end{align*} This proves the lemma in Type AII.
\end{proof}

 For each $a,b,e,f$, write $W^{\theta,\upsilon,\sigma}_{a,b,e,f}$ for the relation associated to $a,b,e,f$ so that 
\begin{align*}
W^{\theta,\upsilon,\sigma}_{a,b,e,f} = d_{ea}\otimes x_{fb}-\tau^{\theta}_{\upsilon,\sigma}(d_{ea}\otimes x_{fb}) \end{align*}for all $a,b,e,f$.
By Proposition \ref{prop:twisting-homcase}, the vector 
 space spanned by the $W^{\theta,\upsilon,\sigma}_{a,b,e,f}$ for $a,b,e,f\in \{1, \dots, {\rm rank}(\mathfrak{g})\}$ is 
 a left $U_q(\mathfrak{g})$-submodule  of $\sum_{i,j,k,l}\mathbb{C}(q)d_{ij} \otimes x_{kl} +\sum_{i,j,k,l}\mathbb{C}(q)x_{kl}\otimes d_{ij} $ with a trivial $\mathcal{B}_{\theta}$-module structure. Write  $\mu$ for the bilinear map from $\sum_{i,j,k,l}\mathbb{C}(q)d_{ij} \otimes x_{kl} $ to $\mathbb{C}(q)$ defined by the bilinear form of Proposition \ref{prop:bilinearAIAII}. 
 
\begin{proposition}\label{prop:Wdefn10}
 For each $\upsilon,\sigma$, the vector space spanned by  
\begin{align}\label{Wdefn10}
\{{W}^{\theta,\upsilon,\sigma}_{e,a,f,b} - \mu( d_{ea}\otimes x_{fb})|\ a,b,e,f\in \{1,\dots, {\rm rank}(\mathfrak{g})\}\}
\end{align}
 is  a left $U_q(\mathfrak{g})$ and  trivial right $\mathcal{B}_{\theta}$-submodule of   $\sum_{i,j,k,l}\mathbb{C}(q)d_{ij} \otimes x_{kl} +\sum_{i,j,k,l}\mathbb{C}(q)x_{kl}\otimes d_{ij} +\mathbb{C}(q)$. 
\end{proposition}
\begin{proof}
For the left $U_q(\mathfrak{g})$ action, the proof is the same as the proof for Proposition \ref{prop:Wdefn} using $x_{ij}$ instead of $t_{ij}$ and $d_{ij}$ instead of $\partial_{ij}$.   The assertion about the right action follow the discussion preceding the proposition.  The trivial right action of $\mathcal{B}_{\theta}$ also follows using elements $b\in \mathcal{B}_{\theta}$ instead of elements  $u\in U_q(\mathfrak{g})$ and  the arguments in Proposition \ref{prop:Wdefn}. \end{proof}

\subsection{PBW deformations for Homogeneous Spaces}
In this section, we lift the graded  algebras $\mathcal{A}^{\theta}_{\upsilon,\upsilon}$ to the non-graded setting using the methods of Section \ref{section:PBWdef}.  It should be noted that for  diagonal type, this follows directly from the isomorphism of Corollary \ref{cor:relnsgradedWeyl} and the results in Section \ref{section:PBWdef} up to consideration of the right $\mathcal{B}_{\theta}$ action.  Nevertheless, we include this case in the results below.

Write $\mathscr{P}_{\theta}$ as a quotient  $T(Y)/\langle I\rangle$ and $\mathscr{D}_{\theta}$ as a quotient  $T(Z)/\langle J\rangle$ where  $Y$ and $J$ are vector spaces and $I$ and $J$ are subvector spaces of $T(Y)$ and $T(Z)$ respectively spanned by the defining relations for these algebras.  Moreover, we only need to consider the quadratic defining relations for $\mathscr{P}_{\theta}$ and $\mathscr{D}_{\theta}$, and so 
 $I\subset Y\otimes Y, J\subset Z\otimes Z$. This is because we  can take the vector spaces $Y$ and $Z$ to be isomorphic  (via the natural quotient map)  to  $\sum_{i\in \{1, \dots, n\}, j\in S_i} \mathbb{C}(q)x_{ij}$ and $\sum_{i\in \{1, \dots, n\}, j\in S_i}\mathbb{C}(q) d_{ij}$ respectively where  $S_i=\{i,\dots, n\}$ in Type AI, $S_i = \{i+1, \dots, n\}$ in Type AII, and $S_i= \{n+1, \dots, 2n\}$ in the diagonal  case.  Hence both $\mathscr{P}_{\theta}$ and $\mathscr{D}_{\theta}$ are homogeneous quadratic algebras since $Y$ and $Z$ are homogeneous quadratic ideals.  Moreover, by Lemma \ref{lemma:PBW}, $\mathscr{P}_{\theta}$   admits a PBW basis and the same holds for $\mathscr{D}_{\theta}$ via the anti-isomorphism relating the two algebras.  Thus, by \cite{SW} Theorem 3.1, both $\mathscr{P}_{\theta}$ and $\mathscr{D}_{\theta}$ are Koszul algebras.  

As explained in the previous paragraph, the map from $T(Y)$ to $\mathscr{P}_{\theta}$ is a vector space isomorphism when restricted to the vector subspace $Y$ of $T(Y)$. Similarly, the map from $T(Z)$ to $\mathscr{D}_{\theta}$ restricts to an isomorphism of $Z$ to its image inside of $\mathscr{D}_{\theta}$.   It follows that the bilinear form of Proposition \ref{prop:bilinearAIAII} can be lifted to bilinear forms on $Y\times Z$.   Write $\mu$ for the corresponding linear map from $Y\otimes Z$ to $\mathbb{C}(q)$. Let  $E_{\mu}$ be the algebra defined as in   (\ref{defn:emu}).

\begin{theorem}\label{thm:deformation1}
For both $\upsilon=0$ and $\upsilon=1$, the algebra $E_{\mu}$ is a PBW deformation of $E=\mathcal{A}^{\theta}_{\upsilon,\upsilon}= A\otimes_{\tau}B$ where $\tau=\tau^{\theta}_{\upsilon,\upsilon}$, $A=\mathscr{P}_{\theta}$ and $B=\mathscr{D}_{\theta}$.  Moreover, $E_{\mu}$ inherits a (left) $U_q(\mathfrak{g})$-module structure and a (right) trivial $\mathcal{B}_{\theta}$-module structure from $A$ and $B$.
\end{theorem}
\begin{proof} The diagonal type follows from Corollary \ref{diagiso}, Theorem \ref{thm:deformation1} and Proposition \ref{prop:Wdefn10}.  So the focus here is on Types AI and AII.  We prove the theorem for $\upsilon=0$.   The $\upsilon = 1$ case is similar.  

By Lemma \ref{lemma:left_action} and the discussion following its proof,  $\sum_{i,j}\mathbb{C}x_{ij}$ is a simple left $U_q(\mathfrak{g})$-module generated by the lowest weight vector $x_{nn}$ in type AI and $x_{2n-1,2n}$ in Type AII. Furthermore,  $\sum_{i,j}\mathbb{C}d_{ij}$ is a simple left $U_q(\mathfrak{g})$-module with highest weight generating vector $d_{nn}$ in Type AI and $d_{2n-1,2n}$ in Type AII. 

Arguing as in the proof of Theorem \ref{thm:deformation1}, it is sufficient to show that $(\mu\otimes Id + (Id\otimes \mu)(\tau\otimes Id))$ vanishes on 
$z\otimes v$ for all $v\in I$ and $(\mu\otimes Id + (Id\otimes \mu)(Id\otimes \tau))$ on $w\otimes y$ for all $w\in J$ where $z= d_{nn}$ and $y=x_{nn}$ in Type AI 
and $z=d_{2n-1,2n}$, $y=x_{2n-1,2n}$ in Type AII.  We prove the first set of conditions using the assumptions on $z$ and $v$.  The proofs for elements $y$ and $w$  are analogous.

By Theorem \ref{theorem:graded-hom}, we have 
\begin{align}\label{tau-identity} \tau(d_{nn}\otimes x_{ef}) = \sum_{r,w,p,q,x,y,m,l}(R^{t_2})^{wr}_{xq}(R^{t_2})^{pq}_{mn}(R^{t_2})^{xy}_{fl}(R^{t_2})^{ml}_{en}x_{pw}\otimes d_{ry}
\end{align}
Using the explicit formulas for the entries of $R$ (see Section \ref{section:as}), we see that $(R^{t_2})^{cb}_{an} = r^{cn}_{ab}\neq 0$ if and only if $b=n$ and $a=c$.  Moreover $(R^{t_2})^{an}_{an}=1$ if $n\neq a$ and 
$(R^{t_2})^{an}_{an}=q$ for $n=a$.   If follows that in (\ref{tau-identity}), $l=n$, $e=m$, $y=n$, $f=x$, $q=n$, $p=e$, $r=n$, $w=f$ and so
\begin{align*} \tau(d_{nn}\otimes x_{ef}) =  (R^{t_2})^{fn}_{fn}(R^{t_2})^{en}_{en}(R^{t_2})^{fn}_{fn}(R^{t_2})^{en}_{en}x_{ef}\otimes d_{nn}
= q^{2\delta_{nf}+ 2\delta_{ne}} x_{ef}\otimes d_{nn}
\end{align*}
Similarly in Type AII,  we have 
\begin{align}\label{tau-identity2} \tau(d_{2n-1,2n}\otimes x_{ef})&=  q^{4-\delta_{e,2n-1}-\delta_{f,2n}-\delta_{f,2n-1}} x_{ef}\otimes d_{2n-1,2n}.
\end{align}
for all $e,f$ satisfying $e<f$.

The space $I$ is spanned by  families of relations which can be deduced from Lemma \ref{lemma:explicit-relations}.   For Type AI, using the fact that $\mu(d_{nn}\otimes x_{ij}) = 0$ unless $i=n=j$, 
we only need to show that $(\mu\otimes Id + (Id\otimes \mu)(\tau\otimes Id))$ vanishes on $d_{nn}\otimes v$ where $v$ is one of the following relations in $I$:
\begin{itemize}
\item $x_{nn} \otimes x_{dn} -q^{-2}x_{dn}\otimes x_{nn}$ for $d<n$.
\item $x_{nn}\otimes x_{dk} - x_{dk}\otimes x_{nn} -q^{-1}(q^2-q^{-2})x_{dn}\otimes x_{kn}$ for $d\leq k<n$.
\end{itemize}
Note that both $\mu\otimes Id$ and $(Id\otimes \mu)(\tau\otimes Id)$ vanish on $d_{nn}\otimes x_{d,n}\otimes x_{k,n}$ for $d\leq k<n$ and hence we can ignore this term.   By Proposition \ref{prop:bilinearAIAII}  and (\ref{tau-identity}) , we have 
\begin{align*}
(\mu\otimes Id) (d_{nn} \otimes (x_{nn}\otimes x_{dk}-q^{-2\delta_{kn}}x_{dk}\otimes x_{nn}))  = (\mu\otimes Id) (d_{nn} \otimes x_{nn}\otimes x_{dk}) = q^{-1}x_{dk}
\end{align*}
and
\begin{align*}
(Id\otimes \mu)&(\tau\otimes Id)) (d_{nn} \otimes \left(x_{nn}\otimes x_{dk}-q^{-2\delta_{kn}}x_{dk}\otimes x_{nn}\right))  \cr &= (Id\otimes \mu)(q^4x_{nn} \otimes d_{nn}\otimes x_{dk}-
x_{dk}\otimes d_{nn}\otimes x_{nn})\cr &= -(Id\otimes \mu)(x_{dk}\otimes d_{nn}\otimes x_{nn})=-q^{-1}x_{dk}
\end{align*}
for all $d$ and $k$ satisfying $d\leq k\leq n$ and $d<n$.  Thus $(\mu\otimes Id + (Id\otimes \mu)(\tau\otimes Id))$ vanishes on the desired elements in $d_{nn}\otimes I$ and thus on the
entire set  $Z\otimes I$ in Type AI.  

Now consider Type AII.  
We need to show that $(\mu\otimes Id + (Id\otimes \mu)(\tau\otimes Id)$ vanishes on $d_{2n-1,2n}\otimes v$ where $v$ is one of the following relations in $I$:
\begin{itemize}
\item $x_{2n-1,2n} \otimes x_{d,2n} -q^{-1}x_{d,2n}\otimes x_{2n-1,2n}$ for $d<2n-1$.
\item $x_{2n-1,2n}\otimes x_{d,2n-1} - q^{-1}x_{d,2n-1}\otimes x_{2n-1,2n}$ for $d<2n-1$.
\item $x_{2n-1,2n}\otimes x_{dk} - x_{dk}\otimes x_{2n-1,2n} - (q-q^{-1}) (qx_{d,2n}\otimes x_{k,2n-1} -x_{d,2n-1}\otimes x_{k,2n})$ for $d<k<2n-1$.
\end{itemize}
Note that both $(\mu\otimes Id)$ and $(Id\otimes \mu)(\tau\otimes Id))$ vanish on $d_{2n-1,2n}\otimes  (qx_{d,2n}\otimes x_{k,2n-1} -x_{d,2n-1}\otimes x_{k,2n})$ for $d<k<2n-1$.  Hence it is sufficient to show that $(\mu\otimes Id + (Id\otimes \mu)(\tau\otimes Id))$ vanishes on $d_{2n-1,2n}\otimes (x_{2n-1,2n} \otimes x_{dk} -q^{-\delta_{k,2n-1}-\delta_{k,2n}}x_{dk}\otimes x_{2n-1,2n}).$ We have 
\begin{align*}
(\mu\otimes Id) (d_{2n-1,2n} &\otimes (x_{2n-1,2n} \otimes x_{dk} -q^{-\delta_{k,2n-1}-\delta_{k,2n}}x_{dk}\otimes x_{2n-1,2n})) \cr& = (\mu\otimes Id) (d_{2n-1,2n} \otimes x_{2n-1,2n}\otimes x_{dk}) = x_{dk}
\end{align*}
For the other term, using (\ref{tau-identity2}) we see that 
\begin{align*}
(Id\otimes \mu)(\tau\otimes Id)(d_{2n-1,2n} &\otimes (x_{2n-1,2n} \otimes x_{dk} -q^{-\delta_{k,2n-1}-\delta_{k,2n}}x_{dk}\otimes x_{2n-1,2n}))
\cr&=-(Id\otimes \mu)(\tau\otimes Id)( q^{-\delta_{k,2n-1}-\delta_{k,2n}}x_{dk}\otimes x_{2n-1,2n}))=-x_{dk}
\end{align*}
for $d<k$ and $d<2n-1$ as desired.

The arguments with the roles of $x$ and $d$ switched are similar as well as those for $\upsilon =1$ instead of $\upsilon=0$.  The final assertion follows from writing the twisted tensor product as a quotient and applying Proposition \ref{prop:Wdefn10}.
\end{proof}

Note that Theorem \ref{thm:deformation1} gives us two quantum Weyl algebras in the nongraded case for Types AI and AII.  The same holds for the diagonal family.  Moreover, one checks as in the discussion following Theorem \ref{thm:deformation} that the same construction does not extend to the other two cases.  We write 
$\mathcal{W}_{00}^{\theta}$ for the deformation of $\mathcal{A}_{00}^{\theta}$ and $\mathcal{W}_{11}^{\theta}$ for the deformation of $\mathcal{A}_{11}^{\theta}$.  The graded equality relating the $d_{ab}$ and $x_{ef}$ as given in (\ref{relnsforxd}) becomes the following in the nongraded case:
\begin{align*}d_{ab}x_{ef} = \sum_{r,w,p,q,x,y,m,l}(S_{\alpha}^{t_2})_{xq}^{wr}(S_{\alpha}^{t_2})_{ma}^{pq}(S_{\upsilon}^{t_2})^{xy}_{fl}(S_{\upsilon}^{t_2})^{ml}_{eb}x_{pw}d_{ry} + q^{-\delta_{ef}}\delta_{ae}\delta_{bf}
\end{align*}
for all $a,b,e,f\in \{1, \dots, {\rm rank}(\mathfrak{g})\}$ where $ S_0 = R_{\mathfrak{g}}$ and $S_1 = (R_{\mathfrak{g}})_{21}^{-1}$.
The next result,  which is  
a nongraded version of Corollary \ref{cor:relnsgradedWeyl} for $\mathcal{W}_{00}^{\theta}$, gives the general shape of what these expanded relations look like.   A similar result holds for $\mathcal{W}_{11}^{\theta}$ with the 
opposite inequalities. 
\begin{corollary}\label{cor:relnsWeyl}
The following inclusions hold for the quantum  Weyl algebra $\mathcal{W}_{00}^{\theta}$
\begin{align*}
d_{ab}x_{ef} - q^{\delta_{af}+\delta_{ae}+\delta_{bf} +\delta_{be}} x_{ef}d_{ab}-q^{\delta_{ef}}\delta_{ae}\delta_{bf} \in  \sum_{(e',f',a',b')>(e,f,a,b)} \mathbb{C}(q)x_{e'f'}d_{a'b'}
\end{align*}
for all $a,b,e,f\in \{1, \dots, {\rm rank}(\mathfrak{g})\}$ where
\begin{itemize}
\item $a\leq b$ and $e\leq f$ in Type AI
\item $a<b$ and $e<f$ in Type AII
\item $a\leq n<b$ and $e\leq n<f$ in  diagonal type
\end{itemize}
and $(e',f',a',b')>(e,f,a,b)$ if and only if $e'\geq e,f'\geq f, a'\geq a, b'\geq b$ and at least one of these inequalities is strict.  
\end{corollary}

\begin{remark} \label{remark:comparisons}
Recall that in the diagonal case, the quantum Weyl algebra $\mathcal{W}_{00}^{\theta}$ is isomorphic as an algebra to $\mathcal{W}_{00}$ which in turn is the same as the quantum Weyl algebra $\mathscr{P}\mathscr{D}_q(\Mat_N)$  arising in the theory of quantum bounded symmetric domains (\cite{B}, \cite{SV}, \cite{VSS}).  It is natural to ask whether a similar isomorphism holds between the quantum Weyl algebras of this paper for Types AI and AII and the corresponding ones arising from the quantum bounded symmetric domain  theory.  For example, in Section 2 of \cite{B3}, generators and relations are given for ${\rm Pol}(\Mat_2^{sym})_q$.  One of the first issues that arises in viewing this algebra as a quantum analog of the Weyl algebra on symmetric $2\times 2$ matrices, is that the scalars that show up in the relations are different, and this difference is not just by a power of $q$.  Hence it is not clear how to normalize the generators of this algebra as was done in \cite{B} for ${\rm Pol}(\Mat_N)_q$ in order to view ${\rm Pol}(\Mat_2^{sym})_q$ as a quantum analog of the Weyl algebra. 

 There are also problems on the graded level.  Some of the relations match those of $\mathcal{W}_{00}^{\theta}$ for $n=2$ while others do not.  In particular,  it is straightforward to check that the map $x_{ij}$ to $z_{ij}$ for $i\geq j$ defines an algebra isomorphism of $\mathscr{P}_{\theta}$ onto the subalgebra of ${\rm Pol}(\Mat_2^{sym})_q$ generated by the $z_{ij}$.  Since $z_{ij} \mapsto z_{ij}^*$ is an antiautomorphism, the analogous assertion holds for $\mathscr{D}_{\theta}$.  However, these isomorphisms extend to only some of  the relations involving the $z_{ij}$ and $z_{ij}^*$.  For instance, a direct computation yields
\begin{align*} d_{21}x_{21} = q^2 x_{21}d_{21} + q^2(q-q^{-1})x_{22}d_{22}+1.
\end{align*}
However, from \cite{B3}, we see that 
\begin{align*}z^*_{21}z_{21} = q^2 z_{21}z^*_{21} + q(q-q^{-1})z_{22}z^*_{22}+(1-q^2).
\end{align*}
Note that the coefficient of $x_{22}d_{22}$ does not match $z_{22}z_{22}^*$.  It is possible that being more clever about the maps from $\mathscr{P}_{\theta}$ and $\mathscr{D}_{\theta}$ into ${\rm Pol}(\Mat_2^{sym})_q$ might yield an isomorphism with $\mathcal{W}_{00}^{\theta}$ on the graded level, but this is certainly not obvious.  
\end{remark}

\begin{remark} \label{remark:comp2} Note that the defining relations for both $\mathscr{P}_{\theta}$ and $\mathscr{D}_{\theta}$ as given in Propositions \ref{prop:Palgebra} and \ref{prop:Dalgebra} are closely related to the reflection equations. This makes these two algebras into quotients of what are called reflection equation algebras.  The relations coming from the twisting map also resemble reflection equations (see Lemma \ref{lemma:matrixform} and the discussion afterwards.)

There are other quantum Weyl algebras constructed using reflection equation type relations  in the literature and it is natural to ask whether there are any connections. For instance, in \cite{CGJ} (see  Definition 3.5)   a quantum Weyl  algebra for GL$_2$ is presented that is built using two reflection equation algebras. However, the differential and polynomial subalgebras in \cite{CGJ} are isomorphic and hence, the Weyl algebras in \cite{CGJ} differ from the ones presented here.  
There are also interesting Weyl-like algebras studied in \cite{GPS} also made up of 
two reflection equation algebras.  Once again, a comparison of the matrix relations of \cite{GPS} to the ones here reveal significant differences, so  it seems unlikely that the algebras in \cite{CGJ} and \cite{GPS} are closely related  to the Weyl algebras of this paper.
\end{remark}

\small
\section{Appendix: Commonly used notation}
We list here commonly used symbols and notation along with the first section (post the introduction) in which each item appears.

\medskip
\noindent
{\bf Section 2.1:} $\Delta$, $\epsilon$, $S$

\medskip
\noindent
{\bf Section 2.2:} $\mathcal{T}$, ${\rm Mat}_N$, $t_1,t_2$

\medskip
\noindent
{\bf Section 2.4:} $\epsilon_i$, $(\cdot, \cdot)$, $\Lambda_N$

\medskip
\noindent
{\bf Section 3.1:} $U_q(\mathfrak{gl}_N)$, $K_{\epsilon_j}^{\pm 1}$, $E_j$, $F_j$, $K_j$, $K_{\beta}$, $({\rm ad}\ E_i), ({\rm ad}\ F_i), ({\rm ad}\ K_{\epsilon_j})$, ${\natural}$, $U_h(\mathfrak{gl}_N)$,$H_{\epsilon_i}$, $L(\lambda)$

\medskip
\noindent
{\bf Section 3.2:} $\rho$, $\rho(a)^t$, $V,W,V^*,W^*$, $v_i$, $w_i$, $v_i^*$, $w_i^*$,$a^*$

\medskip
\noindent
{\bf Section 3.3:} $\mathcal{R}$, $\mathcal{R}_{21}^{-1}$, ${\rm exp}$ $E_{\beta_j}, F_{\beta_j}$, $[a,b]_q$, $[a,b]_{q^{-1}}$,$E_{i,i+1}, F_{i+1,i}$, $[r]_q!$

\medskip
\noindent
{\bf Section 3.4:} $R$, $r^{ij}_{kl}$

\medskip
\noindent
{\bf Section 4.1:} $R_{\zeta}$, $A(R_{\zeta})$, $M$, $T(M)$, $m_{ij}$

\medskip
\noindent
{\bf Section 4.2:} $R_{\rho}$, $\mathcal{O}_q({\rm Mat}_N)$, $t_{ij}$, $T_1, T_2$, $\iota$

\medskip
\noindent
{\bf Section 4.3:} $\mathcal{O}_q({\rm Mat}_N)^{op}$, $\partial_{ij}$, $P$, $P_1$, $P_2$

\medskip
\noindent
{\bf Section 5.1:} $\theta$,   $B_i$, $\mathcal{B}_{\theta}(b)$, $b_i$, ${\rm rank}(\mathfrak{g})$

\medskip
\noindent
{\bf Section 5.2:} $\mathscr{P}$, $R_{\mathfrak{g}}$, $J_{(n)}(a),$ $J_1, J_2$, $J(a)$, $J^k$, $J^k_{r,s}$, $x_{ij}(a)$, $\mathscr{D}$, $d_{ij}(c)$, $J_{r,s}$, $x_{ij}$, $d_{ij}$, $\hat{s}$

\medskip
\noindent
{\bf Section 5.3:} $\mathscr{P}_{\theta}$, $\mathscr{D}_{\theta}$, $X$, $X_1$, $X_2$, $\hat{X}$,

\medskip
\noindent
{\bf Section 6.1:} $m_C$, $m_A$, $m_B$, $\tau$

\medskip
\noindent
{\bf Section 6.2:} ${\bf u}\langle \cdot, \cdot \rangle$,  ${\bf v}\langle \cdot, \cdot \rangle$, 
 $R_{\zeta,\xi}$, ${\bf y}\langle \cdot, \cdot \rangle$

\medskip
\noindent
{\bf Section 6.3:} $R_0$, $R_1$, ${\bf u}_{\upsilon}\langle \cdot, \cdot \rangle$,  ${\bf v}_{\sigma}\langle \cdot, \cdot \rangle$, $\tau_{\upsilon,\upsilon}$, $\mathcal{A}_{\upsilon,\sigma}$, $\bar{a}$

\medskip
\noindent
{\bf Section 7.1:} $S_0, S_1$, $G$, $G_1, G_2$

\medskip
\noindent
{\bf Section 7.2:} $\mathcal{O}$, $\mathcal{O}^{op}$, $t_{ij}'$, $\partial_{ij}'$, $\tau_{\alpha,\beta,\upsilon,\sigma}$, $\mathcal{A}_{\alpha,\beta,\upsilon,\sigma}$, $\tau'_{\alpha,\beta,\upsilon,\sigma}$, $\mathcal{A}'_{\alpha,\beta,\upsilon,\sigma}$, $x_{ij}'$, $d_{ij}'$, $\mathscr{P}'_{\theta}$, $\mathscr{D}'_{\theta}$

\medskip
\noindent
{\bf Section 7.3:} $\tau^{\theta}_{\alpha,\upsilon}$, $\mathcal{A}^{\theta}_{\alpha,\upsilon}$

\medskip
\noindent
{\bf Section 8.1:} $W^{\upsilon,\sigma}_{a,b,e,f}$

\medskip
\noindent
{\bf Section 8.2:} $E_{\mu}$

\medskip
\noindent
{\bf Section 8.3:} $\mathcal{W}_{00}, \mathcal{W}_{11}$

\medskip
\noindent
{\bf Section 8.4:} $\mathcal{L}$, $W_{e,a,f,b}^{\theta,\upsilon,\sigma}$

\medskip
\noindent
{\bf Section 8.5:} $\mathcal{W}^{\theta}_{00}, \mathcal{W}^{\theta}_{11}$

\end{document}